\documentclass[11pt, letterpaper, pdflatex]{amsart}
\usepackage[foot]{amsaddr}

\usepackage{hyperref}
\hypersetup{%
colorlinks=true,
linkcolor=blue
}

\usepackage{amsmath,amsthm,amssymb}
\usepackage{mathrsfs,bm,mathtools}
\usepackage{enumitem}

\mathtoolsset{showonlyrefs=true}

\usepackage{tikz}
\usetikzlibrary{cd}

\newcommand{\jbracket}[1]{\left\langle {#1} \right\rangle}

\newcommand{\rmop}[1]{\mathop{\mathrm{#1}}}

\newcommand{\beq}{\begin{equation}}
\newcommand{\eeq}{\end{equation}}

\newcommand{\iu}{\mathrm{i}}

\newcommand{\df}{\mathrm{d}}

\newcommand{\p}{\partial}

\renewcommand{\Re}{\rmop{Re}}

\newcommand{\Ker}{\rmop{Ker}}
\newcommand{\Ran}{\rmop{Ran}}
\newcommand{\rot}{\rmop{rot}}

\newcommand{\Cbb}{\mathbb{C}}

\newcommand{\Nbb}{\mathbb{N}}

\newcommand{\Rbb}{\mathbb{R}}

\newcommand{\Ccal}{\mathcal{C}}
\newcommand{\Dcal}{\mathcal{D}}

\newcommand{\Hcal}{\mathcal{H}}

\newcommand{\Jcal}{\mathcal{J}}
\newcommand{\Kcal}{\mathcal{K}}
\newcommand{\Lcal}{\mathcal{L}}
\newcommand{\Mcal}{\mathcal{M}}

\newcommand{\Pcal}{\mathcal{P}}
\newcommand{\Qcal}{\mathcal{Q}}

\newcommand{\Scal}{\mathcal{S}}
\newcommand{\Tcal}{\mathcal{T}}

\newcommand{\Vcal}{\mathcal{V}}
\newcommand{\Wcal}{\mathcal{W}}
\newcommand{\Xcal}{\mathcal{X}}
\newcommand{\Ycal}{\mathcal{Y}}

\newcommand{\Ga}{\alpha}

\newcommand{\Gd}{\delta}

\newcommand{\Gve}{\varepsilon}

\newcommand{\Gl}{\lambda}

\newcommand{\Gm}{\mu}

\newcommand{\Gs}{\sigma}

\newcommand{\Go}{\omega}

\newcommand{\GD}{\Delta}

\newcommand{\GG}{\Gamma}

\newcommand{\GO}{\Omega}

\newcommand{\freespace}{\vspace{4mm}}

\numberwithin{equation}{section}

\newtheorem{prop}{Proposition}[section]
\newtheorem{theo}[prop]{Theorem}
\newtheorem{coro}[prop]{Corollary}
\newtheorem{lemm}[prop]{Lemma}

\theoremstyle{definition}

\newtheorem*{note*}{Note}
\newtheorem*{claim*}{Claim}
\newtheorem*{exam*}{Example}
\newtheorem*{rema*}{Remark}
\newtheorem*{exer*}{Exercise}

\newcommand{\Hsb}{\Hcal}

\newcommand{\Hdf}{\Hsb_\mathrm{div}}

\newcommand{\Hdrf}{\Hsb_\mathrm{div, rot}}

\newcommand{\divf}{\mathrm{div}}

\newcommand{\nv}{n}

\renewcommand{\iu}{i}

\renewcommand{\df}{d}

\title[Vector field decomposition and eNP spectrum]{A decomposition theorem of surface vector fields and spectral structure of the  Neumann-Poincar\'e operator in elasticity}
\thanks{This work was supported by NRF (of S. Korea) grants 2019R1A2B5B01069967 and 2022R1A2B5B01001445, and a KIAS Indivisual Grant (MG089001) at Korea Institute for Advanced Study.}

\author{Shota Fukushima$^1$}
\address{$^1$Department of Mathematics and Institute of Applied Mathematics, Inha University, Incheon 22212, S. Korea}
\email{shota.fukushima.math@gmail.com}

\author{Yong-Gwan Ji$^2$}
\address{$^2$School of Mathematics, Korea Institute for Advanced Study, Seoul 02455, S. Korea}
\email{ygji@kias.re.kr}

\author{Hyeonbae Kang$^1$}
\email{hbkang@inha.ac.kr}

\subjclass[2020]{Primary 47A10; Secondary 31A10, 31B10, 35Q74}

\begin{document}

\begin{abstract}
We prove that the space of vector fields on the boundary of a bounded domain with the Lipschitz boundary in three dimensions is decomposed into three subspaces: elements of the first one extend to the inside the domain as divergence-free and rotation-free vector fields, the second one to the outside as  divergence-free and rotation-free vector fields, and the third one to both the inside and the outside as divergence-free harmonic vector fields. We then show that each subspace in the decomposition is infinite-dimensional. We also prove under a mild regularity assumption on the boundary that the decomposition is almost direct in the sense that any intersection of two subspaces is finite-dimensional. We actually prove that the dimension of intersection is bounded by the first Betti number of the boundary. In particular, if the boundary is simply connected, then the decomposition is direct. We apply this decomposition theorem to investigate spectral properties of the Neumann-Poincar\'e operator in elasticity, whose cubic polynomial is known to be compact. We prove that each linear factor of the cubic polynomial is compact on each subspace of decomposition separately and those subspaces characterize eigenspaces of the Neumann-Poincar\'e operator. We then prove all the results for three dimensions, decomposition of surface vector fields and spectral structure, are extended to higher dimensions. We also prove analogous but different results in two dimensions.
\end{abstract}

\maketitle

\tableofcontents

\section{Introduction}

It is well known that any complex-valued function on the boundary $\p\GO$ of a bounded domain in $\Rbb^2\simeq \Cbb$ is decomposed into two parts, one of which extends to $\GO$ and the other to $\Rbb^2 \setminus \overline{\GO}$ as holomorphic functions, or equivalently by taking complex conjugates, as anti-holomorphic functions (see Section \ref{sec_2D}). It amounts to, for a vector field $(f_1,f_2)$ on $\p\GO$ after being identified with $f_1+if_2$, its being decomposed into two parts, one of which extends to $\GO$ as a divergence-free and rotation-free vector field and the other to $\Rbb^2 \setminus \overline{\GO}$. In this paper we prove an analogous, but different  decomposition theorem in three (and higher) dimensions: any vector field on the boundary $\p\GO$ of a bounded domain in $\Rbb^3$ is decomposed into three parts: the first one extends to $\GO$ as a divergence-free and rotation-free vector field, the second one to $\Rbb^3 \setminus \overline{\GO}$ as a divergence-free and rotation-free vector field, and the third one to both $\GO$ and $\Rbb^3 \setminus \overline{\GO}$ as divergence-free harmonic vector fields.

Investigation of decompositions of the surface vector fields in this paper is motivated by the spectral theory of Neumann-Poincar\'e operator in elasticity.
The Neumann-Poincar\'e operator in elasticity (abbreviated to eNP operator afterwards) is the boundary integral operator on $\p\GO \subset \Rbb^m$ whose integral kernel is defined to be the conormal derivative of the Kelvin matrix of the fundamental solution to the Lam\'e system of the linear elasticity. Unlike the NP operator for the Laplacian which is compact on $H^{1/2}(\p\GO)$ if $\p\GO$ is $C^{1,\Ga}$ for some $\Ga>0$, the eNP operator is not compact on $H^{1/2}(\p\GO)^m$ ($H^{1/2}(\p\GO)$ is the $1/2$-Sobolev space on $\p\GO$) even if $\p\GO$ is smooth (see \cite{DKV-Duke-88}).
However, it is recently proved that it is polynomially compact. In fact, if we set
\beq
    \label{eq_k0}
    k_0:=\frac{\Gm}{2(\Gl+2\Gm)}
\eeq
where $(\Gl, \Gm)$ is the pair of Lam\'e constants and denote the eNP operator by $\Kcal$, then it is proved in \cite{AJKKY} that $\Kcal^2-k_0^2I$ is compact but $\Kcal\pm k_0 I$ are not compact if $m=2$ and $\p\GO$ is $C^{1,\Ga}$, and in \cite{AKM19} that $\Kcal (\Kcal^2-k_0^2I)$ is compact but $\Kcal(\Kcal\pm k_0I)$, $\Kcal^2-k_0^2 I$ are not compact if $m=3$ and $\p\GO$ is $C^{\infty}$ (see also \cite{MR}). The assumption that $\p\GO$ is $C^{\infty}$-smooth is required in that paper when $m=3$ since the techniques of pseudo-differential operators are utilized. In fact, the compactness of $\Kcal (\Kcal^2-k_0^2I)$ is proved \cite{AKM19} by reducing the calculation of $\Kcal (\Kcal^2-k_0^2I)$ to a matrix operation of its principal symbol. We remark that the principal symbol of $\Kcal$ was already found in \cite{AAL99}, and that the diagonalization of the (matrix-valued) principal symbol of $\Kcal$ is calculated in \cite{CRSV22}. As a consequence of such results, it is shown that the spectrum of the eNP operator consists of eigenvalues converging to $k_0$ and $-k_0$ in two dimensions, and $k_0$, $-k_0$ and $0$ in three dimensions. Recently, convergence rates of eigenvalues have been studied in two dimensions \cite{AKM20}, and Weyl-type asymptotics of some combinations of eigenvalues in three dimensions have been derived \cite{Rozenblum}. In the recent paper \cite{Rozenblum23}, those of the signed eigenvalues converging to each accumulation point are obtained. These results are extensions to the eNP of the corresponding results for the NP operator associated with the Laplace operator \cite{Miya, MR19}. Eigenvalues of the eNP operator on the sphere are explicitly computed in \cite{DLL19}.

A natural question arises: Does eigenfunctions corresponding to eigenvalues, converging to $k_0$ and $-k_0$ in two dimensions, and to $k_0$, $-k_0$ and $0$ in three dimensions, have special properties? A related question is if it is possible to decompose $H^{1/2}(\p\GO)^m$ ($m=2,3$) into subspaces in some meaningful way so that $\Kcal+k_0 I$ and $\Kcal-k_0 I$ are compact on the respective subspaces in two dimensions, and $\Kcal+k_0 I$, $\Kcal-k_0 I$ and $\Kcal$ are compact on the respective subspaces in three dimensions? We prove in this paper that above mentioned decompositions yield right subspaces for these questions.
There are straight-forward decompositions of the space $H^{1/2}(\p\GO)^m$: decomposition into ranges of $\Kcal-k_0 I$ and $\Kcal+k_0 I$ in two dimensions, and ranges of $\Kcal(\Kcal-k_0 I)$, $\Kcal(\Kcal+k_0 I)$ and $\Kcal^2-k_0^2 I$ in three dimensions. But it is not clear if these subspaces have any meaningful properties, and even if it is so, it is difficult to figure them out.

We now prepare ourselves for more precise description of main results and methods.
Let $\GO\subset \Rbb^m$ be a bounded domain with the Lipschitz boundary $\p\GO$. We assume $m=2,3$ here, even if we deal with higher dimensional case as well in this paper. We set
\[
\Hsb:=H^{1/2}(\p\GO)
\]
for ease of notation. As we will explain later (see \eqref{eq_inner_product}), the space $\Hsb^m$ ($m \ge 2$) is equipped with the inner product defined in terms of the single layer potential of the Lam\'e system. The norm on $\Hsb^m$ induced by the inner product is equivalent to the product norm induced by the usual $H^{1/2}$-norm, and the eNP operator is symmetrized by this inner product. The dual spaces of $\Hsb$ and $\Hsb^m$ are denoted by $\Hsb^*=H^{-1/2}(\p\GO)$ and $(\Hsb^m)^*=H^{-1/2}(\partial\Omega)^m=(\Hsb^*)^m$, respectively.

We define subspaces of $\Hsb^m$ which are used for decomposition of surface vector fields. For $f \in \Hsb^m$, let $v_-^f \in H^1(\GO)^m$ be the unique solution to
\beq\label{eq_laplace_bvp}
    \begin{cases}
        \GD u=0 & \text{in } \GO, \\
        u=f & \text{on } \p\GO.
    \end{cases}
\eeq
We also let $v_+^f$ be the unique solution to
\beq
    \label{eq_laplace_bvp_ext}
    \begin{cases}
        \GD u=0 & \text{in } \Rbb^m\setminus \overline{\GO}, \\
        u=f & \text{on } \p\GO, \\
        u=O(|x|^{-1}) & \text{as } |x|\to \infty
    \end{cases}
\eeq
such that $\nabla v_+^f \in L^2 (\Rbb^m \setminus \overline{\GO})^m$ if $m=2$ and $v_+^f \in H^1 (\Rbb^m \setminus \overline{\GO})^m$ if $m=3$. If $m=2$, we need to impose the condition $\int_{\p\GO} f =0$ to ensure existence of the solution to \eqref{eq_laplace_bvp_ext}. We define subspaces of $\Hsb^m$ as follows:
\beq
    \label{eq_free_spaces}
    \begin{aligned}
        \Hdrf^-&:=\{ f\in \Hsb^m \mid \nabla\cdot v_-^f=0, \, \nabla\times v_-^f=0 \text{ in } \GO\}, \\
        \Hdrf^+&:=\{ f\in \Hsb^m \mid \nabla\cdot v_+^f=0, \, \nabla\times v_+^f=0 \text{ in } \Rbb^m \setminus\overline{\GO}\}, \\
        \Hdf&:=\{ f\in \Hsb^m \mid \nabla \cdot v_-^f=0 \text{ in } \Omega \text{ and } \nabla \cdot v_+^f=0 \text{ in } \Rbb^m\setminus \overline{\Omega}\}.
    \end{aligned}
\eeq
If $m=2$, $\nabla \times$ is replaced by $\mbox{rot}$. We equip the subspaces \eqref{eq_free_spaces} with the subspace topologies of $\Hsb^3$, or equivalently, the topologies induced by the restriction of the inner product \eqref{eq_inner_product} on $\Hsb^3$.

As mentioned at the beginning of Introduction, in two dimensions any vector field on $\p\GO$ is decomposed into two parts in a unique way, one extends to $\GO$ and the other to $\Rbb^2 \setminus \overline{\GO}$ as divergence-free and rotation-free vector fields. Such a decomposition can be expressed as
\[
\Hsb^2=\Hdrf^- \oplus \Hdrf^+
\]
(see Theorem \ref{coro_direct_sum_2d}). In particular, $\Hdf=\{0\}$ in two dimensions. In three dimensions, the situation is completely different. It turns out that $\Hdf$ is of infinite dimensions if $m=3$, and $\Hsb^3$ is decomposed as $\Hdrf^- + \Hdrf^+ + \Hdf$.

To investigate the decomposition when $m=3$, we introduce an inner product on $\Hsb^3$.
With the pair of Lam\'e constants $(\Gl, \Gm)$, the Lam\'e operator $\Lcal_{\Gl,\Gm}$ is defined by
\beq\label{eq_lame_operator}
    \Lcal_{\Gl, \Gm}u:=\Gm \Delta u+(\Gl+\Gm)\nabla (\nabla\cdot u).
\eeq
We assume that the Lam\'e constants $\Gl$ and $\Gm$ satisfy the following condition
\begin{equation}
    \label{eq_s_convex_condition}
    \Gm>0, \quad 3\Gl+2\Gm>0.
\end{equation}
Let $\GG = ( \GG_{ij} )_{i, j = 1}^3$ be the Kelvin matrix of fundamental solutions to the Lam\'{e} operator $\Lcal_{\Gl,\Gm}$, namely,
\beq\label{Kelvin}
  \GG_{ij}(x) =
- \frac{\Ga_1}{4 \pi} \frac{\Gd_{ij}}{|x|} - \frac{\Ga_2}{4 \pi} \frac{x_i x_j}{|x|^3},
\eeq
where
\beq
    \label{eq_alpha_12}
    \begin{aligned}
        \Ga_1
        &:=\frac{1}{2}\left( \frac{1}{\Gm}+\frac{1}{\Gl+2\Gm}\right)
        =\frac{1+2k_0}{2\Gm}, \\
    \Ga_2
    &:=\frac{1}{2}\left( \frac{1}{\Gm}-\frac{1}{\Gl+2\Gm}\right)
    =\frac{1-2k_0}{2\Gm},
    \end{aligned}
\eeq
and $\Gd_{ij}$ is Kronecker's delta (see \cite{Kup-book-65}). Here, $k_0$ is the number defined by \eqref{eq_k0}. Let $\Scal$ be the single layer potential with respect to the Lam\'e system, namely,
\beq
    \label{eq_single_layer_potential}
    \Scal [f](x):=\int_{\p \GO} \Gamma (x-y)f(y)\, \df \sigma (y).
\eeq
We consider $\Scal [f](x)$ for $x \in \p\GO$ as well as for $x \in \Rbb^3$.

It is known (see \cite{AJKKY}) that under the condition \eqref{eq_s_convex_condition} the linear operator $f\mapsto \Scal [f]|_{\p\GO}$ defines an isomorphism from $(\Hsb^3)^*$ to $\Hsb^3$. Thus the bilinear form
\beq\label{eq_inner_product}
    \jbracket{f, g}_*:=-\jbracket{\Scal^{-1}[f], g}_{\p\GO}
\eeq
defines an inner product on $\Hsb^3$ equivalent to the usual Sobolev inner product, where the right-hand side is the duality pairing between $(\Hsb^3)^*$ and $\Hsb^3$. Here $\Scal^{-1}: \Hsb^3\to (\Hsb^3)^*$ is the inverse mapping of $f\mapsto \Scal [f]|_{\p\GO}$. Note that the inner product \eqref{eq_inner_product} depends on the Lam\'e constants $(\lambda, \mu)$.
In what follows, we equip the Hilbert space $\Hsb^m$ with the inner product \eqref{eq_inner_product}.

Observe that the definition \eqref{eq_inner_product} makes sense even if $(\mu, k_0)=(1, 1/2)$, or equivalently, $(\lambda, \mu)=(-1, 1)$ which does not satisfy the condition \eqref{eq_s_convex_condition}. If $(\mu, k_0)=(1, 1/2)$, then $\Ga_1=1$ and $\Ga_2=0$, and hence $\Scal[f]=\Scal_0[f]$, where $\Scal_0$ is the single layer potential with respect to the Laplacian:
\beq\label{eq_lapla_single}
    \Scal_0[\varphi](x):=-\frac{1}{4\pi}\int_{\partial\Omega}\frac{1}{|x-y|}\varphi(y)\, \df\sigma (y)
\eeq
for $\varphi$ belonging to either $\Hsb^*$ or $(\Hsb^3)^*$. It is known (\cite{Verchota84}) that $\Scal_0$ is an isomorphism from $(\Hsb^3)^*$ to $\Hsb^3$. Thus \eqref{eq_inner_product} is well-defined.

We now state main results of this paper regarding decomposition of vector fields. We present only three dimensional results for simplicity of presentation. The two and higher dimensional results are presented in the last two sections.

\begin{theo}
    \label{theo_decomposition_div_free}
Let $\GO$ be a bounded domain in $\Rbb^3$ with the Lipschitz boundary. The Hilbert space $\Hsb^3$ is decomposed as
\beq\label{eq_H3orthogonal}
\Hsb^3=\Hdrf^-\oplus \Hdrf^+ \oplus (\Hdrf^-+\Hdrf^+)^\perp,
\eeq
where the decomposition is orthogonal with the inner product \eqref{eq_inner_product}, and
\beq\label{eq_perpsub}
(\Hdrf^-+\Hdrf^+)^\perp \subset \Hdf.
\eeq
In particular, we have
\beq\label{eq_decomposition_div_free}
        \Hsb^3=\Hdrf^- \oplus \Hdrf^+ + \Hdf.
    \eeq
\end{theo}

We emphasize that the subspaces $\Hdrf^-$, $\Hdrf^+$ and $\Hdf$ are infinite-dimensional as shown in Theorem \ref{theo_infinite_dim_c1a}. This result plays a crucial role in investigation of spectral properties of the eNP operator.

We do not know whether the equality holds in \eqref{eq_perpsub}, or equivalently,
\beq\label{eq_perpsub2}
\Hdrf^{\pm} \cap \Hdf = \{ 0 \},
\eeq
and hence the decomposition \eqref{eq_decomposition_div_free} is orthogonal  for a general domain $\GO$.
However, we show that if $\p\GO$ is $C^{1,\Ga}$ for some $\alpha>1/2$, then the inclusion \eqref{eq_perpsub} is of finite codimension, or equivalently, $\Hdrf^{\pm} \cap \Hdf$ is of finite dimensional as the following theorem shows (see also Theorem \ref{theo_finite_dim}).

\begin{theo}\label{theo_codim_betti}
    Let $\GO$ be a bounded domain in $\Rbb^3$. If $\partial\Omega$ is $C^{1, \alpha}$ for some $\alpha>1/2$, then the following inequality holds:
    \begin{equation}
        \label{eq_codim_betti}
        \dim \Hdf /(\Hdrf^-+\Hdrf^+)^\perp = \dim \Hdf \cap (\Hdrf^-+\Hdrf^+) \leq b_1(\partial\Omega).
    \end{equation}
\end{theo}

Here $b_1(X)$ denotes the first Betti number of the topological space $X$. In particular, if $\partial\Omega$ is simply connected, then $b_1(\partial\Omega)=0$ and thus the sum in \eqref{eq_decomposition_div_free} is orthogonal.

\begin{coro}\label{coro_simply_connected}
    Let $\GO$ be a bounded domain in $\Rbb^3$. If $\partial\Omega$ is $C^{1, \alpha}$ for some $\alpha>1/2$ and simply connected, then the orthogonal decomposition
    \[
        \Hsb^3=\Hdrf^- \oplus \Hdrf^+ \oplus \Hdf
    \]
    holds.
\end{coro}

We apply decomposition theorem for vector fields to investigation of eNP operator's spectral structure. The eNP operator, denoted by $\Kcal$, is defined by
\beq
    \label{eq_np}
    \Kcal [f](x):=\rmop{p.v.}\int_{\p\GO} (\p_{\nu_y}\Gamma (x-y))^T f(y)\, \df \sigma (y), \quad x \in \p\GO,
\eeq
where $\rmop{p.v.}$ stands for the Cauchy principal value ($T$ for transpose),
$\p_\nu$ denotes the conormal derivative on $\p\GO$ associated with the Lam\'e operator $\Lcal_{\Gl,\Gm}$, namely,
\beq
    \label{eq_dnu}
    \p_\nu u(x):=\Gl (\nabla\cdot u)\nv_x+2\Gm (\widehat\nabla u)\nv_x,
\eeq
where $\nv_x$ is the outward normal vector at $x\in \p\GO$, and $\widehat\nabla$ is the symmetric gradient, namely, $\widehat\nabla u= (\nabla u + \nabla u^T)/2$. The operator $\Kcal$ is a singular integral operator and bounded on $L^2(\p\GO)^3$ (\cite{DKV-Duke-88}) and on $H^1 (\partial\Omega)^3$ (\cite[Proposition 2.4]{AAL99}), and thus on $\Hsb^3=H^{1/2}(\partial\Omega)^3$ by the interpolation. The eNP operator $\Kcal$ is related to the single layer potential $\Scal$ by the jump relation
\beq
\p_\nu \Scal [f]|_{\pm} = \left( \Kcal^*\pm \frac{1}{2}I \right)[f] \quad \textrm{on } \p\GO,
\eeq
where $\Kcal^*$ is the adjoint operator of $\Kcal$ on $L^2(\p\GO)^m$. Here and throughout this paper, the subscripts $-$ and $+$ respectively denote the limits in the direction of the normal vector to $\p\GO$ from inside and outside $\GO$, for example,
\[
\p_\nu u|_{\pm}(x) := \lim_{t \to +0} \left[\Gl \nabla\cdot u(x\pm t \nv_x)\nv_x +2\Gm \widehat\nabla u(x\pm t \nv_x) \nv_x\right]
\]
for almost every $x \in \p\GO$. The operators $\Kcal^*$ and $\Kcal$ are bounded on $\Hsb^*$ and $\Hsb$, respectively (see, for example, \cite{AJKKY}).

As mentioned earlier, the eNP operator $\Kcal$ is not compact  even if $\p\GO$ is smooth, and moreover, neither $\Kcal (\Kcal\pm k_0I)$ nor $\Kcal^2-k_0^2I$ are compact, where $k_0$ is the number given in \eqref{eq_k0}. However, $\Kcal (\Kcal^2-k_0^2I)$ is compact on $\Hsb^3$ if $\p\GO$ is $C^{\infty}$. The first main result of this paper regarding polynomial compactness of the eNP operator is the following.

\begin{theo}\label{theo_free_np_3D}
Let $\GO$ be a bounded domain in $\Rbb^3$. If $\partial\Omega$ is $C^{1,\Ga}$ for some $\Ga >0$,  then the restricted operators
\begin{align*}
    \Kcal + k_0I&: \Hdrf^- \longrightarrow \Hsb^3, \\
    \Kcal - k_0I&: \Hdrf^+ \longrightarrow \Hsb^3
\end{align*}
are compact. If $\partial\Omega$ is $C^{1, \alpha}$ for some $\alpha>1/2$, then the restricted operator
\[
    \Kcal: \Hdf \longrightarrow \Hsb^3
\]
is compact.
\end{theo}

As immediate consequences of this theorem together with Theorem  \ref{theo_decomposition_div_free}, we extend the result of \cite{AKM19} to domains with $C^{1,\Ga}$ boundaries:

\begin{coro}\label{coro_polynomially_cpt_3d}
If $\GO$ is a bounded domain in $\Rbb^3$ whose boundary is $C^{1,\Ga}$ for some $\Ga >1/2$, then the operator $\Kcal(\Kcal^2-k_0^2I)$ is compact on $\Hsb^3$.
\end{coro}

To prove Theorem \ref{theo_decomposition_div_free} and \ref{theo_free_np_3D}, we introduce an eNP-type operator, denoted by $\Kcal^\divf$, related to the divergence-free solutions to the Lam\'e system (the superscript $\divf$ in $\Kcal^\divf$ stands for `divergence-free'). In fact, $\Kcal^\divf$ is defined in terms of a certain conormal derivative of the Kelvin matrix of the fundamental solution. The double layer potential defined by this conormal derivative produces divergence-free solutions to the Lam\'e system (see Subsection \ref{subsec_divfree}). It turns out that $\Kcal^\divf$ is the conjugate of the Cauchy transform in two dimensions (see Section \ref{sec_2D}).

It is proved (Theorem \ref{theo_div_rot_free_kd}) that the subspace
$\Hdrf^\pm$ is the null space of the operator $\Kcal^\divf \pm 1/2I$, namely, $\Hdrf^\pm=\Ker \left( \Kcal^\divf \pm 1/2I\right)$. This result plays a crucial role in proving Theorem \ref{theo_decomposition_div_free}.
The operator $\Kcal^\divf$ carries all the information of non-compact part of the eNP operator. In fact, $\Kcal +2k_0 \Kcal^\divf$ is compact on $\Hsb^m$ if $\p\GO$ is $C^{1,\Ga}$,
from which compactness of $\Kcal \pm k_0I$ on $\Hdrf^{\mp}$ in Theorem \ref{theo_free_np_3D} follows. To prove compactness of $\Kcal$ on $\Hdf$, we derive a general formula for $\Kcal$ expressed in terms of the single layer potential for the Laplace operator (Proposition \ref{lemm_np_modulo_cpt}). To prove Theorem \ref{theo_codim_betti}, we employ the de Rham theory and the singular homology theory for which we refer to \cite{Bott-Tu82, Bredon93, Hatcher02, Lee13}.

By the analogy to the Plemelj's symmetrization principle, which is expressed as
\beq
    \label{eq_plemelj}\Kcal \Scal=\Scal \Kcal^*,
\eeq
the operator $\Kcal$ is realized as a self-adjoint operator on $\Hsb^m$ (see \cite{AJKKY, Duduchava-Natroshvili98}). It is worthwhile mentioning that the Plemelj's symmetrization principle holds for the Laplace operator and the corresponding NP operator is self-adjoint \cite{KPS-ARMA-07}. The operator $\Kcal(\Kcal^2-k_0^2I)$ is a self-adjoint compact operator on $\Hsb^3$, and hence has real eigenvalues converging to $0$. We use infinite dimensionality of the subspaces $\Hdrf^{-}$, $\Hdrf^{+}$, and $\Hdf$ to prove that the operators $\Kcal^2-k_0^2I$, $\Kcal(\Kcal-k_0I)$, and $\Kcal(\Kcal+k_0I)$ are not compact on $\Hsb^3$. As a consequence we obtain the following theorem, which was proved in \cite{AKM19} when $\p\GO$ is $C^\infty$.

\begin{theo}\label{thm_eigenvalue3D}
If $\GO$ is a bounded domain in $\Rbb^3$ whose boundary is $C^{1,\Ga}$ for some $\Ga >1/2$, then the eigenvalues of the eNP operator $\Kcal$ on $\Hsb^3$ consist of three infinite real sequences converging to $0$, $k_0$, and $-k_0$.
\end{theo}

The following theorem shows that the eigenfunctions corresponding to eigenvalues converging to $0$, $k_0$ and $-k_0$ respectively lie in $\Hdf$, $\Hdrf^+$ and $\Hdrf^-$ in the limit. This is natural in view of Theorem \ref{theo_free_np_3D}. Throughout this paper, $\| f \|_*^2 =\jbracket{f, f}_*$.

\begin{theo}\label{theo_enp_eigenfunctions}
    Let $\GO$ be a bounded domain in $\Rbb^3$ whose boundary is $C^{1,\Ga}$ for some $\Ga >1/2$. Let $\{ f_j\}_{j=1}^\infty$ be an orthonormal system in $\Hsb^3$ consisting of eigenfunctions of $\Kcal$. Let $\Gl_j\in \Rbb$ be the corresponding eigenvalue, {\it i.e.}, $\Kcal [f_j]=\Gl_j f_j$. We decompose $f_j$ into the sum $f_j=f^-_j+f^+_j+f^o_j$ where
    \[
        f^\pm_j \in \Hdrf^\pm, \quad
        f^o_j\in (\Hdrf^++\Hdrf^-)^\perp \subset \Hdf.
    \]
    Then the following statements hold as $j \to \infty$:
    \begin{enumerate}[label={\rm (\roman*)}]
        \item \label{enum_eigen_0}If $\Gl_j\to 0$, then $\| f_j-f^o_j\|_*\to 0$.
        \item \label{enum_eigen_+}If $\Gl_j\to k_0$, then $\|f_j-f^+_j\|_*\to 0$.
        \item \label{enum_eigen_-}If $\Gl_j\to -k_0$, then $\|f_j-f^-_j\|_*\to 0$.
    \end{enumerate}
\end{theo}

Moreover, we obtain estimates on the rotation and the divergence of the solutions to the interior and exterior problems \eqref{eq_lame_bvp} and \eqref{eq_lame_bvp_ext} with the eigenfunctions $f_j$ as the boundary values (Theorem \ref{theo_div_rot_asymptotic}).

All the results in three dimensions are carried to higher dimensions without much alteration as we show in Section \ref{sec:higher}. We use the exterior calculus to define the div-free eNP operator in higher dimensions. In two dimensions, we prove that the div-free eNP operator $\Kcal^\divf$ is actually the conjugate of the Cauchy transform, as mentioned before. Using this fact, we prove that the decomposition $\Hsb^2=\Hdrf^- \oplus \Hdrf^+$ holds, that $\Kcal \pm k_0 I$ is compact on $\Hdrf^\mp$, and that a theorem analogous to Theorem \ref{theo_enp_eigenfunctions} holds.

This paper is organized as follows. In section \ref{sec:divfree}, the div-free eNP operator $\Kcal^\divf$ is defined, its important properties are derived, and Theorem \ref{theo_decomposition_div_free} is proved. In section \ref{sect_sphere}, the case when $\GO$ is a ball is dealt with. In particular, we prove that if $\GO$ is a ball, the decomposition \eqref{eq_decomposition_div_free} is actually an orthogonal decomposition, and subspaces $\Hdrf^\pm$ and $\Hdf$ are infinite-dimensional by counting dimensions of those subspaces in the space of homogeneous harmonic polynomials. We discuss the relation between the decomposition \eqref{eq_decomposition_div_free} and the eNP eigenfunctions on the sphere obtained in \cite{DLL19}. In section \ref{sec:theo12}, Theorem \ref{theo_free_np_3D} is proved. In section \ref{sec:sub}, it is proved that the subspaces $\Hdrf^\pm$ and $\Hdf$ on the Lipschitz boundaries are infinite-dimensional using the results for spheres. It is also proved that the inclusion $(\Hdrf^-+\Hdrf^+)^\perp \subset \Hdf$ in \eqref{eq_perpsub} is of finite codimension. In Section \ref{sec:topology}, we prove Theorem \ref{theo_free_np_3D} which provides an upper bound on the codimension in terms of the first Betti number. In section \ref{sec:eigenspace}, Theorem \ref{thm_eigenvalue3D} and \ref{theo_enp_eigenfunctions} regarding spectral structure of the eNP operator are proved. Section \ref{sec:higher} and \ref{sec_2D} are for the higher- and two-dimensional cases, respectively. The main part of this paper ends with short discussions, where we mention one problem arising from this work. We include in Appendix \ref{sect_smoothing} a proof of a smoothing property of weakly singular integral operators. This fact is used in Section \ref{sec:theo12} and \ref{sec:topology}. We also include in Appendix \ref{sect_dilation_elastic} a proof of the fact that by dilating the domain if necessary, the single layer potential in two-dimensions becomes invertible and the bilinear form in \eqref{eq_inner_product} is a genuine inner product.

Throughout this paper we use notation $A \lesssim B$ to imply $A \le CB$ for some constant $C$.

\section{Div-free eNP operator and proof of Theorem \ref{theo_decomposition_div_free}}\label{sec:divfree}

\subsection{Div-free eNP operator}\label{subsec_divfree}

Let $\GO$ be a bounded domain in $\Rbb^3$ with the Lipschitz boundary.
Since $\GD u = \nabla (\nabla \cdot u) - \nabla \times (\nabla \times u)$, the Lam\'e operator $\Lcal_{\Gl, \Gm}$ defined by \eqref{eq_lame_operator} can be represented as
\beq
    \label{eq_lame_vc}
    \Lcal_{\Gl, \Gm}u=-\Gm\nabla \times (\nabla \times u)+\frac{\Gm}{2k_0}\nabla (\nabla \cdot u),
\eeq
where $k_0$ is the number defined by \eqref{eq_k0}. Then, the conormal-type derivative
    \beq\label{eq_conormal_d}
        \partial^\divf_\nu u(x):=-\Gm \nv_x\times (\nabla \times u)(x)+\frac{\Gm}{2k_0}(\nabla \cdot u)(x)\nv_x
    \eeq
    defined for $x \in \p\GO$, which is different from the conormal derivative defined in \eqref{eq_dnu}, is compatible with the representation \eqref{eq_lame_vc} of the Lam\'e operator in the sense of Lemma \ref{lemm_green_vc} below. We use the notation $\partial^\divf_\nu$ to emphasize that it is a conormal-type derivative but different from $\partial_\nu$. The superscript $\divf$ is used since the double layer potential defined by this conormal-type derivative is divergence-free as we will see later (Lemma \ref{lemm_Dcalf}). In what follows, we call $\partial^\divf_\nu$ the div-free conormal derivative.

\begin{rema*}
    One can consider more general conormal-type derivative
    \[
        u(x) \longmapsto ( \mu \nabla u(x)^T+r \nabla u(x))\nv_x+(\mu+\lambda-r)(\nabla \cdot u(x))\nv_x
    \]
    (see \cite[(6.1.4)]{HMT10}, \cite{Mitrea99} for example). Here $r\in [-\mu, \mu]$ is a parameter. The the div-free conormal derivative $\partial^\divf_\nu$ corresponds to the case when $r=-\mu$, while the usual conormal derivative $\partial_\nu$ does to the case when $r=\mu$. The case $r=\mu (\mu+\lambda)/(3\mu+\lambda)$ is employed in \cite{DKV-Duke-88} (see also \cite{Kup-book-65}).
\end{rema*}

By the divergence theorem, we obtain the following lemma. Here $\jbracket{\cdot, \cdot}_{\GO}$ denotes either the inner product on $L^2(\GO)^3$ or $L^2(\GO)$ (or the duality pairing between $H^{-1}(\GO)^3$ and $H^{1}(\GO)^3$ or between $H^{-1}(\GO)$ and $H^{1}(\GO)$), and $\jbracket{\cdot, \cdot}_{\overline{\GO}^c}$ does analogously for $\overline{\GO}^c:=\Rbb^3\setminus \overline{\GO}$. Likewise, $\jbracket{\cdot, \cdot}_{\p\GO}$ denotes either the inner product on $L^2(\p\GO)^3$ or the duality pairing between $(\Hsb^3)^*$ and $\Hsb^3$.

\begin{lemm}\label{lemm_green_vc}
    Let $\Omega\subset \Rbb^3$ be a bounded Lipschitz domain.
   \begin{enumerate}[label=\rm{(\roman*)}]
        \item \label{enum_lame_bilinear_in} It holds for all $u, v\in H^1(\GO)^3$ that
        \beq
            \label{eq_lame_bilinear_1_in}
                       \jbracket{\Lcal_{\Gl, \Gm}u, v}_{\GO}
            =\jbracket{\partial^\divf_\nu u, v}_{\p\GO}
            -\Gm\jbracket{\nabla\times u, \nabla\times v}_{\GO}
            -\frac{\Gm}{2k_0}\jbracket{\nabla\cdot u, \nabla\cdot v}_{\GO}
        \eeq
        and
        \beq
            \label{eq_lame_bilinear_in}
            \begin{aligned}
                &\jbracket{\Lcal_{\Gl, \Gm}u, v}_{\GO}-\jbracket{u, \Lcal_{\Gl, \Gm}v}_{\GO} =\jbracket{\partial^\divf_\nu u, v}_{\p\GO}
            -\jbracket{u, \partial^\divf_\nu v}_{\p\GO}.
            \end{aligned}
        \eeq

        \item \label{enum_lame_bilinear_ex} If $u,v\in H^1(\overline{\GO}^c)^3$ satisfy $|u(x)|$, $|v(x)|=O(|x|^{-1/2-\delta})$, $|\nabla u(x)|$, $|\nabla v(x)|=O(|x|^{-3/2-\delta})$ and $|\Lcal_{\lambda, \mu} u(x)|$, $|\Lcal_{\lambda, \mu} v(x)|=O(|x|^{-5/2-\delta})$ as $|x|\to \infty$ for some $\delta>0$, then we have
        \beq
            \label{eq_lame_bilinear_1_ex}
            \begin{aligned}
                &\jbracket{\Lcal_{\Gl, \Gm}u, v}_{\overline{\GO}^c}
            =-\jbracket{\partial^\divf_\nu u, v}_{\p\GO}
            -\Gm\jbracket{\nabla\times u, \nabla\times v}_{\overline{\GO}^c}
            -\frac{\Gm}{2k_0}\jbracket{\nabla\cdot u, \nabla\cdot v}_{\overline{\GO}^c}
            \end{aligned}
        \eeq
        and
        \beq
            \label{eq_lame_bilinear_ex}
            \begin{aligned}
                &\jbracket{\Lcal_{\Gl, \Gm}u, v}_{\overline{\GO}^c}-\jbracket{u, \Lcal_{\Gl, \Gm}v}_{\overline{\GO}^c}
            =-\jbracket{\partial^\divf_\nu u, v}_{\p\GO}
            +\jbracket{u, \partial^\divf_\nu v}_{\p\GO}.
            \end{aligned}
        \eeq
    \end{enumerate}

\end{lemm}

\begin{proof}
\ref{enum_lame_bilinear_in} Using the representation \eqref{eq_lame_vc} of the Lam\'e operator $\Lcal_{\Gl, \Gm}$, \eqref{eq_lame_bilinear_1_in} immediately follows from the formulas
\begin{align}
    &\jbracket{\nabla \varphi, u}_{\GO}=\jbracket{\varphi, \nv\cdot u}_{\p\GO}-\jbracket{\varphi, \nabla \cdot u}_{\GO}, \label{eq_grad_adj} \\
    &\jbracket{\nabla\times u, v}_{\GO}=\jbracket{\nv\times u, v}_{\p\GO}+\jbracket{u, \nabla \times v}_{\GO} \label{eq_rot_adj}
\end{align}
for $\varphi\in H^1(\Omega)$ and $u, v\in H^1(\Omega)^3$. These formulas can be easily proved by using  divergence theorem. The identity \eqref{eq_lame_bilinear_in} can be proved by using \eqref{eq_lame_bilinear_1_in} twice.

\ref{enum_lame_bilinear_ex} can proved similarly.
\end{proof}

We define the double layer potential $\Dcal^\divf$ associated with the conormal derivative \eqref{eq_conormal_d} by
\beq
    \label{eq_dd}
    \Dcal^\divf [f](x):=\int_{\p\GO}(\p^\divf_{\nu_y}\Gamma (x-y))^T f(y)\, \df \Gs (y), \quad x\in \Rbb^3\setminus \p \GO
\eeq
for $f \in \Hsb^3$ where $\p^\divf_{\nu_y}\Gamma (x-y)$ is defined by
$$
\p^\divf_{\nu_y}\Gamma (x-y) b= \p^\divf_{\nu_y}(\Gamma (x-y) b)
$$
for any constant vector $b$. Then we have the following lemma.

\begin{lemm}\label{theo_sd_d}
    Let $\Omega\subset \Rbb^3$ be a bounded Lipschitz domain.
    \begin{enumerate}[label=\rm{(\roman*)}]
        \item \label{enum_sd_d_in}If $u\in H^1(\GO)^3$ satisfies $\Lcal_{\Gl, \Gm}u=0$ in the weak sense, then we have
        \beq
            \label{eq_sd_d_in}
            \Dcal^\divf [u|_-](x)-\Scal [\partial^\divf_\nu u|_-](x)
            =\begin{cases}
               u(x) & \text{if } x\in \GO, \\
               0 & \text{if } x\in \Rbb^3\setminus \overline{\GO}.
            \end{cases}
        \eeq
        \item \label{enum_sd_d_ex}If $u\in H^1(\Rbb^3\setminus \overline{\GO})^3$ satisfies $\Lcal_{\Gl, \Gm}u=0$, $|u(x)|=O(|x|^{-1/2-\delta})$ and $|\nabla u(x)|=O(|x|^{-3/2-\delta})$ as $|x|\to \infty$ for some $\delta>0$, then we have
        \beq
            \label{eq_sd_d_ex}
            \Dcal^\divf [u|_+](x)-\Scal [\partial^\divf_\nu u|_+](x)
            =\begin{cases}
               0 & \text{if } x\in \GO, \\
               -u(x) & \text{if } x\in \Rbb^3\setminus \overline{\GO}.
            \end{cases}
        \eeq
    \end{enumerate}
\end{lemm}

\begin{proof}
\ref{enum_lame_bilinear_in} We fix $x\in \Rbb^3\setminus \partial\GO$ and $b\in \Rbb^3$, and apply \eqref{eq_lame_bilinear_in} to $u(y)$ with $\Lcal_{\lambda, \mu}u=0$ and $v(y)=\Gamma (x-y)b$. Then we have
\[
    0=b\cdot \Scal [\partial^\divf_\nu u](x)-b\cdot \Dcal^\divf [u|_-](x)
\]
if $x\in \Rbb^3\setminus \overline{\Omega}$. Since $b\in \Rbb^3$ is arbitrary, we have \eqref{eq_sd_d_in} in the case of $x\in \Rbb^3\setminus \overline{\Omega}$.

If $x\in \Omega$, we take a sufficiently small $\varepsilon>0$ such that
\[
    B_\varepsilon(x):=\{ y\in \Rbb^3 \mid |x-y|<\varepsilon\}\subset \Omega.
\]
Then we have
\begin{equation}
    \label{eq_sd_d_in_wip}
    \begin{aligned}
    0=&b\cdot \Scal [\partial^\divf_\nu u|_-](x)-b\cdot \int_{\partial B_\varepsilon (x)} \Gamma (x-y) \partial^\divf_\nu u(y)\, \df\sigma (y) \\
    &-b\cdot \Dcal^\divf [u|_-](x)+b\cdot \int_{\partial B_\varepsilon (x)} \left(\partial^\divf_{\nu_y}\Gamma (x-y) \right)^T u(y)\, \df \sigma (y).
\end{aligned}
\end{equation}
Since $\Gamma (x-y)=O(|x-y|^{-1})$ as $|x-y|\to 0$, we have
\[
    \lim_{\varepsilon\to +0}\int_{\partial B_\varepsilon (x)} \Gamma (x-y) \partial^\divf_\nu u(y)\, \df\sigma (y)=0.
\]
On the other hand, a direct calculation using local coordinates yields
\[
    \lim_{\varepsilon\to +0}\int_{\partial B_\varepsilon (x)} \left(\partial^\divf_{\nu_y}\Gamma (x-y) \right)^T u(y)\, \df \sigma (y)=u(x).
\]
Hence \eqref{eq_sd_d_in_wip} becomes
\[
    0=b\cdot \Scal [\partial^\divf_\nu u|_-](x)-b\cdot \Dcal^\divf [u|_-](x)
    +b\cdot u(x).
\]
Since $b\in \Rbb^3$ is arbitrary, we obtain \eqref{eq_sd_d_in} in the case when $x\in \GO$.

\ref{enum_lame_bilinear_ex} is proved similarly.
\end{proof}

We recall the definition \eqref{eq_conormal_d} of $\partial^\divf_\nu$ to show
\begin{equation}\label{abc}
    \partial^\divf_\nu (\Gamma (x)b)=\frac{(x\cdot b)\nv_x-(\nv_x \cdot b)x}{4\pi |x|^3}+\frac{x\cdot \nv_x}{4\pi |x|^3}b
\end{equation}
for any constant vector $b$. We include the derivation of \eqref{abc} here for readers' sake:
\begin{align*}
    \partial^\divf_\nu (\Gamma (x)b)
    &=\partial^\divf_\nu \left(-\frac{\alpha_1}{4\pi}\frac{b}{|x|}-\frac{\alpha_2}{4\pi}\frac{(b\cdot x)x}{|x|^3}\right) \\
    &=\frac{\alpha_1}{4\pi}\mu \nv_x \times \left(\nabla\times \frac{b}{|x|}\right)+\frac{\alpha_2}{4\pi}\nv_x\times \left(\nabla \times \frac{(b\cdot x)x}{|x|^3}\right) \\
    &\quad -\frac{\alpha_1}{4\pi}\left(\nabla\cdot \frac{b}{|x|}\right)\nv_x-\frac{\alpha_2}{4\pi} \left(\nabla\cdot \frac{(b\cdot x)x}{|x|^3}\right)\nv_x \\
    &=\frac{(x\cdot b)\nv_x-(\nv_x \cdot b)x}{4\pi |x|^3}+\frac{x\cdot \nv_x}{4\pi |x|^3}b.
\end{align*}
Thus, $\Dcal^\divf$ admits the following decomposition:
\begin{equation}
    \label{eq_d_dp_decomposition}
    \Dcal^\divf =-\Dcal_1+\Dcal_2
\end{equation}
where
\begin{align}
    \Dcal_1[f](x)&=\frac{1}{4\pi}\int_{\p \GO} \frac{(x-y)(\nv_y\cdot f(y))-\nv_y ((x-y)\cdot f(y))}{|x-y|^3}\, \df \Gs (y) \label{eq_d1}
\end{align}
and
\begin{align}
    \Dcal_2[f](x)&=-\frac{1}{4\pi} \int_{\p \GO} \frac{(x-y)\cdot \nv_y}{|x-y|^3}f(y)\, \df \Gs (y) \label{eq_d2}
\end{align}
for $x\in \Rbb^3\setminus \partial\Omega$. From explicit form of $\Dcal^\divf$, we can observe that it is independent of Lam\'e constants even though both of the div-free conormal derivative $\partial^\divf_\nu$ and the fundamental solution \eqref{Kelvin} depend on them. Moreover, $\Dcal_2$ is nothing but the double layer potential with respect to the Laplacian.

We obtain the following lemma using the explicit forms of the operators $\Dcal_1$ and $\Dcal_2$.

\begin{lemm}\label{lemm_Dcalf}
Let $\Omega\subset \Rbb^3$ be a bounded Lipschitz domain. Then, for $f\in \Hsb^3$, it holds that
\beq
\GD \Dcal^\divf [f] =0, \quad \nabla \cdot \Dcal^\divf [f] =0 \quad \mbox{in } \Rbb^3\setminus \partial\Omega.
\eeq
\end{lemm}

\begin{proof}
Since $(x-y)/|x-y|^3$ is harmonic in $x$ if $x \neq y$, we see that $\GD \Dcal^\divf [f] =0$ in $\Rbb^3\setminus \partial\Omega$.

Straight-forward computations show that
\begin{align*}
&\nabla \cdot \Dcal_j[f](x) \\
&= - \frac{1}{4\pi}\int_{\p \GO} \frac{\nv_y\cdot f(y)}{|x-y|^3}\,\df \Gs (y) + \frac{3}{4\pi}\int_{\p \GO} \frac{((x-y) \cdot \nv_y)( (x-y)\cdot f(y))}{|x-y|^5}\, \df \Gs (y)
\end{align*}
for $j=1,2$, and hence $\nabla \cdot \Dcal^\divf [f] =0$ in $\Rbb^3\setminus \partial\Omega$.
\end{proof}

We define the corresponding eNP-type operator $\Kcal^\divf$ by
\beq
    \label{eq_kd}
    \Kcal^\divf [f](x):=\rmop{p.v.}\int_{\p\GO}(\p^\divf_{\nu_y}\Gamma (x-y))^T f(y)\, \df \Gs (y).
\eeq
Because $\nabla \cdot \Dcal^\divf [f] =0$, we respectively call $\Dcal^\divf$ and $\Kcal^\divf$ the div-free double layer potential and the div-free eNP operator throughout this paper.

The div-free eNP operator $\Kcal^\divf$ admits the following decomposition accordingly:
\begin{equation}
    \label{eq_kd_decomposition}
        \Kcal^\divf=-\Kcal_1+\Kcal_2,
\end{equation}
where
\begin{equation}
    \label{eq_k1}
    \Kcal_1[f](x)=\frac{1}{4\pi}\,\mathrm{p.v.}\int_{\p \GO} \frac{(x-y)(\nv_y\cdot f(y))-\nv_y ((x-y)\cdot f(y))}{|x-y|^3} \df \Gs (y)
\end{equation}
and
\begin{equation}
    \label{eq_k2}
    \Kcal_2[f](x)=-\frac{1}{4\pi} \,\mathrm{p.v.}\int_{\p \GO} \frac{(x-y)\cdot \nv_y}{|x-y|^3}f(y) \df \Gs (y)
\end{equation}
for $x\in \partial\Omega$. On the other hand, the eNP operator $\Kcal$ defined in \eqref{eq_np} is decomposed as
\beq
    \label{eq_np_decomposition}
    \Kcal=2k_0 (\Kcal_1+\Kcal_2)-3(1-2k_0)\Kcal_3,
\eeq
where the operator $\Kcal_3$ is defined as
\beq
    \label{eq_k3}
    \Kcal_3[f](x):=\frac{1}{4\pi}\mathrm{p.v.}\int_{\p\GO} \frac{((x-y)\cdot f(y))((x-y)\cdot \nv_y)(x-y)}{|x-y|^5}\, \df \Gs (y)
\eeq
(see also \cite{AKM19}). Like the div-free double layer potential, the div-free eNP operator $\Kcal^\divf$ is independent of the Lam\'e constants and the operator $\Kcal_2$ is nothing but the NP operator with respect to the Laplacian.

We also observe from \eqref{eq_d1} that the operator $\Dcal_1$ can be expressed in terms of the single layer potential $\Scal_0$ with respect to the Laplacian defined by \eqref{eq_lapla_single}: if $f \in \Hsb^3$, then
\beq\label{eq_double_single}
\Dcal_1[f](x) = \nabla \Scal_0[f \cdot \nv](x) - (\nabla \cdot \Scal_0[\nv_1 f](x), \nabla \cdot \Scal_0[\nv_2 f](x), \nabla \cdot \Scal_0[\nv_3 f](x)).
\eeq
Here, $\nv=(\nv_1,\nv_2,\nv_3)$.

It is known (see \cite{AAL99, EFV92, Verchota84}) that the following jump formula holds:
\begin{equation}
    \label{eq_jump_grad}
    \nabla \Scal_0[\varphi]|_\pm (x)=\pm \frac{1}{2} \varphi(x) \nv_x +\frac{1}{4\pi}\rmop{p.v.}\int_{\partial\Omega} \frac{x-y}{|x-y|^3}\varphi(y)\, \df\sigma (y)
\end{equation}
for $\varphi \in \Hsb$. Using this relation, we prove the following jump formula for the div-free double layer potential.
\begin{prop}
    \label{prop_dd_jump}
    Let $\Omega\subset \Rbb^3$ be a bounded Lipschitz domain. Then it holds that
    \beq
        \label{eq_dd_jump}
        \Dcal^\divf [f]|_\pm (x)=\left(\Kcal^\divf\mp \frac{1}{2}I\right)[f](x), \quad\textrm{a.e. } x \in \p \GO
    \eeq
    for all $f\in \Hsb^3$.
\end{prop}

\begin{proof}
It follows from \eqref{eq_jump_grad} that
\[
\nabla \cdot \Scal_0[\nv_j f]|_\pm (x) = \pm \frac{1}{2} (f \cdot \nv)(x) \nv_j +\frac{1}{4\pi}\rmop{p.v.}\int_{\partial\Omega} \frac{(x-y) \cdot f(y) }{|x-y|^3} \nv_j(y)\, \df\sigma (y)
\]
for $j=1,2,3$.
Thus one can see from \eqref{eq_double_single} that
\begin{align*}
    \Dcal_1[f]|_\pm (x)=\pm \frac{1}{2}((\nv_x\cdot f(x))\nv_x - (\nv_x\cdot f(x))\nv_x)+\Kcal_1[f](x)=\Kcal_1[f](x),
\end{align*}
where $\Kcal_1$ is defined in \eqref{eq_k1}. Since $\Dcal_2$ is the double layer potential with respect to the Laplacian, it holds that
\begin{align} \label{double_jump_laplacian}
    \Dcal_2[f]|_\pm (x)=\Kcal_2[f](x)\mp \frac{1}{2}f(x),
\end{align}
where $\Kcal_2$ is defined in \eqref{eq_k2}.
Hence, from \eqref{eq_d_dp_decomposition} and \eqref{eq_kd_decomposition}, we have
\[
    \Dcal^\divf [f]|_\pm=-\Dcal_1[f]|_\pm+\Dcal_2[f]|_\pm
    =-\Kcal_1[f]+\Kcal_2[f]\mp \frac{1}{2}f
    =\Kcal^\divf [f]\mp \frac{1}{2}f
\]
as desired.
\end{proof}

We introduce the subspaces
\begin{equation}
    \label{eq_div_free_pm}
    \begin{aligned}
        \Hdf^-&:=\{ f\in \Hsb^m \mid \nabla\cdot v_-^f=0 \text{ in } \GO\}, \\
        \Hdf^+&:=\{ f\in \Hsb^m \mid \nabla\cdot v_+^f=0 \text{ in } \Rbb^m \setminus\overline{\GO}\},
    \end{aligned}
\end{equation}
where $v_-^f$ and $v_+^f$ are the unique solutions to the boundary value problems \eqref{eq_laplace_bvp} and \eqref{eq_laplace_bvp_ext} respectively. Clearly, it holds that
\beq
    \label{eq_free_spaces_shorthand}
    \Hdf=\Hdf^-\cap \Hdf^+.
\eeq

Lemma \ref{lemm_Dcalf} and Proposition \ref{prop_dd_jump}  together with the decomposition
\[
f = \left(\Kcal^\divf+ \frac{1}{2}I\right)[f] - \left(\Kcal^\divf- \frac{1}{2}I\right)[f]
\]
yields the following theorem.

\begin{theo}
For a bounded Lipschitz domain $\Omega\subset \Rbb^3$, it holds that
\beq\label{eq_Hdfpm}
\Hsb^3 = \Hdf^- + \Hdf^+.
\eeq
\end{theo}

We emphasize that the sum in \eqref{eq_Hdfpm} is not direct because $\Hdf=\Hdf^- \cap \Hdf^+$ is infinite-dimensional as shown in Theorem \ref{theo_infinite_dim_c1a}. The decomposition \eqref{eq_Hdfpm} holds in two dimensions as well. In that case the sum is direct (see Theorem \ref{coro_direct_sum_2d}).

We also investigate fundamental properties of the $L^2$-adjoint $(\Kcal^\divf)^*$ of $\Kcal^\divf$. The operator $(\Kcal^\divf)^*$ is represented as
\[
    (\Kcal^\divf)^*[f](x)=\rmop{p.v.}\int_{\p\GO} \p^\divf_{\nu_x}\Gamma (x-y) f(y)\, \df \Gs (y).
\]
Then $(\Kcal^\divf)^*$ enjoys the following jump relation analogous to the eNP operator.

\begin{prop}
    \label{theo_jump_npds}
    Let $\Omega\subset \Rbb^3$ be a bounded Lipschitz domain. Then it holds that
    \beq
        \label{eq_jump_npds}
        \partial^\divf_\nu \Scal [f]|_\pm =\left( (\Kcal^\divf)^*\pm \frac{1}{2}I\right)[f]
    \eeq
    for $f\in (\Hsb^3)^*$.
\end{prop}

\begin{proof}
The jump relation \eqref{eq_jump_npds} follows from \eqref{abc} and \eqref{eq_jump_grad}. In fact, we have
\begin{align*}
    &\partial^\divf_\nu \Scal [f]|_\pm (x) \\
    &=\pm \frac{1}{2} \big( (\nv_x\cdot f(x))\nv_x-(\nv_x\cdot f(x))\nv_x+|\nv_x|^2 f(x) \big)+(\Kcal^\divf)^*[f](x) \\
    &=(\Kcal^\divf)^* [f](x)\pm \frac{1}{2}f(x),
    \end{align*}
    as desired.
\end{proof}

We now prove that the div-free eNP operator $\Kcal^\divf$ is symmetrized by the single layer potential.

\begin{prop}
    \label{theo_plemelj_kd}
    Let $\Omega\subset \Rbb^3$ be a bounded Lipschitz domain. Then the operator $\Kcal^\divf$ satisfies the relation
    \begin{equation}
        \label{eq_plemelj_kd}
        \Scal(\Kcal^\divf)^*=\Kcal^\divf\Scal.
    \end{equation}
\end{prop}

\begin{proof}
    We substitute $u=\Scal [f]$ to \eqref{eq_sd_d_in} to obtain
    \[
        \Dcal^\divf [\Scal [f]](x)=\Scal[\partial^\divf_\nu \Scal [f]|_-](x)
    \]
    for $x\in \Rbb^3 \setminus \overline{\GO}$. We then take a limit to $\p\GO$ from outside to obtain
    \[
        \Dcal^\divf [\Scal [f]]|_+=\Scal[\partial^\divf_\nu \Scal[f]|_-].
    \]
    By the jump relations \eqref{eq_dd_jump} and \eqref{eq_jump_npds}, we have
    \[
        \left(\Kcal^\divf- \frac{1}{2}I\right) \Scal [f]=\Scal \left( (\Kcal^\divf)^*-\frac{1}{2}I\right)[f],
    \]
    which yields \eqref{eq_plemelj_kd}.
\end{proof}

The symmetrization principle (Proposition \ref{theo_plemelj_kd}) implies that the div-free eNP  operator $\Kcal^\divf$ is also self-adjoint with respect to the inner product \eqref{eq_inner_product}.

\begin{coro}
    \label{coro_kd_self_adjoint}
    Let $\Omega \subset \Rbb^3$ be a bounded Lipschitz domain. Then $\Kcal^\divf: \Hsb^3 \to \Hsb^3$ is a bounded self-adjoint operator with respect to the inner product \eqref{eq_inner_product}.
\end{coro}

\begin{proof}
    We only have to show that $\Kcal^\divf$ is bounded on $\Hcal^3$. In addition to the Lam\'e constants $(\lambda, \mu)$, we take another Lam\'e constants $(\lambda^\prime, \mu^\prime)$ satisfying \eqref{eq_s_convex_condition} and let the constant $k_0^\prime>0$ and the eNP operator $\Kcal^\prime$ be analogously defined by \eqref{eq_k0} and \eqref{eq_np} from $(\lambda^\prime, \mu^\prime)$, so that
    \[
        \det
        \begin{pmatrix}
            2k_0 & -3(1-2k_0) \\
            2k_0^\prime & -3(1-2k_0^\prime)
        \end{pmatrix}
        =-6(k_0-k_0^\prime)\neq 0.
    \]
    Then, by \eqref{eq_np_decomposition}, the operators $\Kcal_1+\Kcal_2$ and $\Kcal_3$ are represented as a linear combination of the eNP operators $\Kcal$ and $\Kcal^\prime$, which are bounded on $\Hsb^3$ as we already mentioned in Introduction. Here, note that, although the inner product \eqref{eq_inner_product} on $\Hsb^3$ depends on the Lam\'e constants, the associated norm is equivalent to the usual Sobolev norm regardless of the Lam\'e constants. Thus, the operators $\Kcal_1+\Kcal_2$ and $\Kcal_3$ are bounded on $\Hcal^3$. Since $\Kcal_2$ is the NP operator for the Laplacian, it is bounded on $\Hsb^3$ (essentially due to Verchota \cite{Verchota84}). Hence, all operators $\Kcal_j$ ($j=1, 2, 3$) are bounded on $\Hsb^3$ and so is $\Kcal^\divf$ by \eqref{eq_kd_decomposition}.
\end{proof}

\subsection{Proof of Theorem \ref{theo_decomposition_div_free}}

In this subsection we prove Theorem \ref{theo_decomposition_div_free}. In doing so, the following theorem is crucial.

\begin{theo}
    \label{theo_div_rot_free_kd}
    If $\Omega\subset \Rbb^3$ is a bounded Lipschitz domain, then we have
    \beq
        \label{eq_div_rot_free_kd}
        \Hdrf^\pm=\Ker \left( \Kcal^\divf\pm \frac{1}{2}I \right).
    \eeq
\end{theo}

For the proof of Theorem \ref{theo_div_rot_free_kd}, we prepare two lemmas. The first lemma is the representation of the divergence-free and rotation-free solutions by the double layer potential $\Dcal^\divf$. The second lemma provides two identities which relate the div-free eNP operator $\Kcal^\divf$ to the $L^2$ norm of the divergence and rotation of the solution to the Lam\'e equation.

In what follows, we define
\beq\label{eq_uf}
u^f(x):=\Scal[\Scal^{-1}[f]](x), \quad x\in \Rbb^3
\eeq
for $f\in \Hsb^3$. Note that $u^f|_{\overline{\Omega}}$ and $u^f|_{\Rbb^3\setminus \Omega}$ are respectively the unique solutions to the interior and exterior boundary value problems
\beq
    \label{eq_lame_bvp}
    \begin{cases}
        \Lcal_{\Gl, \Gm}u=0 & \text{in } \GO, \\
        u=f & \text{on } \p\GO
    \end{cases}
\eeq
and
\beq
    \label{eq_lame_bvp_ext}
    \begin{cases}
        \Lcal_{\Gl, \Gm}u=0 & \text{in } \Rbb^3 \setminus \overline{\GO}, \\
        u=f & \text{on } \p\GO, \\
        u(x)=O(|x|^{-1}) & \text{as } |x|\to \infty
    \end{cases}.
\eeq
See \cite[Corollary 4.2]{Howell80} for the uniqueness of the solution to \eqref{eq_lame_bvp_ext}.
We emphasize that if $f \in \Hdf^-$, then the solution $v_-^f$ to \eqref{eq_laplace_bvp} is the solution to \eqref{eq_lame_bvp} for any pair $(\Gl, \Gm)$ of Lam\'e constants. Likewise, if $f \in \Hdf^+$, then $v_+^f$ is the solution to \eqref{eq_lame_bvp_ext}.

\begin{lemm}
    \label{lemm_div_rot_dd}
    Let $\Omega\subset\Rbb^3$ be a bounded Lipschitz domain. If $f\in \Hdrf^-$, then we have
    \beq
        \label{eq_div_rot_dd_in}
        u^f=\Dcal^\divf [f] \quad \text{in } \GO,
    \eeq
    and if $f\in \Hdrf^+$, then
    \beq
        \label{eq_div_rot_dd_ex}
        u^f=-\Dcal^\divf [f] \quad \text{in } \Rbb^3\setminus\overline{\GO}.
    \eeq
\end{lemm}

\begin{proof}
    If $f\in \Hdrf^-$, then we have $\p_\nu^\divf u^f|_-=0$ by the definition \eqref{eq_conormal_d} of the conormal derivative $\partial^\divf_\nu$. We apply \eqref{eq_sd_d_in} for $u^f$ and obtain $\Dcal^\divf [f](x)=u^f(x)$ for $x\in \GO$.

    Similarly, if $f\in \Hdrf^+$, then we have $\p_\nu^\divf u^f|_+=0$. Noting that $u^f=\Scal[\Scal^{-1}[f]]$ satisfies $|u^f(x)|=O(|x|^{-1})$ and $|\nabla u^f(x)|=O(|x|^{-2})$ as $|x|\to \infty$, we apply \eqref{eq_sd_d_ex} for $u^f$ and obtain $\Dcal^\divf [f](x)=-u^f(x)$ for $x\in \Rbb^3\setminus\overline{\GO}$.
\end{proof}

\begin{lemm}
    \label{lemm_div_rot_npd_identity}
    Let $\Omega\subset \Rbb^3$ be a bounded Lipschitz domain. Then, for $f, g \in \Hsb^3$, we have
    \beq
        \label{eq_div_rot_npd_identity_in}
        \jbracket{\left(-\Kcal^\divf+\frac{1}{2}I\right)[f], g }_*
        =\Gm \jbracket{ \nabla\times u^f, \nabla\times u^g }_{\GO}+\frac{\Gm}{2k_0}
        \jbracket{ \nabla \cdot u^f, \nabla \cdot u^g }_{\GO}
    \eeq
    and
    \beq
        \label{eq_div_rot_npd_identity_ex}
        \jbracket{\left( \Kcal^\divf+ \frac{1}{2}I\right)[f], g }_*
        =\Gm \jbracket{ \nabla\times u^f, \nabla\times u^g }_{\overline{\GO}^c} +\frac{\Gm}{2k_0}\jbracket{ \nabla \cdot u^f, \nabla \cdot u^g }_{\overline{\GO}^c}.
    \eeq
\end{lemm}

\begin{proof}
    We apply \eqref{eq_lame_bilinear_1_in} to $u=u^f=\Scal[\Scal^{-1}[f]]$ and $v=u^g =\Scal[\Scal^{-1}[g]]$ to obtain
    \[
        \jbracket{\partial^\divf_\nu \Scal [\Scal^{-1}[f]]|_-, g}_{\p\GO}
        =\Gm \jbracket{ \nabla\times u^f, \nabla\times u^g }_{\GO}+\frac{\Gm}{2k_0}
        \jbracket{ \nabla \cdot u^f, \nabla \cdot u^g }_{\GO}.
    \]
    By the jump relation (Theorem \ref{theo_jump_npds}) and the symmetrization principle (Theorem \ref{theo_plemelj_kd}), we have
    \begin{align*}
        \jbracket{\left((\Kcal^\divf)^*-\frac{1}{2}I\right)\Scal^{-1}[f], g}_{\p\GO}
        &=\jbracket{\Scal^{-1}\left(\Kcal^\divf- \frac{1}{2}I\right)[f], g}_{\p\GO} \\
        &=\jbracket{\left(-\Kcal^\divf+ \frac{1}{2}I\right)[f], g}_*.
    \end{align*}
    Hence we obtain \eqref{eq_div_rot_npd_identity_in}.

    The identity \eqref{eq_div_rot_npd_identity_ex} is proved similarly.
\end{proof}

By adding formulas \eqref{eq_div_rot_npd_identity_in} and \eqref{eq_div_rot_npd_identity_ex}, we obtain the following corollary.

\begin{coro}\label{cor213}
If $\Omega\subset \Rbb^3$ is a bounded Lipschitz domain, then, for $f,g\in \Hsb^3$, we have
\beq\label{eq_2.23+2.24}
\begin{aligned}
    \jbracket{f,g}_*=
    &\mu \jbracket{ \nabla\times u^f, \nabla\times u^g }_{\GO} +\mu \jbracket{ \nabla\times u^f, \nabla\times u^g }_{\overline{\GO}^c} \\
    &+\frac{\mu}{2k_0}\jbracket{ \nabla \cdot u^f, \nabla \cdot u^g }_{\GO} +\frac{\mu}{2k_0} \jbracket{ \nabla \cdot u^f, \nabla \cdot u^g }_{\overline{\GO}^c}.
\end{aligned}
\eeq
\end{coro}

Moreover, we obtain a spectral radius estimate of $\Kcal^\divf$.
\begin{coro}\label{coro_spec_rad_kd}
    For a bounded Lipschitz domain $\Omega\subset \Rbb^3$, we have $\sigma (\Kcal^\divf)\subset [-1/2, 1/2]$.
\end{coro}

\begin{proof}
    By \eqref{eq_div_rot_npd_identity_in}, we have
    \[
        \jbracket{\left(-\Kcal^\divf+\frac{1}{2}I\right)[f], f }_*
        =\Gm \| \nabla\times u^f \|_{\GO}^2+\frac{\Gm}{2k_0}
        \| \nabla \cdot u^f\|_{\GO}^2\geq 0
    \]
    for all $f\in \Hsb^3$. Thus we have $\Kcal^\divf\leq 1/2I$. Similarly, by \eqref{eq_div_rot_npd_identity_ex}, we obtain $\Kcal^\divf\geq -1/2I$. Hence we have $\sigma (\Kcal^\divf)\subset [-1/2, 1/2]$.
\end{proof}

In fact, it turns out that the spaces $\Ker (\Kcal^\divf\pm 1/2I)$ are infinite dimensional  if $\p\GO$ is Lipschitz (see Theorem \ref{theo_div_rot_free_kd} and Theorem \ref{theo_infinite_dim_c1a}). Thus the spectral radius of $\Kcal^\divf$ is exactly $1/2$.

We will employ Corollary \ref{cor213} and \ref{coro_spec_rad_kd} in later sections, although they are not used in this section.

\begin{lemm}
    \label{lemm_div_free_closed}
    If $\Omega\subset \Rbb^3$ is a bounded Lipschitz domain, then the subspaces $\Hdf^\pm$ are closed in $\Hsb^3$.
\end{lemm}

\begin{proof}
    Suppose that the sequence $\{ f_j\}_{j=1}^\infty$ in $\Hdf^-$ converges to some function $f\in \Hsb^3$. For an arbitrary integer $j$, we have
    \begin{align*}
        \jbracket{\left(-\Kcal^\divf+\frac{1}{2}I\right)[f_j-f], f_j-f}_*
        &\geq \frac{\Gm}{2k_0}\| \nabla\cdot u^{f_j}-\nabla\cdot u^f \|^2_{\GO} \\
        &=\frac{\Gm}{2k_0}\| \nabla\cdot u^f\|^2_{\GO}
    \end{align*}
    by \eqref{eq_div_rot_npd_identity_in}. We take limit $j\to \infty$ and obtain $\nabla\cdot u^f=0$ in $\GO$. Hence, $f\in \Hdf^-$.

    It is proved similarly that $\Hdf^+$ is a closed subspace of $\Hsb^3$.
\end{proof}

\begin{proof}[Proof of Theorem \ref{theo_div_rot_free_kd}]
    If $f\in \Hdrf^-$, then we infer from the jump relation \eqref{eq_dd_jump} and the identity \eqref{eq_div_rot_dd_in} that
    \[
        f=\left(\Kcal^\divf+ \frac{1}{2}I\right)[f].
    \]
    Thus, $f\in \Ker (\Kcal^\divf- 1/2I)$. Hence, we have $\Hdrf^-\subset \Ker (\Kcal^\divf- 1/2I)$. Similarly, we can prove $\Hdrf^+\subset \Ker (\Kcal^\divf+ 1/2I)$ by \eqref{eq_dd_jump} and \eqref{eq_div_rot_dd_ex}.

    The reverse inclusions $\Ker (\Kcal^\divf\pm 1/2I) \subset \Hdrf^\pm$ are immediate consequences of \eqref{eq_div_rot_npd_identity_in} and \eqref{eq_div_rot_npd_identity_ex}.
\end{proof}

We are ready to prove Theorem \ref{theo_decomposition_div_free}.

\begin{proof}[Proof of Theorem \ref{theo_decomposition_div_free}]
    The subspaces $\Hdrf^-$ and $\Hdrf^+$ are orthogonal to each other since they are eigenspaces corresponding to distinct eigenvalues (Theorem \ref{theo_div_rot_free_kd}) of the self-adjoint operator $\Kcal^\divf$ (Corollary \ref{coro_kd_self_adjoint}).

    To prove \eqref{eq_perpsub}, we first note that
\beq\label{defg}
\Ker \left(\Kcal^\divf+ \frac{1}{2}I\right) + \Ker \left(\Kcal^\divf- \frac{1}{2}I\right)
= \Ker \left(\Kcal^\divf+ \frac{1}{2}I\right)\left(\Kcal^\divf- \frac{1}{2}I\right).
\eeq
In fact, the inclusion relation $\subset$ is easy to prove. The opposite inclusion relation can be shown using the identity
\[
f= -\left(\Kcal^\divf- \frac{1}{2}I\right)[f] + \left(\Kcal^\divf+ \frac{1}{2}I\right)[f].
\]

Since $\Kcal^\divf$ is self-adjoint, we infer from Theorem \ref{theo_div_rot_free_kd} and \eqref{defg} that
    \begin{equation}\label{eq_perp_closure}
        (\Hdrf^-+\Hdrf^+)^\perp =\overline{\Ran\left(\Kcal^\divf- \frac{1}{2}I\right)\left(\Kcal^\divf+ \frac{1}{2}I\right)}.
    \end{equation}
Since $\Hdf=\Hdf^+ \cap \Hdf^-$ is a closed subspace by Lemma \ref{lemm_div_free_closed}, it is enough to prove
    \beq
        \label{eq_div_free_inclusion_weak}
        \Hdf\supset \Ran\left(\Kcal^\divf- \frac{1}{2}I\right)\left(\Kcal^\divf+ \frac{1}{2}I\right).
    \eeq
    If $f \in \Ran\left(\Kcal^\divf- 1/2I\right)\left(\Kcal^\divf+ 1/2I\right)$, then
    \[
        f=\left(\Kcal^\divf+ \frac{1}{2}I\right)[f_+]=\left(\Kcal^\divf- \frac{1}{2}I\right)[f_-]
    \]
    for some $f_+, f_- \in \Hsb^3$. Then we have
    \[
        f=\Dcal^\divf [f_+]|_-=\Dcal^\divf [f_-]|_+
    \]
    by the jump relation \eqref{eq_dd_jump}. Since $\Dcal^\divf [f_+]|_- \in \Hdf^-$ and $\Dcal^\divf [f_-]|_+\in \Hdf^+$ by Lemma \ref{lemm_Dcalf}, we have $f\in \Hdf$. This proves \eqref{eq_perpsub}.
\end{proof}

Before  concluding this section we prove two results to be used in later sections. One shows that the subspace $\Hdf$ is invariant under the div-free eNP operator $\Kcal^\divf$; the other shows that $(\Hdrf^\pm)^\perp\subset \Hdf^\mp$.

\begin{theo}\label{theo_div_free_invariant}
    Let $\Omega\subset \Rbb^3$ be a bounded Lipschitz domain. Then the operator $\Kcal^\divf$ preserves the subspace $\Hdf$.
\end{theo}

\begin{proof}
    Let $f\in \Hdf$. Then, by the jump relation \eqref{eq_dd_jump}, we have
    \[
        \Kcal^\divf [f]=\Dcal^\divf [f]|_\pm\pm\frac{1}{2}f\in \Hdf^\pm.
    \]
    Hence, $\Kcal^\divf [f]\in \Hdf$.
\end{proof}

The following lemma will be used in Section \ref{sec:topology}.

\begin{lemm}\label{lemm_drf_df_pm}
    For a bounded Lipschitz domain $\Omega\subset \Rbb^3$, we have $(\Hdrf^\pm)^\perp\subset \Hdf^\mp$ with respect to the inner product \eqref{eq_inner_product}.
\end{lemm}

\begin{proof}
    We only prove the inclusion relation $(\Hdrf^+)^\perp\subset \Hdf^-$; the other inclusion relation can be proved similarly.

    Since $\Kcal^\divf$ is self-adjoint on $\Hsb^3$ by Corollary \ref{coro_kd_self_adjoint}, we have the orthogonal decomposition
    \begin{equation}
        \label{eq_orthogonal_pm}
        \Hsb^3=\Ker \left(\Kcal^\divf + \frac{1}{2}I\right) \oplus \overline{\Ran \left(\Kcal^\divf + \frac{1}{2}I\right)}.
    \end{equation}
    Since $\Ker (\Kcal^\divf + 1/2I)=\Hdrf^+$ by Theorem \ref{theo_div_rot_free_kd}, we have
    \[
        \overline{\Ran \left(\Kcal^\divf + \frac{1}{2}I\right)}=(\Hdrf^+)^\perp.
    \]
    Since $\Hdf^-$ is a closed subspace of $\Hsb^3$ by Lemma \ref{lemm_div_free_closed}, it is enough to show
    \begin{equation}
        \label{eq_ran_kd_pm_included}
        \Ran \left(\Kcal^\divf + \frac{1}{2}I\right)\subset \Hdf^-.
    \end{equation}
    If $f=(\Kcal^\divf + 1/2I)[g]\in \Ran (\Kcal^\divf + 1/2I)$, then we have $f=\Dcal^\divf [g]|_-$ by Proposition \ref{prop_dd_jump}. Since $\Dcal^\divf [g]$ is a divergence-free harmonic vector field in $\Omega$ by Lemma \ref{lemm_Dcalf}, we have $f \in \Hdrf^-$. This proves \eqref{eq_ran_kd_pm_included}.
\end{proof}

\section{The spherical case}\label{sect_sphere}

In this section we consider the case when $\GO$ is a ball. The main results of this section are that if $\GO$ is a ball, then the equality holds in \eqref{eq_perpsub}, or equivalently, the decomposition \eqref{eq_decomposition_div_free} is orthogonal, and that three subspaces involved in the decomposition are infinite-dimensional, which is proved by considering homogeneous spherical harmonics and computing the dimensions of subspaces of homogeneous degree $k$. The result of this section will be used in an essential way to prove that all three subspaces involved in the decomposition are infinite-dimensional even if $\partial \Omega$ is Lipschitz (Theorem \ref{theo_infinite_dim_c1a}).

\subsection{Orthogonal decomposition}

\begin{theo}\label{theo_sphere_orthogonal}
    If $\GO$ is a ball in $\Rbb^3$, then we have the orthogonal decomposition
    \begin{equation}\label{eq_sphere_orthogonal}
        \Hsb^3=\Hdrf^- \oplus \Hdrf^+ \oplus \Hdf
    \end{equation}
    with respect to the inner product \eqref{eq_inner_product}.
\end{theo}

To prove Theorem \ref{theo_sphere_orthogonal},
we assume, without loss of generality, that $\GO$ is the unit ball:
\[
    \GO:=\{ x\in \Rbb^3 \mid |x|=1\}.
\]
For a nonnegative integer $k$, we denote by $P_k$ the space of all homogeneous harmonic polynomial. We also denote the space $\{ p|_{\p\GO} \mid p\in P_k\}$
by the same character $P_k$ by abusing notation (without causing confusion). Since $\Hsb= \overline{\oplus_{k=0}^\infty P_k}$, we see that
\[
    \Hsb^3=\overline{\bigoplus_{k=0}^\infty P_k^3}
\]
where the closure is taken in the norm topology on $\Hsb^3$. For $u: \Rbb^3 \setminus \{0\}\to \Cbb$, we denote by $u^*$ the Kelvin transform of $u$, namely,
\begin{equation}
    \label{eq_inversion}
    u^*(x):=\frac{1}{|x|}u\left(\frac{x}{|x|^2}\right).
\end{equation}
If $u$ is harmonic, then $u^*$ is also harmonic, and vice versa. If $u$ is vector-valued, i.e., if $u=(u_1,u_2,u_3)$, then $u^*:=(u_1^*,u_2^*,u_3^*)$.

To make notation short, we set for $k \ge 0$
\beq\label{eq_wcaldef}
\Wcal_k^\pm = \Hdrf^\pm \cap P_k^3, \quad \Wcal_k^o = \Hdf \cap P_k^3.
\eeq

\begin{lemm}\label{lemm_div_free_tangent}
    If $f\in \Wcal_k^o$ and $p \in P_k^3$ is such that $p|_{\p\GO}=f$, then the following holds:
    \begin{enumerate}[label={\rm (\roman*)}]
    \item \label{enum_sphere_div}$\nabla \cdot p(x)=0$ for $x \in \GO$ and $\nabla \cdot p^*(x)=0$ for $x \in \Rbb^3 \setminus \GO$.
    \item \label{enum_sphere_inv} The function $u^f$ defined by \eqref{eq_uf} satisfies
    \beq\label{eq_ufp}
    u^f(x)=p(x), \ \ x \in \GO; \quad u^f(x)=p^*(x), \ \ x \in \Rbb^3 \setminus \overline{\GO}.
    \eeq
    \item \label{enum_sphere_tan}$x\cdot f(x)=0$ for all $x\in \p\GO$.
    \end{enumerate}
\end{lemm}

\begin{proof}
    By the uniqueness for the interior and exterior problems for the harmonic equations (\eqref{eq_laplace_bvp} and \eqref{eq_laplace_bvp_ext}), we have
    \[
    v_-^f= p, \quad v_+^f=p^*.
    \]
    Since $f\in \Hdf$, we thus have $\nabla\cdot p=0$ in $\GO$ and $\nabla \cdot p^*(x)=0$ for $x \in \Rbb^3 \setminus \GO$. This proves \ref{enum_sphere_div} and \ref{enum_sphere_inv}.

    Since $p$ is a homogeneous polynomial of degree $k$, we have
    \beq\label{eq_pstarp}
        p^*(x)=\frac{1}{|x|^{2k+1}}p(x) \quad \text{in } \Rbb^3 \setminus \GO.
    \eeq
    Thus we see from \ref{enum_sphere_div}
    \[
        0=\nabla\cdot p^*(x)=-\frac{2k+1}{|x|^{2k+3}}x\cdot p(x)+ \frac{1}{|x|^{2k+1}} \nabla\cdot p(x)
    \]
    if $|x| \ge 1$. Since $\nabla\cdot p$ is a homogeneous polynomial, we infer from \ref{enum_sphere_div} again that $\nabla\cdot p=0$ in $\Rbb^3$. Thus \ref{enum_sphere_tan} follows.
\end{proof}

\begin{lemm}
    \label{lemm_divfree_decomposition_sphere}
    We have the decompositions
    \begin{equation}
        \label{eq_divfree_decomposition_sphere}
        \Hdf=\overline{\bigoplus_{k=0}^\infty \Wcal_k^o}
    \end{equation}
    and
    \begin{equation}
        \label{eq_div_rot_free_decomposition_sphere}
        \Hdrf^\pm=\overline{\bigoplus_{k=0}^\infty \Wcal^\pm_k}.
    \end{equation}
\end{lemm}

\begin{proof}
    We only prove \eqref{eq_divfree_decomposition_sphere}; \eqref{eq_div_rot_free_decomposition_sphere} is proved similarly. Let $f\in \Hdf$. Since $\Hcal^3=\overline{\bigoplus_{k=0}^\infty P_k^3}$, we can take functions $f_k\in P_k^3$ such that
    \[
        f=\sum_{k=0}^\infty f_k.
    \]
    Let $u_k:=u^{f_k}$ for ease of notation. Since $\nabla\cdot u^f=0$, we infer from Corollary \ref{cor213} that
    \begin{align*}
        &\left\| \sum_{k=0}^N \nabla\cdot u_k\right\|_{L^2(\GO)}^2 + \left\| \sum_{k=0}^N \nabla\cdot u_k\right\|_{L^2(\Rbb^3 \setminus \GO)}^2 \\
        &= \left\| \nabla\cdot \left( \sum_{k=0}^N u_k-u^f \right) \right\|_{L^2(\GO)}^2 + \left\| \nabla\cdot \left( \sum_{k=0}^N u_k-u^f \right) \right\|_{L^2(\Rbb^3 \setminus \GO)}^2 \\
        &\lesssim \left\| \sum_{k=0}^N f_k-f\right\|_*.
    \end{align*}
    Taking a limit $N\to\infty$, we obtain
    \[
        \sum_{k=0}^\infty \nabla \cdot u_k=0 \quad \text{in } \Omega \cup (\Rbb^3 \setminus \overline{\GO}).
    \]
    According to Lemma \ref{lemm_div_free_tangent} \ref{enum_sphere_inv}, $u_k$ is homogeneous of degree $k$ in $\Omega$ and of degree $-k-1$ in $\Rbb^3 \setminus \overline{\GO}$. Thus we have
    \[
        \nabla \cdot u_k=0 \quad \text{in } \Omega \cup (\Rbb^3 \setminus \overline{\GO})
    \]
    for each $k$, and conclude that $f_k \in \Hdf$ for each $k$. So \eqref{eq_divfree_decomposition_sphere} follows.
\end{proof}

\begin{theo}\label{theo_Pksum}
    It holds for $k\ge 0$ that
    \beq\label{eq_Pksum}
    P_k^3 = \Wcal_k^o \oplus \Wcal_k^- \oplus \Wcal_k^+ ,
    \eeq
    where the decomposition is orthogonal with respect to the inner product \eqref{eq_inner_product}.
\end{theo}


\begin{proof}
By \eqref{eq_decomposition_div_free} and Lemma \ref{lemm_divfree_decomposition_sphere}, we have
\[
P_k^3 = \Wcal_k^o + \Wcal_k^- \oplus \Wcal_k^+.
\]
Thus, in order to prove \eqref{eq_Pksum}, it suffices to prove
\beq\label{eq:xxx}
\Wcal_k^\pm \cap \Wcal_k^o = \{ 0 \}.
\eeq

To prove \eqref{eq:xxx}, let $f \in \Wcal_k^o$. Choose $p \in P_k^3$ such that $p|_{\p\GO}=f$. By Lemma \ref{lemm_div_free_tangent} \ref{enum_sphere_tan}, we have $x \cdot p(x)=0$ in $\Rbb^3$.
Thus, we have for $j=1,2,3$
\[
0= \frac{\p}{\p x_j} (x\cdot p) = p_j(x) + \sum_{k=1}^3 x_k \frac{\p p_k}{\p x_j} .
\]
If further $f \in \Wcal_k^-$, then $\nabla \times p=0$, and hence $\frac{\p p_k}{\p x_j}= \frac{\p p_j}{\p x_k}$  in $\GO$. Thus we have
\[
0= p_j(x) + x \cdot \nabla p_j = (k+1) p_j,
\]
where the last equality holds because $p_j$ is a homogeneous polynomial of degree $k$. This proves $\Wcal_k^- \cap \Wcal_k^o = \{ 0 \}$.

By applying the same argument to $p^*$, one can prove $\Wcal_k^+ \cap \Wcal_k^o = \{ 0 \}$.
\end{proof}

Now we prove Theorem \ref{theo_sphere_orthogonal}.

\begin{proof}[Proof of Theorem \ref{theo_sphere_orthogonal}]
        According to Theorem \ref{theo_decomposition_div_free}, it suffices to prove
        \begin{equation}\label{eq_Hdfequal}
            \Hdf \cap \Hdrf^\pm=\{0\}.
        \end{equation}

        Let $f\in \Hdf \cap \Hdrf^-$. By Lemma \ref{lemm_divfree_decomposition_sphere}, we can decompose $f$ into
        \[
            f=\sum_{k=0}^\infty f^o_k=\sum_{k=0}^\infty f^-_k
        \]
        where $f^o_k\in \Wcal^o_k$ and $f^-_k\in \Wcal^-_k$. Thus we have $f^o_k=f^-_k$ for all $k\geq 0$. Since we already obtained $\Wcal^o_k\cap \Wcal^-_k=\{0\}$ in Theorem \ref{theo_Pksum}, we have $f^o_k=f^-_k=0$. Hence we obtain $f=0$. This proves $\Hdf \cap \Hdrf^-=\{0\}$.

        That $\Hdf \cap \Hdrf^+=\{0\}$ can be proved similarly.
    \end{proof}

\subsection{Dimension of homogeneous subspaces}

In this subsection, we explicitly compute the dimensions of subspaces $\Wcal_k^o$ and $\Wcal_k^\pm$ for each $k$. In the course of computation, we give an alternative proof of  Theorem \ref{theo_sphere_orthogonal}.

We begin with the following proposition which shows that $\Wcal^o_k$ is an eigenspace of $\Scal$ corresponding to an eigenvalue $-1/(\mu (2k+1))$.

\begin{prop}
    \label{prop_div_free_homogeneous_single_layer}
    If $f\in \Wcal_k^o$, then we have
    \beq\label{eq_Scal_ball}
        \Scal^{-1}[f]=-\mu (2k+1)f,
    \eeq
    and thus
    \beq\label{eq_inner_ball}
        \jbracket{f, g}_*=\mu (2k+1)\jbracket{f, g}_{\p\GO}
    \eeq
    for all $g\in \Hsb^3$.
\end{prop}

\begin{proof}
Let $f\in \Wcal_k^o$. Since $u^f=\Scal [\Scal^{-1}f]$, it follows from the jump relation \eqref{eq_jump_npds} that
    \[
        \Scal^{-1}[f]=\partial^\divf_\nu \Scal [\Scal^{-1}f]|_+-\partial^\divf_\nu \Scal [\Scal^{-1}f]|_-=\partial^\divf_\nu u^f|_+-\partial^\divf_\nu u^f|_-.
    \]
    Since $\p\GO$ is the unit sphere, the unit normal vector field on $\p\GO$ is $x$. Let $p \in P_k^3$ be such that $p|_{\partial\Omega}=f$. We have for $x \in \p\GO$
    \begin{align*}
        \partial^\divf_\nu u^f|_+(x)-\partial^\divf_\nu u^f|_-(x) &=\partial^\divf_\nu p^*(x)-\partial^\divf_\nu p(x) \\
        &=-\mu x\times \left(\nabla \times p^*(x)\right) + \mu x\times \left(\nabla \times p(x)\right) \\
        &= \mu (2k+1) \, x\times (x\times f(x)) \\
        &=-\mu (2k+1) f(x),
    \end{align*}
    where the first equality holds by Lemma \ref{lemm_div_free_tangent} \ref{enum_sphere_inv}, the second one by the definition \eqref{eq_conormal_d} of $\partial^\divf_\nu$ and Lemma \ref{lemm_div_free_tangent} \ref{enum_sphere_div}, the third one by the relation \eqref{eq_pstarp} of $p^*$ and $p$, and the last one by Lemma \ref{lemm_div_free_tangent} \ref{enum_sphere_tan}.
Thus, \eqref{eq_Scal_ball} follows.
\end{proof}

\begin{lemm}\label{lemm_orbital_angular_momentum}
    Let $p \in P_k^3$ be such that $p|_{\p\GO} \in \Wcal_k^o$. If we set
    \beq\label{eq_qdef}
        q(x):=-\frac{1}{k(k+1)}x\cdot (\nabla \times p),
    \eeq
    then $q \in P_k$ and
    \begin{equation}
        \label{eq_orbital_angular_momentum}
        p(x)=x\times \nabla q(x).
    \end{equation}
\end{lemm}

\begin{proof}
    Straight-forward computations yield
    \begin{align}
        \nabla q(x)
        &=-\frac{1}{k(k+1)} \Big( \big( (\nabla\times p)\cdot \nabla \big) x+(x\cdot \nabla)(\nabla\times p)+x\times \big( \nabla \times (\nabla \times p) \big) \nonumber \\
        &\quad+(\nabla \times p)\times (\nabla \times x) \Big) \nonumber \\
        &=-\frac{1}{k(k+1)} \Big( \nabla\times p+(x\cdot \nabla)(\nabla\times p)+x\times \big( \nabla\times (\nabla \times p) \big) \Big). \label{eq_nabla_q}
    \end{align}
    Since $p|_{\p\GO} \in \Hdf$, $\nabla\cdot p=0$ by Lemma \ref{lemm_div_free_tangent} \ref{enum_sphere_div}, and hence
    \[
    \nabla \times (\nabla \times p)=\nabla (\nabla \cdot p)-\Delta p=0.
    \]
    Since $\nabla \times p$ is a homogeneous polynomial of degree $k-1$, we have
    \begin{align*}
        (x\cdot \nabla)(\nabla \times p)=(k-1)\nabla\times p.
    \end{align*}
    It then follows from \eqref{eq_nabla_q} that
    \[
        \nabla q(x)=-\frac{1}{k+1}\nabla \times p(x).
    \]
    This implies that $\GD q=0$. Since $q$ is homogeneous of degree $k$, we have $q \in P_k$. Moreover, we have
    \begin{equation}
        \label{eq_nabla_q_reduced}
        x\times \nabla q(x)=-\frac{1}{k+1}x\times (\nabla \times p(x)).
    \end{equation}

    Since $x\cdot p(x)|_{\p\GO}=0$ by Lemma \ref{lemm_div_free_tangent} \ref{enum_sphere_tan} and $x\cdot p(x)$ is a homogeneous polynomial, we have $x\cdot p(x)=0$ for all $x\in \Rbb^3$. Thus we have
    \beq\label{eq_gradxp}
    \begin{aligned}
        0&=\nabla (x\cdot p)=p + p\times (\nabla \times x) +(x\cdot \nabla)p+x\times (\nabla \times p)+p\times (\nabla \times x) \\
        &=(k+1)p+x\times (\nabla \times p),
    \end{aligned}
    \eeq
    where the last equality holds since $(p\cdot \nabla)x=p$, $\nabla \times x=0$, and $(x \cdot \nabla)p=kp$ by the assumption that $p$ is a homogeneous polynomial of degree $k$. It then follows from \eqref{eq_nabla_q_reduced} that
    \[
        x\times \nabla q(x)=-\frac{1}{k+1}\times (-(k+1)p(x))=p(x),
    \]
    which is the desired identity \eqref{eq_orbital_angular_momentum}.
\end{proof}

Proposition \ref{prop_div_free_homogeneous_single_layer} and Lemma \ref{lemm_orbital_angular_momentum} yield an alternative proof of Theorem \ref{theo_sphere_orthogonal}.

\begin{proof}[Proof of Theorem \ref{theo_sphere_orthogonal} by Lemma \ref{lemm_orbital_angular_momentum}]
    Let $f\in \Wcal_k^o$. Take $p\in P_k^3$ such that $f=p|_{\p\GO}$. By Lemma \ref{lemm_orbital_angular_momentum}, we can take $q\in P_k$ such that $p=x\times \nabla q$. Hence we have $v_-^f=x\times \nabla q$ ($v_-^f$ is the solution to the interior boundary value problem \eqref{eq_laplace_bvp}). Moreover, since the Kelvin transform becomes $p^*(x)=(x\times \nabla q)^*(x)=x\times \nabla q^*(x)$, we have the solution to the exterior problem given by $v^f_+=x\times \nabla q^*$. Then, by Proposition \ref{prop_div_free_homogeneous_single_layer}, we have
    \begin{align*}
        \jbracket{f, g}_*
        &=\mu (2k+1)\jbracket{x\times \nabla q, g}_{\p\GO}
        =\mu(2k+1)\int_{\p\GO}(\nv_x \times (\nabla q))\cdot v_-^{g}(x)\, \df \sigma (x) \\
        &=\mu(2k+1)\int_{\p\GO}((\nabla q)\times v_-^{g}(x))\cdot \nv_x \, \df \sigma (x) \\
        &=\mu(2k+1)\int_{\GO}\nabla \cdot ((\nabla q)\times v_-^{g}(x)) \, \df x \\
        &=\mu(2k+1)\int_{\GO}(\underbrace{(\nabla\times \nabla q)}_{=0}\cdot v_-^{g}(x)-\nabla q \cdot (\underbrace{\nabla \times v_-^{g}(x)}_{=0})) \, \df x \\
        &=0
    \end{align*}
    for all $g\in \Wcal^-_k$. Similarly, one can show
    \[
        \jbracket{f, g}_*=0
    \]
    for all $g\in \Wcal^+_k$. If $g \in P_l^3$ for some $l \neq k$, then by Lemma \ref{lemm_orbital_angular_momentum} we have
    \begin{align*}
        \jbracket{f, g}_*
        =\mu (2k+1)\jbracket{f, g}_{L^2(\p\GO)^3} =0.
    \end{align*}
Therefore, we have
    \[
    \Wcal_k^o \subset (\Hdrf^- +\Hdrf^+)^\perp.
    \]
    This together with Lemma \ref{lemm_divfree_decomposition_sphere} yields $\Hdf\subset (\Hdrf^-+\Hdrf^+)^\perp$.
\end{proof}


We now prove the following theorem which characterizes the subspaces $\Wcal_k^o$, $\Wcal_k^-$, and $\Wcal_k^+$ in terms of homogeneous harmonic polynomials.

\begin{theo}
    \label{theo_structure_df_drf}
    For $k\geq 1$, the following maps are linear isomorphisms:
    \begin{align}
        q\in P_k &\longmapsto x\times \nabla q|_{\p\GO}\in \Wcal_k^o, \label{eq_df_sphere_generate}\\
        q\in P_{k+1} & \longmapsto \nabla q|_{\p\GO}\in \Wcal_k^-, \label{eq_drfm_sphere_generate}\\
        q\in P_{k-1} & \longmapsto \nabla q^*|_{\p\GO}\in \Wcal_k^+. \label{eq_drfp_sphere_generate}
    \end{align}
\end{theo}

A few remarks are in order before proving Theorem \ref{theo_structure_df_drf}. We first observe that since $\dim P_k=2k+1$ for each $k \ge 1$, we can compute the dimensions of the subspaces $\Wcal_k^o$, $\Wcal_k^-$, and $\Wcal_k^+$ as the following corollary shows.
\begin{coro}\label{coro_dim}
The following holds for each $k \ge 1$:
\beq\label{eq_dim_df_generate}
        \dim \Wcal_k^o =2k+1, \quad \dim \Wcal_k^- =2k+3, \quad
        \dim \Wcal_k^+ =2k-1 .
\eeq
In particular, the subspaces $\Hdf$, $\Hdrf^-$ and $\Hdrf^+$ are infinite-dimensional.
\end{coro}

We will use Corollary \ref{coro_dim} to prove Theorem \ref{theo_infinite_dim_c1a} which asserts that subspaces $\Hdf$ and  $\Hdrf^\pm$ are of infinite dimensions on domains with Lipschitz boundaries.

One can prove the decomposition \eqref{eq_Pksum} using \eqref{eq_dim_df_generate}.
In fact, according to Theorem \ref{theo_decomposition_div_free} and Lemma \ref{lemm_divfree_decomposition_sphere}, it holds for $k \ge 0$ that
\[
P_k^3 \subset \Wcal_k^o + \Wcal_k^- + \Wcal_k^+.
\]
By \eqref{eq_dim_df_generate}, we have
\[
\dim P_k^3 = 3(2k+1) = \dim \Wcal_k^o + \dim \Wcal_k^- + \dim \Wcal_k^+
\]
if $k \ge 1$. Thus \eqref{eq_Pksum} holds for $k \ge 1$.
One can easily see that
    \[
        \dim \Wcal_0^o= \dim \Wcal_0^+=0, \quad
        \dim \Wcal_0^-=3 .
    \]
Thus \eqref{eq_Pksum} holds for $k =0$   .

\begin{proof}[Proof of Theorem \ref{theo_structure_df_drf}]
    For $q\in P_k$, let $p(x)=x\times \nabla q(x)$. Then, we have
    \[
        \nabla \cdot p(x) =(\nabla \times x)\cdot \nabla q+(\nabla\times \nabla q)\cdot x=0.
    \]
    Since
    \[
        p^*(x)=\frac{1}{|x|^{2k+1}} (x\times \nabla q(x)),
    \]
    we also have
    \[
        \nabla\cdot p^*(x)=-\frac{2k+1}{|x|^{2k+3}} x\cdot (x\times \nabla q(x)) + \frac{1}{|x|^{2k+1}} \nabla\cdot p(x)=0.
    \]
    Thus the function $f:=x\times \nabla q|_{\p\GO}$ belongs to $\Wcal_k^o$, that is, the linear transform defined in \eqref{eq_df_sphere_generate} maps $P_k$ into $\Wcal_k^o$.

    Let $f\in \Wcal_k^o$ and let $p \in P_k^3$ be such that $p|_{\partial\Omega}=f$. Lemma \ref{lemm_orbital_angular_momentum} says that the mapping
    \[
        f=p|_{\partial\Omega} \in \Wcal_k^o \longmapsto -\frac{1}{k(k+1)}x\cdot (\nabla \times p)\in P_k
    \]
    is the inverse mapping of the linear mapping \eqref{eq_df_sphere_generate}. Thus it is an
    isomorphism.

    If $q\in P_{k+1}$, then we have
    \begin{align*}
        \nabla \cdot \nabla q=\Delta q=0, \quad \nabla\times (\nabla q)=0.
    \end{align*}
    Thus $\nabla q|_{\p\GO}\in \Hdrf^-\cap P_k^3$.

    To prove that the linear mapping \eqref{eq_drfm_sphere_generate} is an isomorphism, we let $f\in \Hdrf^-\cap P_k^3$ and take $p\in P_k^3$ such that $p|_{\partial\Omega}=f$. We set $q=(k+1)^{-1}x\cdot p(x)$. Then we have
    \[
        \nabla q(x)=\frac{1}{k+1}(\underbrace{(x\cdot \nabla)p}_{=kp}+\underbrace{(p\cdot \nabla) x}_{=p}+x\times (\underbrace{\nabla \times p}_{=0})+p\times (\underbrace{\nabla \times x}_{=0}))=p(x).
    \]
    Moreover, we obtain $\Delta q=\nabla\cdot p=0$. Hence $q\in P_{k+1}^3$. Therefore the mapping \eqref{eq_drfm_sphere_generate} is surjective. If $q\in P_{k+1}$ satisfies $\nabla q|_{\p\GO}=0$, then $\nabla q=0$ for all $\Rbb^3$ since $\nabla q$ is a homogeneous polynomial. Thus $q(x)$ is a constant. Since $q(x)$ is a homogeneous polynomial of positive degree, we have $q(x)=0$. Therefore the linear mapping \eqref{eq_drfm_sphere_generate} is an isomorphism.

    If $q\in P_{k-1}$, then
    we have
    \[
        \nabla \cdot \nabla q^*=\Delta q^*=0, \quad \nabla \times \nabla q^*=0, \quad \Delta (\nabla q^*)=\nabla (\Delta q^*)=0.
    \]
    Thus $\nabla q^*|_{\p\GO}\in \Wcal_k^+$.

    To prove that the linear mapping \eqref{eq_drfp_sphere_generate} is an isomorphism, we let $f\in \Hdrf^+\cap P_k^3$ and take $p\in P_k^3$ such that $p|_{\partial\Omega}=f$. Since $f\in \Hdrf^+$, we have $\nabla \times p^*=0$. Set $q(x):=-k^{-1}(x\cdot p^*)^*(x)$. Then we have $q^*(x)=-k^{-1}x\cdot p^*(x)$ and
    \[
        \nabla q^*(x)=-\frac{1}{k}(\underbrace{(x\cdot \nabla)p^*}_{=-(k+1)p^*}+\underbrace{(p^*\cdot \nabla) x}_{=p^*}+x\times (\underbrace{\nabla \times p^*}_{=0})+p^*\times (\underbrace{\nabla \times x}_{=0}))=p^*(x).
    \]
    Moreover, we obtain $\Delta q^*=\nabla\cdot p^*=0$. Hence $q=(q^*)^*\in P_{k-1}^3$. Therefore the mapping \eqref{eq_drfp_sphere_generate} is surjective. If $q\in P_{k-1}$ satisfies $\nabla q^*|_{\p\GO}=0$, then $\nabla q^*=0$ for all $\Rbb^3\setminus \{ 0\}$ since $|x|^{2k+1}\nabla q^*$ is a homogeneous polynomial. Thus $q^*(x)$ is a constant. Hence $q(x)=c|x|^{-1}$ for some constant $c$. Since $q(x)$ is a polynomial, we have $q(x)=0$. Therefore the linear mapping \eqref{eq_drfp_sphere_generate} is an isomorphism.
\end{proof}

\begin{exam*}
    We give an explicit representation of $\Wcal_k^o$ and $\Wcal_k^\pm$ for $k=0, 1, 2$ as follows.
    \begin{align*}
        \Wcal_0^o &=\{ 0\}, \\
        \Wcal_0^- &=\rmop{span}\{ (1, 0, 0), (0, 1, 0), (0, 0, 1)\}, \\
        \Wcal_0^+ &=\{ 0\},
    \end{align*}
    \begin{align*}
        \Wcal_1^o &=\rmop{span}\{ (0, -z, y), (z, 0, -x), (-y, x, 0)\}, \\
        \Wcal_1^- &=\rmop{span}\{ (y, x, 0), (z, 0, x), (0, z, y), (x, 0, -z), (0, y, -z)\}, \\
        \Wcal_1^+ &=\rmop{span}\{ (x, y, z)\},
    \end{align*}
    and
    \begin{align*}
        \Wcal_2^o
            &=\rmop{span} \left\{
        \begin{aligned}
            &(-y^2+z^2, xy, -zx),
            (-xy, x^2-z^2, yz), \\
            &(zx, -yz, -x^2+y^2),
            (yz, 0, -xy),
            (0, zx, -xy)
        \end{aligned} \right\}, \\
        \Wcal_2^- &=\rmop{span}\left\{
        \begin{aligned}
            &(-2xy, -x^2-y^2+2z^2, 4yz), (2xy, x^2-y^2, 0), \\
            &(-x^2-y^2+2z^2, -2xy, 4zx), (x^2-y^2, -2xy, 0), \\
            &(-2zx, -2yz, -x^2-y^2+2z^2), (yz, zx, xy), \\
            &(2zx, -2yz, x^2-y^2)
        \end{aligned}\right\}, \\
        \Wcal_2^+ &=\rmop{span}\left\{
            \begin{aligned}
                &(-3xy, x^2-2y^2+z^2, -3yz), \\
                &(-2x^2+y^2+z^2, -3xy, -3zx), \\
                &(-3zx, -3yz, x^2+y^2-2z^2)
            \end{aligned}    \right\}.
    \end{align*}
\end{exam*}

\subsection{ENP eigenfunctions on ball}
According to Theorem \ref{thm_eigenvalue3D} whose proof will be given in section \ref{sec:eigenspace}, the eigenvalues of the eNP operator on domains with $C^{1,\alpha}$ boundary consist of three sequences converging to $0$, $-k_0$ and $k_0$, respectively. If the given domain is a ball, then the eigenvalues are computed explicitly in \cite{DLL19}: they are
\begin{align*}
            \xi_o^k&:=\frac{3}{4k+2}, \\
            \xi_-^k&:=\frac{3\lambda-2\mu (2k^2-2k-3)}{2(\lambda+2\mu)(4k^2-1)}, \\
            \xi_+^k&:=\frac{-3\lambda+2\mu (2k^2+2k-3)}{2(\lambda+2\mu)(4k^2-1)}
        \end{align*}
for $k \ge 1$. The corresponding eigenfunctions are found in the same paper. Suppose that $\Omega\subset \Rbb^3$ is the unit ball centered at $0$. It turns out that if $p_{k}$ is a homogeneous harmonic polynomial of degree $k$, then $\nabla p_{k}(x)\times x$, $\nabla p_{k}(x)$ and $\nabla p_{k-1}^*(x)$ ($x \in \partial \Omega$) are eigenfunctions corresponding to $\xi_o^k$, $\xi_-^k$ and $\xi_+^k$, respectively. According to Theorem \ref{theo_structure_df_drf}, it amounts to saying that $\Wcal_k^o$, $\Wcal_k^-$ and $\Wcal_k^+$ are eigenspaces of the eNP operator corresponding to $\xi_o^k$, $\xi_-^k$ and $\xi_+^k$, respectively. In particular, $\Hdf$, $\Hdrf^-$ and $\Hdrf^+$ are invariant under $\Kcal$, and they are eigenspaces corresponding to the eigenvalues converging to $0$, $-k_0$ and $k_0$, respectively. If $\GO$ is a bounded domain in $\Rbb^3$ whose boundary is $C^{1,\Ga}$ for some $\Ga >1/2$, then $\Hdf$, $\Hdrf^-$ and $\Hdrf^+$ may not be invariant under $\Kcal$ in general. However, Theorem \ref{theo_enp_eigenfunctions} shows that they are eigenspaces in the limiting sense.

\section{Proof of Theorem \ref{theo_free_np_3D}}\label{sec:theo12}

We now consider general bounded domains $\Omega$ in $\Rbb^3$ to prove Theorem \ref{theo_free_np_3D}. In this section, the boundary $\partial\Omega$ is assumed to be $C^{1, \alpha}$ for some $\alpha > 1/2$ except for Lemma \ref{lemm_np23_compact} where $\alpha$ is assumed to be larger than $0$. We begin with the following lemma, which states that the eNP operator $\Kcal$ and $\Kcal^\divf$ respectively defined in \eqref{eq_np} and \eqref{eq_kd}, are equivalent modulo a compact operator:

\begin{lemm}
    \label{lemm_np23_compact}
If $\p\GO$ is $C^{1,\Ga}$ for some $\Ga>0$, then $\Kcal +2k_0 \Kcal^\divf$ is compact on $\Hsb^3$.
\end{lemm}

\begin{proof}
We recall the decompositions \eqref{eq_kd_decomposition} and \eqref{eq_np_decomposition}, and the operators $\Kcal_j$ ($j=1, 2, 3$) therein. We claim that, if $\p\GO$ is $C^{1,\Ga}$ for some $\Ga>0$, then the operators $\Kcal_2$ and $\Kcal_3$ are compact on $\Hsb^3$.

We apply Corollary \ref{coro_T_cpt}. One can directly check that $\Kcal_2$ and $\Kcal_3$ satisfy the conditions \eqref{eq_singularity_0} and \eqref{eq_singularity_1} with $\beta=\alpha$ and $\gamma=1$. For instance, if we denote the (matrix-valued) integral kernel of $\Kcal_3$ by $K_3(x, y)=(K_{3, ij}(x, y))_{i, j=1}^3$, where
\[
    K_{3, ij} (x, y):=\frac{1}{4\pi}\frac{((x-y)\cdot \nv_y)(x_i-y_i)(x_j-y_j)}{|x-y|^5},
\]
then we have
\[
    |K_3(x, y)|\lesssim \frac{|(x-y)\cdot \nv_y|}{|x-y|^3}\lesssim \frac{1}{|x-y|^{2-\alpha}}
\]
because $\partial\Omega$ is $C^{1, \alpha}$. If $2|x-z|<|x-y|$, then
\begin{align*}
    &|K_{3, ij} (x, y)-K_{3, ij} (z, y)| \\
    &\leq
    \frac{1}{4\pi}\frac{|(x-z)\cdot (\nv_y-\nv_x)+(x-z)\cdot \nv_x||z_i-y_i||z_j-y_j|}{|z-y|^5} \\
    &\quad +\frac{|(x-y)\cdot \nv_y|}{4\pi}\left|\frac{(x_i-y_i)(x_j-y_j)}{|x-y|^5}-\frac{(z_i-y_i)(z_j-y_j)}{|z-y|^5}\right| .
\end{align*}
Since $2|x-z|<|x-y|$ implies $|z-y|>|x-y|/2$, we have
\begin{align*}
    |K_{3, ij} (x, y)-K_{3, ij} (z, y)|
    &\lesssim \frac{|x-z|+|x-z|^{1+\alpha}}{|x-y|^{3-\alpha}}
    +|x-y|^{1+\alpha}\frac{|x-z|}{|x-y|^4} \\
    &\lesssim \frac{|x-z|}{|x-y|^{3-\alpha}}.
\end{align*}
Moreover, we obtain $\Kcal_3[b]=-b/6$ for any constant vector field $b\in \Rbb^3$ on $\partial\Omega$. In fact, since $\Kcal_2$ is the NP operator for the Laplace operator, it is well-known that $\Kcal_2[b]=b/2$. It is also known that $\Kcal [b]=b/2$ for any Lam\'e constants (see \cite[Lemma 2.2]{AJKKY}). It thus follows from \eqref{eq_np_decomposition} that
\[
    \frac{1}{2} b =2k_0 \left(\Kcal_1[b]+\frac{1}{2} b\right)-3(1-2k_0)\Kcal_3[b].
\]
Since this relation holds for any Lam\'e constants, or for any $k_0 \in (0, \varepsilon)$ for some $\varepsilon >0$, we see that $\Kcal_3[b]=-b/6$. We then apply Corollary \ref{coro_T_cpt} to infer that $\Kcal_2$ and $\Kcal_3$ are compact on $\Hsb^3$.

Now we substitute the relation $\Kcal^\divf=-\Kcal_1+\Kcal_2$ given in \eqref{eq_kd_decomposition} to \eqref{eq_np_decomposition} to obtain the conclusion.
\end{proof}

In the rest of this section, we always assume that $\partial\Omega$ is $C^{1, \alpha}$ for some $\alpha>1/2$. For the proof of Theorem \ref{theo_free_np_3D}, we derive a representation formula for the eNP operator $\Kcal$ which relates $\Kcal$ with the space $\Hdf$ (Proposition \ref{lemm_np_modulo_cpt}). We begin with the following lemma.

\begin{lemm}\label{lemm:multi}
Let $\psi \in C^{0,\Ga}(\p\GO)$ for some $\Ga>1/2$. The mapping $\varphi \mapsto \varphi \psi$ is bounded on $\Hsb$ and $\Hsb^*$.
\end{lemm}

\begin{proof}
Since $\partial\Omega$ is two-dimensional, the Sobolev norm on $\Hsb=H^{1/2}(\p\GO)$ is equivalent to following norm (see \cite{Gilbarg-Trudinger01})
    \begin{equation}\label{eq_Sobolev_Besov}
        \left(\| \varphi\|_{L^2(\p\GO)}^2+ \int_{\p\GO} \int_{\p\GO} \frac{|\varphi (x)-\varphi (y)|^2}{|x-y|^3}\, \df \sigma (x)\df \sigma (y)\right)^{1/2}.
    \end{equation}
If $\psi \in C^{0,\Ga}(\p\GO)$ for some $\Ga>1/2$, then $|\psi(x)-\psi(y)| \lesssim |x-y|^{\Ga}$, and hence
    \begin{align*}
        & \int_{\p\GO} \int_{\p\GO} \frac{|(\varphi\psi) (x) -(\varphi\psi)(y)|^2}{|x-y|^3}\, \df \sigma (x) \df \sigma (y) \\
        & \lesssim \int_{\p\GO} \int_{\p\GO} \frac{|\varphi (x)-\varphi (y)|^2}{|x-y|^3}\, \df \sigma (x) \df \sigma (y) \\
        &\quad +  \int_{\p\GO} \int_{\p\GO} \frac{|\varphi (x)|^2}{|x-y|^{3-2\Ga}}\, \df \sigma (x) \df \sigma (y) \\
        & \lesssim \int_{\p\GO} \int_{\p\GO} \frac{|\varphi (x)-\varphi (y)|^2}{|x-y|^3}\, \df \sigma (x) \df \sigma (y) + \int_{\p\GO} |\varphi (x)|^2\, \df \sigma (x),
    \end{align*}
    where the last inequality holds since $\Ga >1/2$. Thus, we have
    \[
    \| \varphi \psi \|_\Hsb \lesssim \| \varphi \|_\Hsb
    \]
    for all $\varphi \in \Hsb$.

    If $\varphi \in \Hsb^*$ and $\eta \in \Hsb$, then
    \[
    | \langle \varphi \psi, \eta \rangle_{\p\GO} | \le \| \varphi \|_{\Hsb^*} \| \eta \psi \|_\Hsb \lesssim \| \varphi \|_{\Hsb^*} \| \eta \|_\Hsb.
    \]
    Here $\langle \ , \ \rangle_{\p\GO}$ is the duality pairing between $\Hsb^*$ and $\Hsb$ as before. Thus, we have
    \[
    \| \varphi \psi \|_{\Hsb^*} \lesssim \| \varphi \|_{\Hsb^*}
    \]
    for all $\varphi \in \Hsb^*$.
\end{proof}

Let $g\in \Hsb^3$. A straight-forward calculation shows that
\beq
    \label{eq_div_single_layer}
    \nabla\cdot \Scal [g](x)=\frac{k_0}{2\pi\Gm}\int_{\p\GO} \frac{(x-y)\cdot g(y)}{|x-y|^3}\, \df \Gs (y)
\eeq
for $x\in \Rbb^3\setminus \p\GO$.  Using the single layer potential $\Scal_0$ with respect to the Laplacian as defined in \eqref{eq_lapla_single}, it can be written as
\[
    \nabla\cdot \Scal [g](x)= \frac{2k_0}{\Gm} \nabla \cdot \Scal_0[g](x).
\]
We now invoke the jump formula \eqref{eq_jump_grad} for $\Scal_0$ to derive
the following jump relation:
    \begin{equation}
        (\nabla\cdot \Scal [g])|_\pm (x)=\frac{2k_0}{\Gm} \left(\Jcal_\mathrm{div}[g](x)\pm \frac{1}{2}\nv_x\cdot g(x)\right), \quad x \in \p\GO \label{eq_jump_j_div}
    \end{equation}
for $g\in \Hsb^3$, where the boundary integral operator $\Jcal_\mathrm{div}$ is defined by
\beq
    \label{eq_div_sl_bdry}
    \Jcal_\mathrm{div}[g](x):=\frac{1}{4\pi}\rmop{p.v.}\int_{\p\GO}\frac{(x-y)\cdot g(y)}{|x-y|^3}\, \df \Gs (y), \quad x \in \p\GO.
\eeq

Let $\Jcal_\mathrm{div}^*$ be the (formal) $L^2$-adjoint of $\Jcal_\mathrm{div}$, namely,
\beq
    \label{eq_jds_explicit}
    \Jcal_\mathrm{div}^*[\varphi](x)=-\frac{1}{4\pi}\rmop{p.v.}\int_{\p\GO} \frac{x-y}{|x-y|^3}\varphi(y)\, \df \Gs (y).
\eeq
The jump formula \eqref{eq_jump_grad} can be rewritten as
   \begin{equation}
   \nabla \Scal_0[\varphi]|_\pm (x)= -\Jcal_\mathrm{div}^*[\varphi](x)\pm \frac{1}{2} \varphi(x) \nv_x, \quad x \in \p\GO \label{eq_jump_ds0}
    \end{equation}
for $\varphi\in \Hcal$.

\begin{lemm}\label{lemm:Jdiv}
    If $\p\GO$ is $C^{1,\Ga}$ for some $\Ga > 1/2$, then $\Jcal_\mathrm{div}^*$ is bounded from $\Hsb$ into $\Hsb^3$ and from $\Hsb^*$ into $(\Hsb^3)^*$, and $\Jcal_\mathrm{div}$ is bounded from $\Hsb^3$ into $\Hsb$ and from $(\Hsb^3)^*$ into $\Hsb^*$.
    \end{lemm}

    \begin{proof}
        Since $\p\GO$ is $C^{1,\Ga}$, the unit normal vector field $\nv$ is $C^{0,\Ga}$. Let $\tau_1$ and $\tau_2$ be locally defined orthonormal vector fields tangent to $\p\GO$. We may choose them to be $C^{0,\Ga}$. Then, we have
    \[
    \Jcal_\mathrm{div}^*[\varphi]= (\Jcal_\mathrm{div}^*[\varphi] \cdot \nv) \nv + \sum_{j=1}^2 (\Jcal_\mathrm{div}^*[\varphi] \cdot \tau_j) \tau_j
    \]
    for $\varphi \in \Hsb^*$.

    Since $\Jcal_\mathrm{div}^*[\varphi] \cdot \nv = -\Kcal_0^*[\varphi]$ where $\Kcal_0$ is the NP operator with respect to the Laplace operator. Thus, we have
    \[
    \| \Jcal_\mathrm{div}^*[\varphi] \cdot \nv \|_{\Hsb^*} \le C \| \varphi \|_{\Hsb^*}
    \]
    for some constant $C$. Note that $\Jcal_\mathrm{div}^*[\varphi] \cdot \tau_j= -\tau_j \cdot \nabla \Scal_0[\varphi]$. Since $\Scal_0$ is bounded as a mapping from $\Hsb^*=H^{-1/2}(\p\GO)$ into $\Hsb=H^{1/2}(\p\GO)$ \cite{Verchota84} and $\tau_j \cdot \nabla$ is a tangential derivative on $\p\GO$, we have
    \[
    \| \Jcal_\mathrm{div}^*[\varphi] \cdot \tau_j \|_{\Hsb^*} \le C \| \varphi \|_{\Hsb^*}
    \]
    for some constant $C$. We then glue the above local estimates by a partition of unity and infer from Lemma \ref{lemm:multi} that $\Jcal_\mathrm{div}^*: \Hsb^* \to (\Hsb^*)^3$ is bounded.

    By duality, we see that $\Jcal_\mathrm{div}: \Hsb^3 \to \Hsb$ is bounded, from which we immediately see that $\Jcal_\mathrm{div}^*: \Hsb \to \Hsb^3$ and $\Jcal_\mathrm{div}: (\Hsb^3)^* \to \Hsb^*$ are bounded.
    \end{proof}

We define operators $\Qcal_\pm$ by
\begin{equation}
    \label{eq_qpm}
    \Qcal_\pm [\varphi]:=- \frac{\mu}{k_0} \nabla \Scal_0 [\varphi]|_\pm
\end{equation}
for $\varphi \in \Hcal$. We show that $\Qcal_\pm$ is extended as a bounded from $\Hsb^*$ to $(\Hsb^3)^*$. To do so, we invoke the following theorem.

\begin{theo}[{\cite[Corollary 6.6]{Gesztesy-Mitrea11}}]\label{theo_trace_harmonic}
    If $\Omega\subset \Rbb^m$ ($m\geq 2$) is a bounded domain with the $C^{1, \alpha}$ boundary for some $\alpha>1/2$, then the trace $u\in C^\infty (\overline{\Omega})\mapsto u|_-\in L^2 (\partial\Omega)$ is extended to a unique bounded linear operator
    \[
        \{ u\in L^2 (\Omega) \mid \Delta u\in L^2 (\Omega)\}
        \longrightarrow H^{-1/2}(\partial\Omega).
    \]
    Here, we equip the space $\{ u\in L^2 (\Omega) \mid \Delta u\in L^2 (\Omega)\}$ with the graph norm $\|u\|_{L^2 (\Omega)}+\|\Delta u\|_{L^2 (\Omega)}$.
\end{theo}

\begin{lemm}\label{lem:44}
If $\Omega\subset \Rbb^3$ is a bounded domain with the $C^{1, \alpha}$ boundary for some $\alpha>1/2$, then $\Qcal_\pm$ is bounded from $\Hsb^*$ to $(\Hsb^3)^*$.
\end{lemm}

\begin{proof}
We denote the norms on $(\Hsb^3)^*$ and $\Hsb^*$ by $\|\cdot\|_{(\Hsb^3)^*}$ and $\|\cdot\|_{\Hsb^*}$ respectively. Let $\varphi \in \Hcal^*$. We apply Theorem \ref{theo_trace_harmonic} to each component of $u=\nabla \Scal_0 [\varphi]$ to have
$$
\| \Qcal_-[\varphi] \|_{(\Hcal^*)^3} \lesssim \|\nabla \Scal_0 [\varphi]\|_{L^2 (\Omega)^3} \lesssim \| \varphi \|_{\Hcal^*}.
$$

To deal with $\Qcal_+$, let $B$ be a ball such that $\overline{\Omega}\subset B$. Then, by Theorem \ref{theo_trace_harmonic} again, we have
    \[
    \| u|_-\|_{H^{-1/2}(\partial B \cup \partial \Omega)^3}
        \lesssim \| \nabla \Scal_0 [\varphi] \|_{L^2 (B\setminus \Omega)^3}
        \lesssim \| \nabla \Scal_0 [\varphi] \|_{L^2 (\Rbb^m\setminus \Omega)^3} \lesssim \| \varphi \|_{\Hcal^*}.
 \]
Here we emphasize that $u|_-$ in the left hand side stands for the trace on $\partial B \cup \partial\Omega$ from the interior of $B\setminus \Omega$, which coincides on $\partial\Omega$ with the trace $u|_+$ from the exterior of $\Omega$. It thus follows that
    \[
    \| \Qcal_+[\varphi] \|_{(\Hcal^*)^3} \lesssim \| u|_-\|_{H^{-1/2}(\partial B \cup \partial \Omega)^3}
        \lesssim \| \varphi \|_{\Hcal^*},
 \]
as desired.
\end{proof}

We then obtain the following corollary from the formula \eqref{eq_jump_ds0}:
\begin{coro}
If $\Omega\subset \Rbb^3$ is a bounded domain with the $C^{1, \alpha}$ boundary for some $\alpha>1/2$, then the following identity holds for all $\varphi \in \Hcal^*$:
\begin{align}
   \Qcal_\pm [\varphi] = \frac{\mu}{k_0} \left(\Jcal_\mathrm{div}^*[\varphi]\mp \frac{1}{2} \varphi \nv \right) . \label{eq_jump_ds0_Q}
\end{align}
\end{coro}

We now prove the following proposition which plays a crucial role in proving Theorem \ref{theo_free_np_3D}.

\begin{prop}\label{lemm_np_modulo_cpt}
    Assume that the boundary $\p\GO$ is $C^{1, \Ga}$ for some $\alpha>1/2$. Then the following identity holds for $f \in \Hsb^3$:
    \begin{equation}
        \label{eq_np_modulo_cpt}
        \Kcal [f]= - k_0\Scal \Qcal_-[\nabla \cdot u^f|_+]+ k_0\Scal \Qcal_+[\nabla \cdot u^f|_-]+\Tcal [f],
    \end{equation}
    where $\Tcal$ is a compact operator  on $\Hsb^3$.
\end{prop}

Here, as we will observe in the proof, the function $\Scal \Qcal_+[\nabla\cdot u^f|_-]$ is the image of the following composition of the bounded operators.
\begin{center}
    \begin{tikzcd}[row sep=0.2cm, column sep=0.8cm]
        \Hsb^3 \arrow[r, "u^{[\cdot]}"] & H^1 (\Omega)^3 \arrow[r, "\nabla\cdot {[\cdot]}"] & L^2 (\Omega) \arrow[r, "{[\cdot]}|_-"] & \Hsb^* \arrow[r, "\Qcal_+"] & (\Hsb^3)^* \arrow[r, "\Scal"] & \Hsb^3 \\
        f \arrow[r, mapsto] & u^f \arrow[r, mapsto] & \nabla\cdot u^f \arrow[r, mapsto] & \nabla\cdot u^f|_- \arrow[r, mapsto] & \text{(omitted)}
    \end{tikzcd}
\end{center}
The operator $f\in \Hsb^3 \mapsto \Scal\Qcal_-[\nabla\cdot u^f|_+]\in \Hsb^3$ is described similarly.

\begin{proof}
    Let $f \in \Hsb^3$ and $g= \Scal^{-1}[f]\in (\Hsb^3)^*$ so that $u^f= \Scal[g]$. Since $\Lcal_{\lambda, \mu}u^f=0$ in $\Rbb^3\setminus \partial\Omega$, $\nabla\cdot u^f$ is harmonic there. Moreover, $\nabla\cdot u^f|_\Omega\in L^2 (\Omega)$ and $\nabla \cdot u^f|_{\Rbb^3\setminus \overline{\Omega}}\in L^2 (\Rbb^3\setminus \overline{\Omega})$ by Lemma \ref{lemm_div_rot_npd_identity}. Thus, we can apply Theorem \ref{theo_trace_harmonic} and the argument in the proof of Lemma \ref{lem:44} to infer that $(\nabla\cdot u^f)|_-$ and $(\nabla\cdot u^f)|_+$ belong to $(\Hsb^3)^*$.

It follows from the jump relations \eqref{eq_jump_j_div} and \eqref{eq_jump_ds0_Q} that
    \begin{align*}
    \Qcal_-[\nabla \cdot u^f|_+] &= \frac{k_0}{\Gm} \Qcal_- \left[ \nv \cdot g + 2\Jcal_\mathrm{div}[g] \right] \\
    &=2 \Jcal_\mathrm{div}^*\Jcal_\mathrm{div}[g]+\Jcal_\mathrm{div}^*[ \nv \cdot g] + \Jcal_\mathrm{div}[g] \nv
    + \frac{1}{2} (\nv\cdot g)\nv .
    \end{align*}
    Likewise, we have
    \begin{align*}
    \Qcal_+[\nabla \cdot u^f|_-] =
    2 \Jcal_\mathrm{div}^*\Jcal_\mathrm{div}[g]- \Jcal_\mathrm{div}^*[ \nv \cdot g] - \Jcal_\mathrm{div}[g] \nv
    + \frac{1}{2} (\nv\cdot g)\nv .
    \end{align*}
    It then follows that
    \begin{equation}
        \label{eq_ks_principal}
        \begin{aligned}
            - k_0 \Scal \Qcal_-[\nabla \cdot u^f|_+]+  k_0 \Scal \Qcal_+[\nabla \cdot u^f|_-]
            &=-2k_0 \Scal \big[ \nv \Jcal_\mathrm{div}[g]+ \Jcal_\mathrm{div}^* [\nv\cdot g] \big].
        \end{aligned}
    \end{equation}
    By the definition of $\Jcal_\mathrm{div}$ and $\Jcal_\mathrm{div}^*$, we have
    \begin{align*}
        &\nv_x \Jcal_\mathrm{div}[g](x) + \Jcal_\mathrm{div}^* [\nv\cdot g](x) \\
        &=\frac{1}{4\pi}\int_{\p\GO}\frac{\nv_x ((x-y)\cdot g(y))-(x-y)(\nv_y \cdot g(y))}{|x-y|^3}\, \df\sigma (y).
    \end{align*}
    One can see from the formula \eqref{eq_k1} for the operator $\Kcal_1$ that
    \[
        \Kcal_1^*[g](x)=\frac{1}{4\pi}\int_{\p\GO}\frac{(x-y)(\nv_x\cdot g(y))-\nv_x ((x-y)\cdot g(y))}{|x-y|^3}\, \df \sigma (y).
    \]
    Thus, we obtain
    \begin{equation}
        \label{eq_ks_modulo_cpt}
        \nv_x \Jcal_\mathrm{div}[g] + \Jcal_\mathrm{div}^* [\nv \cdot g]
        = -\Kcal_1^*[g] +\Tcal^\prime [g],
    \end{equation}
    where
    \[
        \Tcal^\prime [g](x)=\frac{1}{4\pi}\int_{\p\GO}\frac{(x-y)((\nv_x-\nv_y) \cdot g(y))}{|x-y|^3}\, \df \sigma (y).
    \]
    Therefore, we have
    \begin{equation}
        \label{eq_k1s_modulo_cpt}
        - k_0 \Scal \Qcal_-[\nabla \cdot u^f|_+]+ k_0 \Scal \Qcal_+[\nabla \cdot u^f|_-]=2k_0 \Scal\Kcal_1^* \Scal^{-1} [f]-2k_0 \Scal \Tcal^\prime \Scal^{-1}[f].
    \end{equation}

If $\p\GO$ is $C^{1,\Ga}$ for some $\Ga>0$, then
    \[
    |\nv_x-\nv_y| \le C |x-y|^\Ga
    \]
    for some constant $C$, and hence $\Tcal^\prime$ is a compact operator on $(\Hsb^3)^*$.
    We infer from \eqref{eq_np_decomposition} that $\Kcal^*$ is decomposed as
    \[
        \Kcal^*=2k_0 (\Kcal_1^*+\Kcal_2^*)+3(1-2k_0)\Kcal_3^*,
    \]
    where $\Kcal_2^*$ and $\Kcal_3^*$ are compact on $(\Hsb^3)^*$. Thus we have
    \[
        2k_0 \Scal\Kcal_1^*\Scal^{-1}= \Scal\Kcal^*\Scal^{-1}
    \]
    modulo a compact operator on $\Hsb^3$. We employ Plemelj's symmetrization principle to obtain
    \[
        \Scal\Kcal^*\Scal^{-1} = \Kcal.
    \]
    Hence, by \eqref{eq_k1s_modulo_cpt}, we obtain
    \[
        -k_0 \Scal \Qcal_-[\nabla \cdot u^f|_+]+k_0 \Scal \Qcal_+[\nabla \cdot u^f|_-]=\Kcal [f]+\Tcal [f]
    \]
    with some compact operator $\Tcal$ on $\Hsb^3$, as desired.
\end{proof}

Now we are ready to prove Theorem \ref{theo_free_np_3D}.
\begin{proof}[Proof of Theorem \ref{theo_free_np_3D}]
    By Lemma \ref{lemm_np23_compact}, $\Kcal + k_0 I = -2k_0 \left(\Kcal^\divf- 1/2I\right) + \text{compact}$ on $\Hsb^3$. Since $\Hdrf^-=\Ker \left( \Kcal^\divf- 1/2I\right)$ (Theorem \ref{theo_div_rot_free_kd}), we infer that $\Kcal + k_0 I$ is compact on $\Hdrf^-$. In the same way, one can show that $\Kcal - k_0 I$ is compact on $\Hdrf^+$.
    The compactness of $\Kcal$ on $\Hdf$ is an immediate consequence of Proposition \ref{lemm_np_modulo_cpt} since $\nabla \cdot u^f|_-= \nabla \cdot u^f|_+ =0$ if $f \in \Hdf$.
\end{proof}

\section{Properties of subspaces \texorpdfstring{$\Hdrf^-$}{H\_\^{\{div,rot\} } -}, \texorpdfstring{$\Hdrf^+$}{H\_\^{\{div,rot\} } +} and \texorpdfstring{$\Hdf$}{H\_\{div\}}}\label{sec:sub}

We first obtain the following theorem which asserts that all the subspaces appearing in the decomposition in Theorem \ref{theo_decomposition_div_free} are of infinite dimensions. This theorem plays a crucial role in proving Theorem \ref{thm_eigenvalue3D} in the next section.

\begin{theo}
    \label{theo_infinite_dim_c1a}
    If $\GO$ be a bounded Lipschitz domain in $\Rbb^3$, then $\Hdrf^-$, $\Hdrf^+$ and $\Hdf$ are infinite-dimensional.
\end{theo}

\begin{proof}
    As before, $P_k$ is the space of all homogeneous harmonic polynomial of degree $k$ and $u^*(x)$ stands for the Kelvin transform of $u$ as defined in \eqref{eq_inversion}.

For each $k$ let $\Vcal_k^-:= \{ p\in P_k^3 \mid \nabla\cdot p=0, \, \nabla\times p =0\}$. Then $\Vcal_k^-$ is isomorphic to $\Wcal_k^-$ defined in \eqref{eq_wcaldef}, and hence $\dim \Vcal_k^- =2k+3$ by Corollary \ref{coro_dim}. If $p \in \oplus_{k=0}^\infty \Vcal_k^-$, then $p|_{\p\GO} \in \Hdrf^-$. The mapping $p \in \oplus_{k=0}^\infty \Vcal_k^- \mapsto p|_{\p\GO} \in \Hdrf^-$ is linear and it is injective because of the maximum principle and homogeneity. Thus, $\dim \Hdrf^- \ge \sum_{k=0}^\infty (2k+3)=\infty$.

To prove that $\dim \Hdrf^+=\infty$, we assume $0\in \GO$ without loss of generality, and define $\Vcal_k^+:=\{ p\in P_k^3 \mid \nabla\cdot p^*=0, \, \nabla\times p^* =0\}$ for $k \ge 1$. Then, $\Vcal_k^+$ is isomorphic to $\Wcal_k^+$, and hence $\dim \Vcal_k^+ =2k-1$. By considering the linear mapping $p \in \oplus_{k=1}^\infty \Vcal_k^+ \mapsto p^*|_{\p\GO} \in \Hdrf^+$, we infer $\dim \Hdrf^+ \ge \sum_{k=1}^\infty (2k-1)=\infty$.

To prove that $\Hdf$ is infinite-dimensional, we first show that if
$u$ is a $C^1$ function in a neighborhood of $\overline{\GO}$, then the following formula holds:
    \beq\label{eq:111}
        \nabla\times \Scal_0[un]|_+=\nabla\times \Scal_0[un]|_-=\Scal_0[\nabla u\times n]|_{\partial\Omega},
    \eeq
where $\Scal_0$ is the single layer potential with respect to the Laplacian given in \eqref{eq_lapla_single}. In fact,
if $x\in \Rbb^3 \setminus \overline{\Omega}$, then we have
    \begin{align*}
        \nabla\times \Scal_0[u \nv](x)
        &=\frac{1}{4\pi}\int_{\partial\Omega}\frac{(x-y)\times \nv_y}{|x-y|^3}u (y)\, \df \sigma (y) \\
        &=-\frac{1}{4\pi}\int_\Omega \nabla_y \times\left(\frac{x-y}{|x-y|^3}u (y)\right)\, \df y \\
        &=-\frac{1}{4\pi}\int_\Omega \nabla u(y)\times \frac{x-y}{|x-y|^3}\, \df y \\
        &=-\frac{1}{4\pi}\int_\Omega \nabla u(y)\times \nabla_y\frac{1}{|x-y|}\, \df y \\
        &=-\frac{1}{4\pi}\int_{\partial\Omega} \frac{\nabla u(y)\times \nv_y}{|x-y|}\, \df \sigma (y) \\
        &=\Scal_0[\nabla u \times \nv](x).
    \end{align*}
    If $x\in \Omega$, we apply the same argument to $\GO \setminus B_\Gve(x)$, where $B_\varepsilon (x)\subset \Rbb^3$ is the ball with the center at $x$ and the radius $\varepsilon$, and send $\Gve$ to $0$ to prove that
    \[
    \nabla\times \Scal_0[u \nv](x)= \Scal_0[\nabla u \times \nv](x).
    \]
    In the course of the proof, we use the fact that $(x-y)\times \nv_y=0$ if $y \in \p B_\Gve(x)$. This proves \eqref{eq:111}.

    The formula \eqref{eq:111} implies
    \beq\label{eq:222}
    \Scal_0[\nabla u\times n]|_{\partial\Omega} \in \Hdf.
    \eeq
Let $\Pcal$ be the space of all harmonic polynomials. Using \eqref{eq:222}, we define a linear transform $\Tcal_0: \Pcal\to \Hdf$ by
    \[
        \Tcal_0 [p]:=\Scal_0[\nabla p \times \nv]|_{\partial\Omega}.
    \]

    If $p \in \Pcal$ and $\Tcal_0 [p] = 0$, then we infer from \eqref{eq:111} that
    \[
       \nabla \times  \Scal_0[p\nv] = 0 \quad\text{in } \Rbb^3 \setminus \p\GO.
    \]
    Therefore we have
    \beq
    \nabla \left(\nabla \cdot \Scal_0[p \nv]\right) = \nabla \times \left(\nabla \times  \Scal_0[p\nv]\right) + \GD \Scal_0[p \nv] = 0 \quad\text{in } \Rbb^3\setminus \p\GO.
    \eeq
    Note that $\nabla \cdot \Scal_0[p \nv] = -\Dcal_0[p]$ where $\Dcal_0$ is the double layer potential with respect to the Laplace operator. Thus $\Dcal_0[p]$ is constant in each connected component of $\Rbb^3\setminus \p \GO$.
    Therefore we have from \eqref{double_jump_laplacian} that $p = \Dcal_0[p] |_- - \Dcal_0[p] |_+$ is a piecewise constant function on $\p\GO$, and hence $p$ is constant in $\Rbb^3$.
    This implies $\dim \Ker \Tcal_0 \le 3$ and hence
    \[
        \dim \Ran \Tcal_0=\dim (\Pcal/\Ker \Tcal_0)=\infty.
    \]
    This proves, in particular, that $\dim \Hdf = \infty$.
\end{proof}

In the rest part of this section, we assume that the boundary $C^{1, \alpha}$ for some $\alpha>1/2$. We now prove the following theorem which shows that the decomposition in Theorem \ref{theo_decomposition_div_free} is `almost' direct sum in the sense that intersections of two subspaces appearing in the decomposition are finite-dimensional.

\begin{theo}\label{theo_finite_dim}
     If $\p\GO$ is $C^{1, \Ga}$ for some $\Ga>1/2$, then
     \beq
     \dim (\Hdf \cap \Hdrf^\pm) < \infty.
     \eeq
\end{theo}

\begin{proof}
    Since $\Kcal$ is compact on $\Hdf$ by Theorem \ref{theo_free_np_3D}, $\Kcal + k_0 I \equiv k_0 I$ modulo compact operators on $\Hdf \cap \Hdrf^-$. Since $\Kcal + k_0 I$ is compact on $\Hdf \cap \Hdrf^-$ by Theorem \ref{theo_free_np_3D} again, we infer that $k_0 I$ is compact on $\Hdf \cap \Hdrf^-$. Thus, $\Hdf \cap \Hdrf^-$ is of finite dimensions. It can be proved similarly that $\Hdf \cap \Hdrf^+$ is of finite dimensions.
\end{proof}

Since
\[
    \Hdf \cap (\Hdrf^-+\Hdrf^+)=(\Hdf \cap \Hdrf^-)\oplus (\Hdf \cap \Hdrf^+),
\]
we see immediately from Theorem \ref{theo_finite_dim} that the inclusion $(\Hdrf^-+\Hdrf^+)^\perp \subset \Hdf$ in \eqref{eq_perpsub} is of finite codimension, as stated in the following corollary.

\begin{coro}\label{coro_finite_codim}
If $\p\GO$ is $C^{1, \Ga}$ for some $\Ga>1/2$, then
\beq
\dim \Hdf/(\Hdrf^-+\Hdrf^+)^\perp < \infty.
\eeq
\end{coro}

We now investigate the spectral structure of the div-free eNP operator.

\begin{theo}
    \label{theo_isolated_c1a}
    Assume that $\p\GO$ is $C^{1, \Ga}$ for some $\alpha>1/2$.
    \begin{itemize}
    \item[{\rm (i)}] The spectrum $\sigma (\Kcal^\divf)$ of $\Kcal^\divf$ on $\Hcal^3$ consists of $-1/2$, $1/2$, $0$ and a sequence of eigenvalues $ \lambda_j \in (-1/2, 1/2)$ of finite multiplicities converging to $0$, where $\pm 1/2$ are eigenvalues of infinite multiplicities.
    \item[{\rm (ii)}] The spaces $\rmop{Ran}(\Kcal^\divf\pm 1/2I)$ are closed in $\Hsb^3$. In particular, we have
    \begin{equation}
        \label{eq_perp_ran}
        (\Hdrf^-+\Hdrf^+)^\perp=\Ran \left( \Kcal^\divf+\frac{1}{2}I\right)\left( \Kcal^\divf -\frac{1}{2}I\right).
    \end{equation}
    \end{itemize}
\end{theo}

\begin{proof}
    Set $\Xcal:=(\Hdrf^+ + \Hdrf^-)^\perp$ for convenience. Since $\Hdrf^\pm$ and $\Xcal$ are invariant under $\Kcal^\divf$ by Theorem \ref{theo_div_rot_free_kd} and Theorem \ref{theo_div_free_invariant}, we have
    \begin{align*}
    \Gs(\Kcal^\divf) &= \Gs(\Kcal^\divf, \Hdrf^+) \cup \Gs(\Kcal^\divf, \Hdrf^-) \cup \Gs(\Kcal^\divf, \Xcal) \\
    &= \left\{-\frac{1}{2}, \frac{1}{2}\right\} \cup \Gs(\Kcal^\divf, \Xcal),
    \end{align*}
    where $\pm 1/2$ are eigenvalues of infinite multiplicities since $\dim \Hdrf^\pm=\infty$ (Theorem \ref{theo_infinite_dim_c1a}). Since $\Kcal$ is compact on $\Xcal$ (Theorem \ref{theo_free_np_3D}), so is $\Kcal^\divf$ by Lemma \ref{lemm_np23_compact}. Thus $\Gs(\Kcal^\divf, \Xcal)$ consists of $0$ and eigenvalues of finite multiplicities converging to $0$.  Since $\sigma (\Kcal^\divf)\subset [-1/2, 1/2]$ (Corollary \ref{coro_spec_rad_kd}), the proof is complete.

    Since $\Gs(\Kcal^\divf, \Hdrf^- \oplus \Xcal)$ consists of $1/2$ and eigenvalues in $(-1/2, 1/2)$ of finite multiplicities converging to $0$, $\Kcal^\divf + 1/2I$ is invertible on $\Hdrf^- \oplus \Xcal$, and hence $\rmop{Ran}(\Kcal^\divf + 1/2I) = \Hdrf^- \oplus \Xcal$, in particular, $\rmop{Ran}(\Kcal^\divf + 1/2I)$ is closed in $\Hsb^3$. It can be proved in the same way that we have $\rmop{Ran}(\Kcal^\divf - 1/2I) = \Hdrf^+ \oplus \Xcal$. Since $\Kcal^\divf \pm 1/2I$ is invertible on $\Xcal$, we have $\Xcal= (\Kcal^\divf + 1/2I)(\Kcal^\divf - 1/2I)(\Xcal)$, and \eqref{eq_perp_ran} follows.
\end{proof}

\section{Proof of Theorem \ref{theo_codim_betti}}\label{sec:topology}

In this section, we prove Theorem \ref{theo_codim_betti}. We prove the inequality
\beq\label{000}
\dim \Hdf /(\Hdrf^-+\Hdrf^+)^\perp \le b_1(\p\GO)
\eeq
under the assumption that $(\mu, k_0)=(1, 1/2)$, or equivalently, $(\lambda, \mu)=(-1, 1)$. It amounts to saying that the inner product on $\Hcal^3$ defined in \eqref{eq_inner_product} is given by
\beq\label{eq_inner_product2}
    \jbracket{f, g}_*:=-\jbracket{\Scal_0^{-1}[f], g}_{\p\GO}
\eeq
where $\Scal_0$ is the single layer potential for the Laplace operator. In other words, $(\Hdrf^-+\Hdrf^+)^\perp$ in \eqref{000} is with respect to this inner product. Once we prove \eqref{000} with this inner product, then \eqref{000} holds for any pair of Lam\'e parameters. It is because the identity
\[
    \dim \Hdf /(\Hdrf^-+\Hdrf^+)^\perp = \dim \Hdf \cap (\Hdrf^-+\Hdrf^+)
\]
holds for any pair of Lam\'e parameters and the right hand side is independent of Lam\'e parameters.

For the proof of Theorem \ref{theo_codim_betti}, we introduce subspaces $\Mcal^\pm$ of $\Hdrf^\pm$ defined by
\begin{equation}
    \label{eq_mpm}
    \begin{aligned}
        \Mcal^-&:=\{ f\in \Hdrf^- \mid
            \exists \varphi\in C^\infty (\Omega) \text{ such that $v^f_-=\nabla \varphi$ in $\GO$} \}, \\
        \Mcal^+&:=\{ f\in \Hdrf^+ \mid
            \exists \varphi\in C^\infty (\Rbb^3\setminus\overline{\Omega}) \text{ such that  $v^f_+=\nabla \varphi$ in $\Rbb^3\setminus\overline{\Omega}$} \}.
    \end{aligned}
\end{equation}
Here $v^f = \Scal_0 [\Scal_0^{-1}[f]]$. We prove the following proposition.

\begin{prop}\label{prop_suff}
    Let $\Omega\subset \Rbb^3$ be a $C^{1, \alpha}$ domain for some $\alpha>1/2$. Then we have
    \[
        (\Hdrf^\mp)^\perp\subset \Hdf^\pm \subset (\Mcal^\mp)^\perp .
    \]
    In particular, we have
    \begin{equation}
        \label{eq_Hdf_above_below}
        (\Hdrf^-+\Hdrf^+)^\perp \subset \Hdf \subset (\Mcal^-+\Mcal^+)^\perp.
    \end{equation}
\end{prop}

That $(\Hdrf^\mp)^\perp\subset \Hdf^\pm$ is already proved in Lemma \ref{lemm_drf_df_pm} (it holds even if the inner product is given by \eqref{eq_inner_product2}). To prove $\Hdf^\pm \subset (\Mcal^\mp)^\perp$, we prepare three lemmas.

\begin{lemm}\label{lemm_asymptotic}
    Let $K\subset \Rbb^3$ be a compact set. If a $C^1$ function $\varphi$ on $\Rbb^3\setminus K$ satisfies $|\nabla\varphi(x)|=O(|x|^{-2})$ as $|x|\to \infty$, then the limit $\varphi(\infty):=\lim_{|x|\to \infty}\varphi (x)\in \Rbb$ exists and $|\varphi (x)-\varphi (\infty)|=O(|x|^{-1})$ as $|x|\to \infty$.
\end{lemm}

\begin{proof}
    Take a sufficiently large $R>0$ such that
    \[
        B_R:=\{ x\in \Rbb^3 \mid |x|<R\} \supset K\cup \{0\}.
    \]
    Let $F_R: B_R\setminus \{0\}\to \Rbb^3\setminus \overline{B_R}$ be the inversion with respect to $|x|=R$, that is, $F_R(x):=R^2x/|x|^2$. Then, for $x\in B_R\setminus \{0\}$, we have
    \[
        \nabla (\varphi \circ F_R)(x)
        =R^2\left(\frac{1}{|x|^2}(\nabla \varphi)(F_R(x))-\frac{2x}{|x|^4}x\cdot \nabla \varphi (F_R(x))\right).
    \]
    Since $|\nabla\varphi (x)|=O(|x|^{-2})$ as $|x|\to \infty$, we have
    \[
        |\nabla (\varphi\circ F_R)(x)|\leq C \quad (x\in B_R\setminus \{0\})
    \]
    for some constant $C>0$ independent of $x\in B_R\setminus \{0\}$. Hence we have
    \begin{equation}
        \label{eq_phif_origin}
        |(\varphi\circ F_R)(x)-(\varphi\circ F_R)(y)|
        \leq C|x-y|
    \end{equation}
    for all $x$, $y\in B_R\setminus \{0\}$. Therefore the limit $\varphi(\infty):=\lim_{x\to 0}(\varphi \circ F_R)(x)$ exists. Equivalently, we have $\lim_{|x|\to \infty} \varphi (x)=\varphi (\infty)$. Moreover, by taking $y\to 0$ in \eqref{eq_phif_origin}, we have
    \[
        |(\varphi\circ F_R)(x)-\varphi (\infty)|
        \leq C|x|
    \]
    for all $x\in B_R\setminus \{ 0\}$, and hence
    \[
        |\varphi (x)-\varphi (\infty)|\leq C|x|^{-1}
    \]
    for all $x \in \Rbb^3\setminus \overline{B_R}$.
    The proof is completed.
\end{proof}

\begin{lemm}\label{lemm_adjoint_qp}
    Let $\partial\Omega$ is $C^{1, \alpha}$ for some $\alpha>1/2$. The following relation holds for all $g\in (\Hsb^3)^*$ and $\psi\in \Hsb$:
    \[
        \jbracket{g, \nabla \Scal_0[\psi]|_\pm}_{\partial\Omega}=-\frac{\mu}{2k_0}\jbracket{\nabla\cdot \Scal [g]|_\mp, \psi}_{\partial\Omega}.
    \]
\end{lemm}

\begin{proof}
    We first note that $\Jcal_\mathrm{div}^*[\psi] \in \Hcal^3$ and $\Jcal_\mathrm{div}[g] \in (\Hsb^3)^*$ by Lemma \ref{lemm:Jdiv}. We invoke the jump formulas \eqref{eq_jump_j_div} and \eqref{eq_jump_ds0} and obtain
    \begin{align*}
        \jbracket{g, \nabla \Scal_0[\psi]|_\pm}_{\partial\Omega}
        &=\jbracket{g, -\Jcal_\mathrm{div}^*[\psi]\pm \frac{1}{2}\psi n}_{\partial\Omega}
        =-\jbracket{\Jcal_\mathrm{div}[g]\mp \frac{1}{2}g\cdot \nv, \psi}_{\partial\Omega} \\
        &=-\frac{\mu}{2k_0}\jbracket{\nabla\cdot \Scal [g]|_\mp, \psi}_{\partial\Omega},
    \end{align*}
    as desired.
\end{proof}

\begin{lemm}\label{lemm_mplus_decay}
    Assume $\partial\Omega$ is $C^{1, \alpha}$ for some $\alpha>1/2$. Then we have the following.
    \begin{enumerate}[label=(\arabic*)]
        \item \label{enum_inv_quad}If $f\in \Hdrf^+$, then $|v^f_+(x)|=O(|x|^{-2})$ as $|x|\to \infty$.
        \item \label{enum_inv_linear}If $f\in \Mcal^+$, then there exists a function $\varphi\in C^\infty (\Rbb^3\setminus \overline{\Omega})$ such that $v^f_+=\nabla \varphi$ in $\Rbb^3\setminus \overline{\Omega}$ and $|\varphi (x)|=O(|x|^{-1})$ as $|x|\to \infty$. Moreover, if $B\subset \Rbb^3$ is an open ball which contains $\overline{\Omega}$, then $\varphi \in H^1(B\setminus \overline{\Omega})$.
        \item \label{enum_scp_regular} If $\varphi\in C^\infty (\Rbb^3\setminus \overline{\Omega})$ is as in \ref{enum_inv_linear}, then there exists a function $\psi\in \Hsb$ such that $\varphi =\Scal_0[\psi]$ in $\Rbb^3\setminus \overline{\Omega}$.
    \end{enumerate}
\end{lemm}

\begin{proof}
    \ref{enum_inv_quad} Let $\Ycal$ be the space of constant vector-valued functions. Since $\Ycal \subset \Hdrf^-$, we have $\Hdrf^+\subset (\Hdrf^-)^\perp \subset \Ycal^\perp$. Thus, if $f \in \Hdrf^+$, then $f \in \Ycal^\perp$, namely, \[
    \langle \Scal_0^{-1}[f], \psi \rangle_{\p\GO} =0
    \]
    for any constant function $\psi$, in particular, $\int_{\p\GO} \Scal_0^{-1}[f] \, \df \Gs=0$. Since $v_+^f =\Scal_0 [\Scal_0^{-1}[f]]$, we have
    \begin{align*}
    v_+^f(x)&= \int_{\p\GO} \Gamma_0(x-y) \Scal_0^{-1}[f](y) \, \df \Gs(y) \\
    &= \int_{\p\GO} \left[\Gamma_0(x-y)- \Gamma_0(x-y_0)\right] \Scal_0^{-1}[f](y) \, \df \Gs(y)
    \end{align*}
    for some $y_0$, where $\Gamma_0$ is the fundamental solution to the Laplacian. Thus, $|v^f_+(x)|=O(|x|^{-2})$ as $|x|\to \infty$.

    \freespace\noindent\ref{enum_inv_linear} Since $|\nabla\varphi(x)|=O(|x|^{-2})$ as $|x|\to \infty$ by \ref{enum_inv_quad}, we can apply Lemma \ref{lemm_asymptotic} to infer that the limit $\varphi (\infty)=\lim_{|x|\to \infty} \varphi(x)$ exists. Replacing $\varphi (x)$ with $\varphi(x)-\varphi(\infty)$ if necessary, we can assume that $\varphi(\infty)=0$.

    Let $B$ be a ball containing $\Omega$. Since $\nabla\varphi=v_+^f\in H^1(B\setminus \overline{\Omega})^3\subset L^2(B\setminus \overline{\Omega})^3$, we have $\varphi \in L^2(B\setminus \overline{\Omega})$. Thus we have $\varphi\in H^1(B\setminus \overline{\Omega})$.

    \freespace\noindent\ref{enum_scp_regular} Since $\varphi\in H^1(B\setminus \overline{\Omega})$, the boundary value $\varphi|_+$ exists on $\p\GO$ and belongs to $\Hsb$. Moreover, $\varphi$ is harmonic in $\Rbb^3\setminus \overline{\Omega}$ and $|\varphi (x)|=O(|x|^{-1})$ as $|x|\to \infty$. Thus, if we set $\psi:=\Scal_0^{-1}[\varphi|_+]\in \Hsb^*$, then we have $\varphi=\Scal_0[\psi]$ in $\Rbb^3\setminus \overline{\Omega}$.

    We claim $\psi\in \Hsb$. In fact, by the jump formula for $\Scal_0$ in terms of the NP operator $\Kcal_0^*$ for the Laplacian, we have
    \begin{equation}
        \label{eq_nf_nps}
        \nv\cdot \varphi|_+ =(\nv\cdot \nabla)\Scal_0[\psi]|_+=\Kcal_0^*[\psi]+\frac{1}{2}\psi.
    \end{equation}
    Since $\partial\Omega$ is $C^{1, \alpha}$ for some $\alpha>1/2$, we have $\nv\cdot \varphi|_+\in H^{1/2}(\partial\Omega)$ by Lemma \ref{lemm:multi}. Thus we have $(\Kcal_0^*+1/2I)[\psi]\in H^{1/2}(\partial\Omega)$. As we remarked in the proof of Lemma \ref{lemm_np23_compact}, the integral kernel of the operator $\Kcal_0$, which is same as $\Kcal_2$ in \eqref{eq_np_decomposition}, satisfies the conditions \eqref{eq_singularity_0} and \eqref{eq_singularity_1} with $\beta=\alpha$ and $\gamma=1$ and $\Kcal_0[1]=1/2$, we can apply Corollary \ref{coro_regularity} \ref{enum_regularity_T1_adj} to obtain $\psi\in L^2 (\partial\Omega)$.

    One can also prove that the integral kernel of $\Kcal_0^*$ satisfies the conditions \eqref{eq_singularity_0} and \eqref{eq_singularity_1} with $\beta=1$ and $\gamma=\alpha$. Then, we apply Corollary \ref{coro_regularity} \ref{enum_regularity} to obtain $\psi\in H^{1/2}(\partial \Omega)$, noting $\alpha>1/2$.
\end{proof}

Now we prove Proposition \ref{prop_suff}.

\begin{proof}[Proof of Proposition \ref{prop_suff}]
    We already proved $(\Hdrf^\pm)^\perp\subset \Hdf^\mp$ in Lemma \ref{lemm_drf_df_pm} as mentioned before.

    Let $f\in \Hdf^-$ and set $g:=\Scal^{-1}[f]\in (\Hsb^3)^*$. Then $\nabla\cdot \Scal[g]=0$ in $\Omega$. Let $h\in\Mcal^+$. Then there exists a function $\varphi\in C^\infty (\Rbb^3\setminus \overline{\Omega})$ such that $v^h_+=\nabla \varphi$ in $\Rbb^3\setminus \overline{\Omega}$. By Lemma \ref{lemm_mplus_decay}, there exists a function $\psi\in \Hsb$ such that $\varphi=\Scal_0[\psi]$ in $\Rbb^3\setminus \overline{\Omega}$. Then, since $\psi\in \Hsb$ and $\nabla\Scal_0[\psi]|_+=h\in \Hsb^3$, we can apply Lemma \ref{lemm_adjoint_qp} and obtain
    \[
        \jbracket{f, h}_*=-\jbracket{g, h}_{\partial\Omega} = -\jbracket{g, \nabla \Scal_0[\psi]|_+}_{\partial\Omega}=\frac{\mu}{2k_0}\jbracket{\nabla\cdot \Scal [g]|_-, \psi}_{\partial\Omega}=0.
    \]
    Hence we have $\Hdf^-\subset (\Mcal^+)^\perp$.

    Similarly, one can prove $\Hdf^+\subset (\Mcal^-)^\perp$.
\end{proof}

Next we estimate the dimension of $\Hdrf^\pm/\Mcal^\pm$ by using the de Rham theory (for the de Rham theory, see \cite{Bott-Tu82, Bredon93, Lee13} for example).

\begin{prop}\label{prop_quotient_betti}
    Assume that $\partial\Omega$ is $C^{1, \alpha}$ for some $\alpha>1/2$.
    \begin{enumerate}[label=(\arabic*)]
        \item \label{enum_hm_inner}$\dim (\Hdrf^-/\Mcal^-)\leq b_1(\Omega)$.
        \item \label{enum_hm_exterior}$\dim (\Hdrf^+/\Mcal^+)\leq b_1(\Rbb^3\setminus \overline{\Omega})$.
    \end{enumerate}
\end{prop}

\begin{proof}
    We denote the first singular homology group of a topological space $X$ with the real coefficient by $H^\mathrm{sing}_1(X)$. We remark that $H^\mathrm{sing}_1(X)$ is a real vector space and the first Betti number of $X$ is defined as the dimension of the first homology group, that is, $b_1(X):=\dim H^\mathrm{sing}_1(X)$. Moreover, we denote the de Rham cohomology group of an open subset $U\subset \Rbb^3$ by $H^1_\mathrm{dR}(U)$, which is defined as
    \begin{equation}
        \label{eq_derham}
        H^1_\mathrm{dR}(U):=\frac{\Ker(\nabla\times : C^\infty (U)^3\longrightarrow C^\infty (U)^3)}{\Ran (\nabla: C^\infty (U)\longrightarrow C^\infty (U)^3)}
    \end{equation}
    in terms of the vector calculus. The following formula holds:
    \begin{equation}
        \label{eq_de_rham_betti}
        \dim H^1_\mathrm{dR}(U)=b_1(U).
    \end{equation}
    In fact, this follows from the de Rham theorem (see \cite[Theorem 18.14]{Lee13}) and (a simple case of) universal coefficient theorem (see \cite[p.198 and p.199]{Hatcher02}):
    \[
        H^1_\mathrm{dR}(U)\simeq H_\mathrm{sing}^1(U)\simeq H^\mathrm{sing}_1(U)^*
    \]
    for any open sets $U\subset \Rbb^3$ where $\simeq$ means a vector space isomorphism and the vector space $H_\mathrm{sing}^1(U)$ stands for the first singular cohomology group with real coefficient and $H^\mathrm{sing}_1(U)^*$ stands for the dual space of $H^\mathrm{sing}_1(U)$.

    In the following, we only prove the assertion \ref{enum_hm_exterior}. The assertion \ref{enum_hm_inner} is proved similarly.
    We define a linear mapping
    \[
        A^+: \Hdrf^+ \longrightarrow H^1_\mathrm{dR}(\Rbb^3\setminus\overline{\Omega}), \quad
        A^+(f):=[v^f_+]\in H^1_\mathrm{dR}(\Rbb^3\setminus\overline{\Omega})
    \]
    where $[v^f_+]$ stands for the de Rham cohomology class to which the function $v^f_+\in \Rbb^3\setminus \overline{\Omega}$ belongs, that is, the image of $v^f_+$ by the natural projection
    \[
        \Ker (\nabla\times: C^\infty (\Rbb^3\setminus \overline{\Omega})\to \Rbb^3\setminus \overline{\Omega}) \longrightarrow H_\mathrm{dR}^1 (\Rbb^3\setminus \overline{\Omega}).
    \]

    Since $f\in \Hdrf^+$, $\nabla\times v_+^f=0$ in $\Rbb^3\setminus\overline{\Omega}$ and hence the operator $A^+$ is well-defined.

    By definition of the de Rham cohomology, we can easily show $\Ker A^+=\Mcal^+$, which implies that $A^+$ induces the injection $\overline{A^+}: \Hdrf^+/\Mcal^+\to H^1_\mathrm{dR}(\Rbb^3\setminus\overline{\Omega})$ defined by $\overline{A^+}([f]):=A^+(f)$.

    Hence we have $\dim (\Hdrf^+/\Mcal^+)\leq \dim H^1_\mathrm{dR}(\Rbb^3\setminus\overline{\Omega})$ by the injectivity of $\overline{A^+}: \Hdrf^+/\Mcal^+\to H^1_\mathrm{dR}(\Rbb^3\setminus\overline{\Omega})$. This and \eqref{eq_de_rham_betti} imply
    \[
        \dim (\Hdrf^+/\Mcal^+)\leq \dim H^1_\mathrm{dR}(\Rbb^3\setminus\overline{\Omega})=b_1(\Rbb^3\setminus\overline{\Omega}).
    \]
    This completes the proof.
\end{proof}

Now we are ready to prove Theorem \ref{theo_codim_betti}.

\begin{proof}[Proof of Theorem \ref{theo_codim_betti}]
    By Proposition \ref{prop_suff}, we have
    \begin{align*}
        &\dim \Hdf/(\Hdrf^-+\Hdrf^+)^\perp \\
        &\leq \dim (\Mcal^-+\Mcal^+)^\perp/(\Hdrf^-+\Hdrf^+)^\perp \\
        &=\dim (\Hdrf^-+\Hdrf^+)/(\overline{\Mcal^-+\Mcal^+}) \\
        &\leq \dim (\Hdrf^-+\Hdrf^+)/(\Mcal^-+\Mcal^+).
    \end{align*}
    Since $\Mcal^\pm\subset \Hdrf^\pm$ and $\Hdrf^- \cap \Hdrf^+=\{0\}$, we have
    \[
        (\Hdrf^-+\Hdrf^+)/(\Mcal^-+\Mcal^+)\simeq (\Hdrf^-/\Mcal^-)\oplus (\Hdrf^+/\Mcal^+).
    \]
    Hence we have
    \begin{align*}
        &\dim (\Hdrf^-+\Hdrf^+)/(\Mcal^-+\Mcal^+) \\
        &=\dim (\Hdrf^-/\Mcal^-)+\dim (\Hdrf^+/\Mcal^+) \\
        &\leq b_1(\Omega)+b_1(\Rbb^3\setminus \overline{\Omega})
    \end{align*}
    by Proposition \ref{prop_quotient_betti}. Since the equality
    \begin{equation}
        \label{eq_betti_sum}
        b_1(\Omega)+b_1(\Rbb^3\setminus \overline{\Omega})=b_1(\partial\Omega)
    \end{equation}
    holds, we obtain the desired inequality \eqref{eq_codim_betti}.

    The identity \eqref{eq_betti_sum} is probably known. However, since we could not find a proper reference, we describe a proof of \eqref{eq_betti_sum} for readers' sake. There exists a small interval $I=(-\varepsilon, \varepsilon)$ such that the mapping
    \[
        (x, t)\in \partial\Omega \times I \longmapsto  x+t\nv_x\in N:=\{ x+t\nv_x \mid x\in \partial\Omega, \, t\in I\}
    \]
    is a homeomorphism. We set $U:=\Omega\cup N$ and $V:=(\Rbb^3\setminus \Omega)\cup N$. Then $U$ and $V$ are homeomorphic to $\Omega$ and $\Rbb^3\setminus \overline{\Omega}$ respectively. Since $\{ U, V\}$ is an open covering of $\Rbb^3$, we invoke the Mayer-Vietoris exact sequence (see \cite[p.149]{Hatcher02} for example)
    \[
        H^\mathrm{sing}_2(\Rbb^3) \longrightarrow H^\mathrm{sing}_1(U\cap V)\longrightarrow H^\mathrm{sing}_1(U)\oplus H^\mathrm{sing}_1(V) \longrightarrow H^\mathrm{sing}_1(\Rbb^3),
    \]
    where $H^\mathrm{sing}_2(\Rbb^3)$ stands for the second homology group of $\Rbb^3$. Since $\Rbb^3$ is contractible, we have $H^\mathrm{sing}_2(\Rbb^3)=\{0\}$ and $H^\mathrm{sing}_1(\Rbb^3)=\{0\}$. Moreover, since $H^\mathrm{sing}_1(U)\simeq H^\mathrm{sing}_1(\Omega)$, $H^\mathrm{sing}_1(V)\simeq H^\mathrm{sing}_1(\Rbb^3\setminus \overline{\Omega})$ and $H^\mathrm{sing}_1(U\cap V)\simeq H^\mathrm{sing}_1(\partial\Omega)$ by the homotopy invariance of the homology group, we obtain the exact sequence
    \[
        0 \longrightarrow H^\mathrm{sing}_1(\partial\Omega) \longrightarrow H^\mathrm{sing}_1(\Omega) \oplus H^\mathrm{sing}_1(\Rbb^3\setminus \overline{\Omega}) \longrightarrow 0.
    \]
    This exact sequence is equivalent to the linear isomorphism
    \begin{equation}
        \label{eq_homology_sum}
        H^\mathrm{sing}_1(\partial\Omega)\simeq H^\mathrm{sing}_1(\Omega) \oplus H^\mathrm{sing}_1(\Rbb^3\setminus \overline{\Omega}).
    \end{equation}
        Thus, we count the dimension of both spaces in \eqref{eq_homology_sum} and obtain the identity \eqref{eq_betti_sum}.
\end{proof}

\section{Proofs of Theorem \ref{thm_eigenvalue3D} and \ref{theo_enp_eigenfunctions}, and more}\label{sec:eigenspace}

In this section, we consider a bounded $C^{1, \alpha}$-domain $\Omega\subset \Rbb^3$ for some $\alpha>1/2$ to prove Theorem \ref{thm_eigenvalue3D} and \ref{theo_enp_eigenfunctions}. Furthermore, we prove Theorem \ref{theo_div_rot_asymptotic} regarding estimates on the rotation and the divergence of the solutions to the problem \eqref{eq_lame_bvp} and \eqref{eq_lame_bvp_ext}.

\begin{proof}[Proof of Theorem \ref{thm_eigenvalue3D}]
Since $\Kcal(\Kcal^2-k_0^2I)$ is compact on $\Hsb^3$ by Corollary \ref{coro_polynomially_cpt_3d}, it suffices to prove that the operators $\Kcal^2-k_0^2I$, $\Kcal(\Kcal-k_0I)$, and $\Kcal(\Kcal+k_0I)$ are not compact on $\Hsb^3$.

According to Theorem \ref{theo_free_np_3D}, $\Kcal^2-k_0^2I \equiv -k_0^2 I$ modulo a compact operator on $\Hdf$. Since $\Hdf$ is infinite-dimensional by Theorem \ref{theo_infinite_dim_c1a}, $\Kcal^2-k_0^2I$ is not compact on $\Hdf$. By Theorem \ref{theo_free_np_3D} again, $\Kcal \equiv -k_0 I$ modulo a compact operator on $\Hdrf^{-}$. Thus, $\Kcal(\Kcal-k_0I) \equiv 2k_0^2 I$ modulo a compact operator on $\Hdrf^{-}$, and hence $\Kcal(\Kcal-k_0I)$ is not compact on $\Hdrf^{-}$.
Similarly, it can be proved that $\Kcal(\Kcal+k_0I)$ is not compact on $\Hdrf^{+}$.
\end{proof}

\begin{proof}[Proof of Theorem \ref{theo_enp_eigenfunctions}]
    Since $\{ f_j\}_{j=1}^\infty$ is an orthonormal system with respect to \eqref{eq_inner_product}, it weakly converges to $0$ as $j\to \infty$. Moreover, the operator $\Kcal+2k_0 \Kcal^\divf$ is compact by Lemma \ref{lemm_np23_compact}. Thus, we have $\|\Kcal [f_j]+2k_0 \Kcal^\divf [f_j]\|_*\to 0$ as $j\to \infty$. Then, by the assumption $\Kcal [f_j]=\lambda_j f_j$, we obtain
    \begin{equation}\label{eq_enp_eigenfunctions_conv}
        \lim_{j\to \infty} \| \lambda_j f_j+2k_0 \Kcal^\divf [f_j]\|_*^2=0.
    \end{equation}
    We set $\Xcal:=\Ran (\Kcal^\divf-1/2I)(\Kcal^\divf+1/2I)=(\Hdrf^-+\Hdrf^+)^\perp$, where the second equality is due to the proof of Theorem \ref{theo_decomposition_div_free}. Since $f^\pm_j\in \Ker (\Kcal^\df\pm 1/2I)=\Hdrf^\pm$ (Theorem \ref{theo_div_rot_free_kd}) and $f^o_j, \Kcal^\divf [f^o_j]\in \Xcal$, and the decomposition \eqref{eq_H3orthogonal} is orthogonal with respect to \eqref{eq_inner_product}, we have
    \begin{align*}
        &\| \lambda_j f_j+2k_0 \Kcal^\divf [f_j]\|_*^2 \\
        &=|\lambda_j +k_0|^2\|f^-_j\|_*^2
        +|\lambda_j -k_0|^2\|f^+_j\|_*^2
        +\| \lambda_jf^o_j+2k_0 \Kcal^\divf [f^o_j]\|_*^2.
    \end{align*}
    Hence, \eqref{eq_enp_eigenfunctions_conv} implies
    \begin{equation}
        \label{eq_enp_eigenfunctions_conv_each}
        \lim_{j\to \infty} |\lambda_j \mp k_0|\|f^\pm_j\|_*
        =\lim_{j\to \infty} \| \lambda_j f^o_j+2k_0 \Kcal^\divf [f^o_j]\|_*=0.
    \end{equation}

    If $\lambda_j\to 0$, then \eqref{eq_enp_eigenfunctions_conv_each} implies
    \[
        \lim_{j\to \infty}\|f^-_j\|_*=\lim_{j\to \infty}\|f^+_j\|_*=0.
    \]
    Thus we have
    \[
        \lim_{j\to \infty}\|f_j-f^o_j\|_*^2=\lim_{j\to \infty} (\|f_j^-\|_*^2+\|f_j^+\|_*^2)=0,
    \]
    which proves \ref{enum_eigen_0}.

    For the proof of \ref{enum_eigen_+}, we assume $\lambda_j\to k_0$. Then, by \eqref{eq_enp_eigenfunctions_conv_each}, we have
    \[
        \lim_{j\to\infty} \| f^-_j\|_*=\lim_{j\to \infty} \| k_0 f^o_j+2k_0 \Kcal^\divf [f^o_j]\|_*=0.
    \]
    The proof of Theorem \ref{theo_isolated_c1a} shows that $\Kcal^\divf+1/2I: \Xcal\to \Xcal$ is invertible. Hence we have
    \[
        \lim_{j\to \infty} \|f^o_j\|_*
        =\lim_{j\to \infty} \left\|\left(\Kcal^\divf+\frac{1}{2}I\right)^{-1}\left(\Kcal^\divf+\frac{1}{2}I\right)[f^o_j]\right\|_*=0.
    \]
    Thus we have
    \[
        \lim_{j\to \infty} \|f_j-f^+_j\|_*^2=\lim_{j\to \infty} (\|f^-_j\|_*^2+\|f^o_j\|_*^2)=0,
    \]
    which proves \ref{enum_eigen_+}. The other statement \ref{enum_eigen_-} is proved similarly.
\end{proof}

The following theorem characterizes eigenfunctions of the eNP operator in terms of corresponding solutions $u^{f_j}$ to the boundary value problems \eqref{eq_lame_bvp} and \eqref{eq_lame_bvp_ext}. The estimates \eqref{eq_div_0_0}, \eqref{eq_div_rot_0_+}, and  \eqref{eq_dvirot01} below characterize the subspaces $\Hdf$, $\Hdrf^+$, and $\Hdrf^-$, respectively. In view of \eqref{eq_2.23+2.24}, \eqref{eq_rot_1_0}, \eqref{eq_div_rot_1_-}, and \eqref{eq_dvirot02} may be regarded as complementary ones.

\begin{theo}
    \label{theo_div_rot_asymptotic}
    We use the same notation as in Theorem \ref{theo_enp_eigenfunctions}.
    \begin{enumerate}[label={\rm (\roman*)}]
        \item \label{enum_div_asymptotic}If $\Gl_j \to 0$, then
        \beq
            \label{eq_div_0_0}
            \lim_{j\to \infty}\| \nabla \cdot u^{f_j}\|_{\GO}=\lim_{j\to \infty}\| \nabla \cdot u^{f_j} \|_{\overline{\GO}^c}= 0
        \eeq
        and
        \beq
            \label{eq_rot_1_0}
            \lim_{j\to\infty}\| \nabla\times u^{f_j}\|_{\GO}=\lim_{j\to \infty}\| \nabla \times u^{f_j}\|_{\overline{\GO}^c}= \frac{1}{\sqrt{2\Gm}}.
        \eeq
        \item \label{enum_div_rot_+}If $\Gl_j \to k_0$, then
        \beq
            \label{eq_div_rot_0_+}
            \lim_{j\to\infty}\| \nabla \cdot u^{f_j}\|_{\overline{\GO}^c}=\lim_{j\to\infty}\| \nabla \times  u^{f_j} \|_{\overline{\GO}^c}=0
        \eeq
        and
        \beq
            \label{eq_div_rot_1_-}
            \lim_{j\to\infty} \left( \Gm \| \nabla\times u^{f_j} \|_{\GO}^2+\frac{\Gm}{2k_0}\| \nabla \cdot u^{f_j} \|_{\GO}^2\right)= 1.
        \eeq
        \item \label{enum_div_rot_-}If $\Gl_j \to -k_0$, then
        \beq\label{eq_dvirot01}
            \lim_{j\to\infty}\| \nabla \cdot u^{f_j}\|_{\GO}=\lim_{j\to\infty}\| \nabla \times u^{f_j} \|_{\GO}= 0
        \eeq
        and
        \beq\label{eq_dvirot02}
            \lim_{j\to\infty} \left( \Gm \| \nabla\times u^{f_j} \|_{\overline{\GO}^c}^2+\frac{\Gm}{2k_0}\| \nabla \cdot u^{f_j} \|_{\overline{\GO}^c}^2\right)= 1.
        \eeq
    \end{enumerate}
\end{theo}

\begin{proof}
Let $\Tcal:=\Kcal+2k_0\Kcal^\divf$, which is compact on $\Hcal^3$ by Lemma \ref{lemm_np23_compact}. As before, since $\{f_j\}_{j=1}^\infty$ is an orthonormal system on $\Hcal^3$, we have $\|\Tcal [f_j]\|_*\to 0$ as $j\to \infty$. Thus, by $\Kcal^\divf [f_j]=-(2k_0)^{-1}(\Kcal [f_j]-\Tcal [f_j])=-(2k_0)^{-1}(\lambda_j f_j-\Tcal [f_j])$, we obtain
\begin{align*}
    \lim_{j\to \infty}\jbracket{\left(\Kcal^\divf \pm \frac{1}{2}I\right)[f_j], f_j}_*
    &=-\frac{1}{2k_0}\lim_{j\to \infty}\jbracket{\left(\lambda_j I-\Tcal\mp k_0 I\right)[f_j], f_j}_* \\
    &=\pm \frac{1}{2}-\frac{1}{2k_0}\lim_{j\to \infty} \lambda_j,
\end{align*}
if the limit $\lim_{j\to \infty}\lambda_j$ exists. Combining this equation with \eqref{eq_lame_bilinear_1_in} and \eqref{eq_lame_bilinear_1_ex}, we obtain
\begin{align}
    \frac{1}{2}+\frac{1}{2k_0}\lim_{j\to \infty} \lambda_j
    &=\lim_{j\to \infty}\left(\Gm \|\nabla\times u^{f_j}\|_{\GO}^2+\frac{\Gm}{2k_0}
    \|\nabla \cdot u^{f_j}\|_{\GO}^2\right), \label{eq_lame_bilinear_lim_in}\\
    \frac{1}{2}-\frac{1}{2k_0}\lim_{j\to\infty}\lambda_j
    &=\lim_{j\to \infty}\left(\Gm \|\nabla\times u^{f_j}\|_{\overline{\GO}^c}^2 +\frac{\Gm}{2k_0}\|\nabla \cdot u^{f_j}\|_{\overline{\GO}^c}^2\right). \label{eq_lame_bilinear_lim_ex}
\end{align}

In what follows, we prove each statement.

\freespace\noindent\ref{enum_div_asymptotic} By Theorem \ref{theo_enp_eigenfunctions} \ref{enum_eigen_0} and Lemma \ref{lemm_div_rot_npd_identity}, we have
\begin{align*}
    \frac{\Gm}{2k_0} \| \nabla \cdot u^{f_j}\|_{\GO}^2
    &=\frac{\Gm}{2k_0} \| \nabla \cdot u^{f_j}-\nabla\cdot u^{f^o_j}\|_{\GO}^2 \\
    &\leq -\jbracket{\left(\Kcal^\divf- \frac{1}{2}I\right)[f_j-f^o_j], f_j-f^o_j}_* \to 0,
\end{align*}
and similarly
\[
    \frac{\Gm}{2k_0} \| \nabla \cdot u^{f_j}\|_{\overline{\GO}^c}^2
    \leq \jbracket{\left(\Kcal^\divf+ \frac{1}{2}I\right)[f_j-f^o_j], f_j-f^o_j}_*
    \to 0
\]
as $j\to \infty$. This proves \eqref{eq_div_0_0}. Moreover, by \eqref{eq_lame_bilinear_lim_in} and \eqref{eq_lame_bilinear_lim_ex}, we obtain
\[
    \frac{1}{2}=\mu \lim_{j\to\infty} \|\nabla \times u^{f_j}\|_\Omega^2
\]
and
\[
    \frac{1}{2}=\mu \lim_{j\to \infty} \|\nabla\times u^{f_j}\|_{\overline{\Omega}^c}^2,
\]
which prove \eqref{eq_rot_1_0}.

\freespace\noindent\ref{enum_div_rot_+} If $\lambda_j\to k_0$, then \eqref{eq_lame_bilinear_lim_in} and \eqref{eq_lame_bilinear_lim_ex} immediately imply \eqref{eq_div_rot_1_-} and \eqref{eq_div_rot_0_+} respectively. The assertion \ref{enum_div_rot_-} is proved similarly to \ref{enum_div_rot_+}.
\end{proof}

\section{Higher dimensional case}\label{sec:higher}

In this section we deal with the case of dimensions higher than 3. As we will see, most results in three dimensions hold to be true in higher dimensions with slight differences in details. We emphasize that the decomposition theorem for the surface vector fields is valid in the same form in higher dimensions. We will highlight those differences and will not repeat arguments if they are the same as the three-dimensional case. The decomposition of vector fields is different in two dimensions; $\Hdf=\{0\}$ as Theorem \ref{coro_direct_sum_2d} shows.

The div-free conormal derivative in higher dimensions is most conveniently introduced if we use the terminology of the exterior calculus. We recall some  of them necessary for our purpose. For more details, see \cite{Bott-Tu82,Bredon93} for example. Let $U$ be an open subset of $\Rbb^m$ ($m\geq 2$) and $p=0, 1, 2$. We denote the $p$-th exterior power of the cotangent bundle over $U$ by $\Lambda^p$, omitting the base space $U$. We denote the space of $C^\infty$-differential $p$-forms, which are the sections of the vector bundle $\Lambda^p$, by $C^\infty (U; \Lambda^p)$. For $p=0$, we just define $C^\infty (U; \Lambda^0):=C^\infty (U)$. We identify the vector-valued function $u=(u_1, \ldots, u_m)\in C^\infty (U)^m$ with a differential 1-form
\[
    u=\sum_{j=1}^m u_j\, \df x_j\in C^\infty (U; \Lambda^1)
\]
where $(x_1, \ldots, x_m)$ is the Cartesian coordinates on $\Rbb^m$. The exterior derivative $\df: C^\infty (U; \Lambda^p)\to C^\infty (U; \Lambda^{p+1})$ ($p=0, 1$) is defined by
\[
    \df \varphi (x):=\sum_{j=1}^m \frac{\partial \varphi}{\partial x_j}\, \df x_j \in C^\infty (U; \Lambda^1)
\]
for a function $\varphi\in C^\infty (U)$ and
\[
    \df u:=\sum_{i, j=1}^m \frac{\partial u_j}{\partial x_i}\, \df x_i \wedge \df x_j \in C^\infty (U; \Lambda^2)
\]
for a differential 1-form $u\in C^\infty (\Rbb^m; \Lambda^1)$. The formal adjoint $\df^*: C^\infty (U; \Lambda^p)\to C^\infty (U; \Lambda^{p-1})$ ($p=1, 2$) is defined as
\[
    \df^* u:=-\sum_{j=1}^m \frac{\partial u_j}{\partial x_j}
\]
for $u\in C^\infty (U; \Lambda^1)$ and
\[
    \df^* A:=\sum_{i, j=1}^m \frac{\partial}{\partial x_j}(A_{ij}-A_{ji})\, \df x_i
\]
for $A=\sum_{i, j=1}^m A_{ij}\, \df x_i \wedge \df x_j\in C^\infty (U; \Lambda^2)$.
Then, one can see easily that the following identity holds:
\begin{equation}
    \label{eq_vector_laplacian_higher}
    \Delta u:=\sum_{i, j=1}^m \frac{\partial^2 u_i}{\partial x_j^2}\, \df x_i
    =-\df \df^* u-{\df}^* \df u
\end{equation}
for $u\in C^\infty(U; \Lambda^1)$, and the Lam\'e operator $\Lcal_{\lambda, \mu}=\Gm \Delta u+(\Gl+\Gm)\nabla (\nabla\cdot u)$ can be expressed, analogously to \eqref{eq_lame_vc}, as
\begin{equation}
    \label{eq_lame_another_higher}
    \Lcal_{\lambda, \mu}u=-\mu\,{\df}^* \df u-\frac{\mu}{2k_0}\,\df  \df^* u.
\end{equation}

Let $\GO$ be a bounded domain in $\Rbb^m$ with the Lipschitz boundary $\p\GO$ and $U$ be either $\Omega$ or $\Rbb^m \setminus \overline{\Omega}$. The div-free conormal derivative $\partial^\divf_\nu$ on $\p\GO$ is defined by
\begin{equation}
    \label{eq_d_conormal_higher}
    \partial^\divf_\nu u(x):=\mu \iota_{\nv_x}\df u(x)-\frac{\mu}{2k_0}(\df^* u(x))\nv^\flat_x .
\end{equation}
Here $\iota_{\nv_x}$ is the interior product, namely,
\[
    \iota_{\nv_x}A:=-\sum_{i, j=1}^m (A_{ij}-A_{ji})\nv_j \, \df x_i\in C^\infty (U; \Lambda^1)
\]
for $\nv_x=(\nv_1, \ldots, \nv_m)$ and $A=\sum_{i, j=1}^m A_{ij}\,\df x_i \wedge \df x_j\in C^\infty(U; \Lambda^2)$, and $\nv^\flat_x$ is the differential 1-form defined by
\[
    \nv^\flat_x:=\sum_{j=1}^m \nv_j \, \df x_j \quad (\text{a.e. } x\in \partial\Omega).
\]

\begin{rema*}[Relation to vector calculus]
We identify
\begin{equation}
            \label{eq_identification_1form}
            \sum_{j=1}^m u_j \, \df x_j \simeq (u_1, \ldots, u_m)
        \end{equation}
and
\begin{equation}\label{eq_identification_2form}
A_1\, \df x_2 \wedge \df x_3 +A_2\, \df x_3 \wedge \df x_1+A_3\, \df x_1 \wedge \df x_2 \simeq (A_1, A_2, A_3).
\end{equation}
Then we have the following.
    \begin{itemize}
        \item The operator $\df^*$ acting on the 1-form $u$ is nothing but the divergence (with the minus sign):
        \[
            \df^* u=-\nabla \cdot u.
        \]
        \item If $m=3$, the exterior derivative $\df u$ of the 1-form $u$ corresponds to the rotation of $u$:
        \[
            \df u=\nabla \times u.
        \]
        \item The 1-form $\df^*A$ of the 2-form $A$ corresponds to the rotation of $A$:
        \[
            \df^*A=\nabla \times A.
        \]
        \item The div-free conormal derivative $\partial^\divf_\nu u$ is represented as
        \begin{equation}
            \label{eq_d_conormal_higher_vector}
            (\partial^\divf_\nu u)_i=\mu \sum_{j=1}^m\left(\frac{\partial u_i}{\partial x_j}-\frac{\partial u_j}{\partial x_i}\right) \nv_j+\frac{\mu}{2k_0}(\nabla \cdot u)\nv_i.
        \end{equation}
        In particular, when $m=3$, since
        \[
            \sum_{j=1}^3 \left(\frac{\partial u_1}{\partial x_j}-\frac{\partial u_j}{\partial x_1}\right) \nv_j=-(\nabla\times u)_3 \nv_2+(\nabla \times u)_2 \nv_3
            =-(\nv\times (\nabla\times u))_1
        \]
        and the other components is calculated similarly, the div-free conormal derivative \eqref{eq_d_conormal_higher_vector} coincides with that in the three dimensional case \eqref{eq_conormal_d} under the identification \eqref{eq_identification_1form} and \eqref{eq_identification_2form}.
        \item For $m=2$, the differential 1-form and 2-form are identified with
        \[
            u_1\, \df x_1+u_2\, \df x_2 \simeq (u_1, u_2), \quad
            A\, \df x_1\wedge \df x_2 \simeq A
        \]
        respectively. Then the exterior derivative and its formal adjoint are identified with
        \[
            \df u=-\frac{\partial u_1}{\partial x_2}+\frac{\partial u_2}{\partial x_1}, \quad \df^*A=-\left( -\frac{\partial A}{\partial x_2}, \frac{\partial A}{\partial x_1}\right).
        \]
        These relations appear in the definition of the div-free eNP in two-dimensional case in Section \ref{sec_2D}. For example, one can see that the div-free conormal derivative \eqref{eq_d_conormal_higher} coincides with \eqref{eq_conormal_d_2d} for the two-dimensional case .
    \end{itemize}
\end{rema*}


We define the $L^2$-inner products of differential $p$-forms ($p=1, 2$) by
\[
    \jbracket{u, v}_{L^2(U; \Lambda^1)}:=\sum_{j=1}^n \jbracket{u_j, v_j}_{L^2(U)}
\]
for differential 1-forms $u=\sum_{j=1}^m u_j\, \df x_j$ and $v=\sum_{j=1}^m v_j \, \df x_j$, and
\[
    \jbracket{A, B}_{L^2(U; \Lambda^2)}:=\frac{1}{2}\sum_{i, j=1}^m \jbracket{A_{ij}-A_{ji}, B_{ij}-B_{ji}}_{L^2(U)}
\]
for differential 2-forms $A=\sum_{i, j=1}^m A_{ij}\, \df x_i \wedge \df x_j$ and $B=\sum_{i, j=1}^m B_{ij}\, \df x_i \wedge \df x_j$. We also define an inner product
\[
    \jbracket{f, g}_{L^2(\partial\Omega; \iota^*\Lambda^1)}:=\sum_{j=1}^m \jbracket{f_j, g_j}_{L^2(\partial\Omega)}
\]
for $f(x)=\sum_{i=1}^m f_j(x)\, \df x_j$ and $g(x)=\sum_{i=1}^m g_j(x)\, \df x_j$ ($x\in \partial\Omega$). Here, in terms of differential geometry, $\iota^*\Lambda^1$ means the pull-back of the vector bundle $\Lambda^1=T^*\Rbb^m$ by the inclusion mapping $\iota: \partial\Omega \hookrightarrow \Rbb^m$. We denote the Sobolev space of order $s$ of the differential $p$-forms on $U$ by $H^s(U; \Lambda^p)$.

For $\varphi\in H^1(\Omega)$, $u\in H^1(\Omega; \Lambda^1)$ and $A\in H^1(\Omega; \Lambda^2)$, we have
\begin{equation}
    \label{eq_green_higher}
    \begin{aligned}
        \jbracket{\df \varphi, u}_{L^2(\Omega; \Lambda^1)}
        &=\jbracket{\varphi \nv^\flat, u}_{L^2(\partial\Omega; \iota^*\Lambda^1)}+\jbracket{\varphi, \df^* u}_{L^2(\Omega)}, \\
        \jbracket{\df u, A}_{L^2(\Omega; \Lambda^2)}
        &=\jbracket{u, \iota_\nv A}_{L^2(\partial\Omega; \iota^*\Lambda^1)}+\jbracket{u, {\df}^*A}_{L^2(\Omega; \Lambda^1)}.
    \end{aligned}
\end{equation}
The corresponding formulas in $\overline{\GO}^c= \Rbb^m\setminus \overline{\GO}$ is obtained similarly. For example, the second identity is obtained as follows. We calculate the integrand as
\begin{align*}
    &\frac{1}{2}\sum_{i, j=1}^m \left(\frac{\partial u_j}{\partial x_i}-\frac{\partial u_i}{\partial x_j}\right) (A_{ij}-A_{ji}) \\
    &=\sum_{j=1}^m \frac{\partial}{\partial x_j}\left( \sum_{i=1}^m u_i (A_{ji}-A_{ij})\right)
    -\sum_{i=1}^m u_i \sum_{j=1}^m \frac{\partial}{\partial x_j}(A_{ji}-A_{ij}).
\end{align*}
Thus we have
\begin{align*}
    &\jbracket{\df u, A}_{L^2(\Omega; \Lambda^2)} \\
    &=\sum_{i=1}^m \int_{\partial\Omega} u_i \sum_{j=1}^m \nv_j (A_{ji}-A_{ij})\, \df\sigma -\sum_{i=1}^m \int_\Omega u_i \sum_{j=1}^m \frac{\partial}{\partial x_j}(A_{ji}-A_{ij})\, \df x \\
    &=\jbracket{u, \iota_\nv A}_{L^2(\partial\Omega; \iota^*\Lambda^1)}+\jbracket{u, \df^*A}_{L^2(\Omega; \Lambda^1)}
\end{align*}
as desired.

Using these formulas, we obtain the following lemma similar to Lemma \ref{lemm_green_vc}. Here again $\jbracket{\cdot, \cdot}_{\GO}$ denotes either the inner product on $L^2(\GO)$ or $L^2(\GO; \Lambda^1)$ or $L^2(\GO; \Lambda^2)$ (or the duality pairing between $H^{-1}(\GO; \Lambda^1)$ and $H^{1}(\GO; \Lambda^1)$), and likewise for $\jbracket{\cdot, \cdot}_{\overline{\GO}^c}$ and $\jbracket{\cdot, \cdot}_{\p\GO}$.

\begin{lemm}\label{theo_lame_bilinear_higher}
    Let $\Omega$ be a bounded domain in $\Rbb^m$ with the Lipschitz boundary $\partial\Omega$. Then the following statements hold.
    \begin{enumerate}[label={\rm (\roman*)}]
        \item \label{enum_lame_bilinear_in_higher} It holds for all $u, v\in H^1(\GO; \Lambda^1)$ that
        \begin{equation}
            \label{eq_lame_bilinear_1_in_higher}
            \begin{aligned}
                &\jbracket{\Lcal_{\lambda, \mu}u, v}_{\GO}
            =\jbracket{\partial^\divf_\nu u|_-, v}_{\p\GO}
            -\mu\jbracket{\df u, \df v}_{\GO}
            -\frac{\mu}{2k_0}\jbracket{\df^* u, \df^* v}_{\GO}
            \end{aligned}
        \end{equation}
        and
        \begin{equation}
            \label{eq_lame_bilinear_in_higher}
            \begin{aligned}
                &\jbracket{\Lcal_{\lambda, \mu}u, v}_{\GO}-\jbracket{u, \Lcal_{\lambda, \mu}v}_{\GO} =\jbracket{\partial^\divf_\nu u|_-, v}_{\p\GO}
            -\jbracket{u, \partial^\divf_\nu v|_-}_{\p\GO}.
            \end{aligned}
        \end{equation}
        \item \label{enum_lame_bilinear_ex_higher} If $u$, $v\in H^1(\overline{\GO}^c)^m$ satisfy $|u(x)|$, $|v(x)|=O(|x|^{-m/2+1-\delta})$, $|\nabla u(x)|$, $|\nabla v(x)|=O(|x|^{-m/2-\delta})$ and $|\Lcal_{\lambda, \mu}u(x)|$, $|\Lcal_{\lambda, \mu}v(x)|=O(|x|^{-m/2-1-\delta})$ as $|x|\to \infty$ for some $\delta>0$, then we have
        \begin{equation}
            \label{eq_lame_bilinear_1_ex_higher}
            \begin{aligned}
                &\jbracket{\Lcal_{\lambda, \mu}u, v}_{\overline{\GO}^c}
            =-\jbracket{\partial^\divf_\nu u|_+, v}_{\p\GO}
            -\frac{\mu}{2}\jbracket{\df u, \df v}_{\overline{\GO}^c}
            -\frac{\mu}{2k_0}\jbracket{\df^* u, \df^* v}_{\overline{\GO}^c}
            \end{aligned}
        \end{equation}
        and
        \begin{equation}
            \label{eq_lame_bilinear_ex_higher}
            \begin{aligned}
                &\jbracket{\Lcal_{\lambda, \mu}u, v}_{\overline{\GO}^c}-\jbracket{u, \Lcal_{\lambda, \mu}v}_{\overline{\GO}^c}
                =-\jbracket{\partial^\divf_\nu u|_+, v}_{\p\GO}
            +\jbracket{u, \partial^\divf_\nu v|_+}_{\p\GO}.
            \end{aligned}
        \end{equation}
    \end{enumerate}
\end{lemm}

In what follows, for $m\geq 2$, we assume that the Lam\'e constants $(\lambda, \mu)$ satisfies
\begin{equation}
    \label{eq_s_convex_condition_higher}
    m\lambda+2\mu>0, \ \mu>0,
\end{equation}
which is a generalization of the condition \eqref{eq_s_convex_condition} to a general dimensional case.
The Kelvin matrix $\Gamma (x)=(\Gamma_{ij}(x))_{i, j=1}^m$ of the fundamental solution to the Lam\'e system in $\Rbb^m$ is given by
\begin{equation}
    \label{eq_kelvin_general}
    \Gamma_{ij}(x)=- \frac{\Ga_1}{(m-2)\omega_m} \frac{\Gd_{ij}}{|x|^{m-2}} - \frac{\Ga_2}{\omega_m} \displaystyle \frac{x_i x_j}{|x|^m}
\end{equation}
where $\omega_m$ is the area of the unit sphere in $\Rbb^m$ and $\alpha_1$ and $\alpha_2$ are constants defined in \eqref{eq_alpha_12} (see \cite[(0.1)]{DKV-Duke-88} or \cite{Kup-book-65} for example). The div-free double layer potential $\Dcal^\divf$ and eNP $\Kcal^\divf$ are defined respectively by \eqref{eq_dd} and \eqref{eq_kd}. One can show that $\Dcal^\divf$ produces div-free solutions (cf. Lemma \ref{lemm_Dcalf}). Like the three-dimensional case, we define
\beq
    \label{eq_k_higher}
    \begin{aligned}
    \Kcal_1[f](x)& =\frac{1}{\Go_m}\,\mathrm{p.v.}\int_{\p \GO} \frac{(x-y)(\nv_y\cdot f(y))-\nv_y ((x-y)\cdot f(y))}{|x-y|^m} \df \Gs (y), \\
    \Kcal_2[f](x) &=-\frac{1}{\Go_m} \,\mathrm{p.v.}\int_{\p \GO} \frac{(x-y)\cdot \nv_y}{|x-y|^m}f(y) \df \Gs (y), \\
    \Kcal_3[f](x)&=\frac{1}{\Go_m}\mathrm{p.v.}\int_{\p\GO} \frac{((x-y)\cdot f(y))((x-y)\cdot \nv_y)(x-y)}{|x-y|^{m+2}}\, \df \Gs (y).
    \end{aligned}
\eeq
Then one can show that the following formulas hold:
\[
\Kcal=2k_0 (\Kcal_1+\Kcal_2)-m(1-2k_0)\Kcal_3, \quad \Kcal^\divf=-\Kcal_1+\Kcal_2.
\]
(cf. \eqref{eq_kd_decomposition}, \eqref{eq_np_decomposition}). Thus $\Kcal +2k_0 \Kcal^\divf$ is compact on $\Hsb^m$ if $\p\GO$ is $C^{1,\Ga}$ for some $\Ga>0$ (cf. Lemma \ref{lemm_np23_compact}).

One can show in the same way as the three-dimensional case that the single layer potential $\Scal$ defined by $\GG$ in \eqref{eq_kelvin_general} is bounded and invertible as the linear mapping from $(\Hsb^*)^m$ onto $\Hsb^m$ so that \eqref{eq_inner_product} yields an inner product on $\Hsb^m$. One can also prove, in the same as the three-dimensional case, the Green's formula (cf. Lemma \ref{theo_sd_d}),  jump formulas (cf. Proposition \ref{prop_dd_jump} and \ref{theo_jump_npds}), and Plemelj's symmetrization principles (cf. \eqref{eq_plemelj} and \eqref{theo_plemelj_kd}).

Like the three dimensional case, in higher dimensions $u^f=\Scal [\Scal^{-1}[f]]$ (in $\Rbb^m$) provides the unique solutions to the interior problem \eqref{eq_lame_bvp} and the exterior  problem \eqref{eq_lame_bvp_ext} (with the condition $u(x)=O(|x|^{-m+2})$ as $|x|\to \infty$).
Thus we apply Lemma \ref{theo_lame_bilinear_higher} to $u=v=u^f$ and obtain
\[
        \jbracket{\left(-\Kcal^\divf+\frac{1}{2}I\right)[f], f }_*
        =\mu \left\| \df u^f \right\|_{\Omega}^2+\frac{\mu}{2k_0}\| \df^* u^f \|_{\Omega}^2
    \]
    and
    \[
        \jbracket{\left(\Kcal^\divf+\frac{1}{2}I \right)[f], f }_*
        =\mu \left\| \df u^f \right\|_{\overline{\Omega}^c}^2+\frac{\mu}{2k_0}\| \df^* u^f \|_{\overline{\Omega}^c}^2
    \]
(cf. Lemma \ref{lemm_div_rot_npd_identity}).

Then, by the same argument as in the three dimensional case, we obtain the following theorem which collects the results analogous to those in three dimensions. We indicate the corresponding theorems in three dimensions. In what follows, the subspaces $\Hdrf^\pm$ are defined as in \eqref{eq_free_spaces} with $\nabla \times v_\pm^f$ replaced by $\df v_\pm^f$. In other words, $\Hdrf^\pm$ is the collection of all $f\in \Hsb^m$ such that
\[
    \nabla\cdot v_\pm^f=0, \quad \frac{\partial (v_\pm^f)_i}{\partial x_j}=\frac{\partial (v_\pm^f)_j}{\partial x_i} \quad (i, j=1, \ldots, m)
\]
in respective space. The orthogonality of subspaces or functions in $\Hsb^m$ is with respect to the inner product \eqref{eq_inner_product}.

\begin{theo}\label{theo_general_dim}
    Let $\Omega$ be a bounded domain in $\Rbb^m$ ($m\geq 3$) with the Lipschitz boundary. Then the following statements hold.
    \begin{enumerate}[label={\rm (L\arabic*)}]
        \item \label{enum_hm_decomposition}(cf. Theorem \ref{theo_decomposition_div_free}) The orthogonal decomposition
        \begin{equation}
            \label{eq_Hmorthogonal}
            \Hsb^m=\Hdrf^-\oplus \Hdrf^+ \oplus (\Hdrf^-+\Hdrf^+)^\perp
        \end{equation}
        and the inclusion
        \begin{equation}
            \label{eq_perpsub_m}
            (\Hdrf^-+\Hdrf^+)^\perp\subset \Hdf
        \end{equation}
        hold.
        \item \label{enum_drf_ker_m}(cf. Theorem \ref{theo_div_rot_free_kd}) $\Hdrf^\pm=\Ker (\Kcal^\divf\pm 1/2I)$.
        \item \label{enum_div_free_inv_md}(cf. Theorem \ref{theo_div_free_invariant}) The operator $\Kcal^\divf$ preserves the subspace $\Hdf$.
        \item \label{enum_infinite_m}(cf. Theorem \ref{theo_infinite_dim_c1a}) The spaces $\Hdrf^\pm$ and $\Hdf$ are infinite-dimensional.
    \end{enumerate}

    Furthermore, if $\partial\Omega$ is $C^{1, \alpha}$ for some $\alpha>0$, then the following statement holds.

    \begin{enumerate}[label={\rm (H\arabic*)}]
        \item \label{enum_compact_drf_m}(cf. Theorem \ref{theo_free_np_3D}) The restricted operators
        \begin{align*}
            \Kcal+k_0I: \Hdrf^-\longrightarrow \Hsb^m, \\
            \Kcal-k_0I: \Hdrf^+\longrightarrow \Hsb^m
        \end{align*}
        are compact. Here $k_0$ is the positive constant defined by \eqref{eq_k0}.
    \end{enumerate}

    Furthermore, if $\partial\Omega$ is $C^{1, \alpha}$ for some $\alpha>1/2$, then the following statements hold.
    \begin{enumerate}[label={\rm (S\arabic*)}]
        \item \label{enum_finite_m}(cf. Theorem \ref{theo_codim_betti} and Corollary \ref{coro_simply_connected}) The inequality
        \[
            \dim \Hdf/(\Hdrf^-+\Hdrf^+)^\perp\leq b_1(\partial\Omega)
        \]
        holds. Furthermore, if $\partial\Omega$ is simply connected, then $\Hsb^m$ admits the orthogonal decomposition $\Hsb^m=\Hdrf^-\oplus \Hdrf^+ \oplus \Hdf$.
        \item \label{enum_compact_m}(cf. Theorem \ref{theo_free_np_3D} and Corollary \ref{coro_polynomially_cpt_3d}) The restricted operator
        \[
            \Kcal: \Hdf \longrightarrow \Hsb^m
        \]
        is compact. Moreover, the operator $\Kcal (\Kcal^2-k_0^2 I)$ is compact on $\Hsb^m$.
        \item \label{enum_ev_md}(cf. Theorem \ref{thm_eigenvalue3D}) The spectrum of the eNP operator $\Kcal$ on $\Hsb^m$ consists of three infinite sequences of real eigenvalues converging to $0$, $k_0$ and $-k_0$.
        \item \label{enum_ev_kd_md}(cf. Theorem \ref{theo_isolated_c1a}) The spectrum $\sigma (\Kcal^\divf)$ of $\Kcal^\divf$ on $\Hcal^m$ consists of $-1/2$, $1/2$, $0$ and sequence of eigenvalues $ \lambda_j \in (-1/2, 1/2)$ of finite multiplicities converging to $0$, where $\pm 1/2$ are eigenvalues of infinite multiplicities. The spaces $\rmop{Ran}(\Kcal^\divf\pm 1/2I)$ are closed in $\Hsb^m$. In particular, we have
        \[
            (\Hdrf^-+\Hdrf^+)^\perp =\Ran (\Kcal^\divf +1/2I)(\Kcal^\divf-1/2I).
        \]
    \end{enumerate}
\end{theo}

The proof of Theorem \ref{theo_general_dim} is almost same as that in the three dimensional case. We make three remarks.

The formula \eqref{eq_np_modulo_cpt} for $\Kcal$ holds in higher dimensions with $\Scal_0$, the single layer potential with respect to the Laplacian on $\Rbb^m$, namely,
    \[
        \Scal_0[f](x):=-\frac{1}{(m-2)\omega_m}\int_{\partial\Omega} \frac{1}{|x-y|^{m-2}}f(y)\, \df\sigma (y).
    \]
The formula is used in proving \ref{enum_compact_m}.

The argument in the proof of Theorem \ref{theo_infinite_dim_c1a} is also applicable if we replace $\nabla$ (gradient) and $\nabla \times$ (rotation) with the exterior derivative $\df$ for functions and differential 1-forms respectively. In fact, if we set $u:=\Scal_0[\varphi \nv]\in C^\infty (\Rbb^m\setminus \p\GO)^m \simeq C^\infty (\Rbb^m\setminus \p\GO; \Lambda^1)$ and assume $\df u=0$, we have $\Delta u=0$ and thus
\[
    -\df \df^*u=\df^*\df u+\Delta u=0 \quad \text{in } \Rbb^m \setminus \p\GO.
\]
Hence $-\df^*u$ is locally constant on $\Rbb^m\setminus \p\GO$. Since $-\df^*u$ is nothing but $\nabla\cdot u$ under the identification $C^\infty (\Rbb^m\setminus \p\GO)^m \simeq C^\infty (\Rbb^m\setminus \p\GO; \Lambda^1)$, we can apply the same argument as in Theorem \ref{theo_infinite_dim_c1a}.

Concerning \ref{enum_finite_m}, the spaces $\Mcal^\pm$ in \eqref{eq_mpm} are defined by replacing $\nabla \varphi$ to $\df \varphi$. Furthermore, the first de Rham cohomology of the open subset $U$ of $\Rbb^m$ is defined as
\[
    H_\mathrm{dR}^1 (U):=\frac{\Ker (\df: C^\infty (U; \Lambda^1)\longrightarrow C^\infty (U; \Lambda^2))}{\Ran (\df: C^\infty (U)\longrightarrow C^\infty (U; \Lambda^1))}.
\]
This coincides with the three-dimensional case \eqref{eq_derham}.

We emphasize that Theorem \ref{theo_sphere_orthogonal} holds for higher dimensions with the same proof as in three-dimensional case:
\begin{theo}\label{theo_sphere_orthogonal_higher}
    If $\Omega$ is a ball in $\Rbb^m$ ($m\geq 3$), then we have the orthogonal decomposition
    \[
        \Hsb^m=\Hdrf^- \oplus \Hdrf^+ \oplus \Hdf.
    \]
\end{theo}

\begin{proof}
    Let
\[
\Wcal_k^\pm = \Hdrf^\pm \cap P_k^m, \quad \Wcal_k^o = \Hdf^\pm \cap P_k^m
\]
as in \eqref{eq_wcaldef}.
    We need to have the relation analogous to \eqref{eq_Scal_ball}. We claim that the following relation holds for $f\in \Wcal_k^o$ in dimension $m$:
    \beq\label{eq_Scal_ball_higher}
        \Scal^{-1}[f]=-\mu (2k+m-2)f.
    \eeq
To show this, let, for $f\in \Wcal_k^o$, $p \in P_k^m$ be such that $p|_{\p\GO}=f$. Then, we have
\[
p^*(x) = \frac{1}{|x|^{2k+m-2}} p(x).
\]
Here, the Kelvin transform $p^*(x)$ of $p(x)$ is defined as
\[
    p^*(x):=\frac{1}{|x|^{m-2}}p\left(\frac{x}{|x|^2}\right).
\]
Thus we have
    \begin{align*}
        \Scal^{-1}[f] &=\partial^\divf_\nu p^*-\partial^\divf_\nu p \\
        &=-\mu \iota_\nv\df p^*  + \mu \iota_\nv\df p \\
        &=-\mu (2k+m-2) f(x),
    \end{align*}
    where the last equality holds since $x \cdot p(x)=0$.
\end{proof}

We also obtain the results analogous to Theorem \ref{theo_enp_eigenfunctions} and Theorem \ref{theo_div_rot_asymptotic} by the same argument as in the three dimensional case.

\begin{theo}
    Let $m\geq 3$ and $\Omega$ be a bounded domain in $\Rbb^m$ whose boundary is $C^{1, \alpha}$ for some $\alpha>1/2$. Let $\{f_j\}_{j=1}^\infty$ be an orthonormal system in $\Hsb^m$ consisting of eigenfunctions of $\Kcal$. Let $\lambda_j\in \Rbb$ be the corresponding eigenvalue, i.e., $\Kcal [f_j]=\lambda_j f_j$. We decompose $f_j$ into the sum $f_j=f^-_j+f^+_j+f^o_j$ where
    \[
        f^\pm_j \in \Hdrf^\pm, \quad f^o_j\in (\Hdrf^-+\Hdrf^+)^\perp \subset \Hdf.
    \]
    Then, the following statements hold.
    \begin{enumerate}[label={\rm (\roman*)}]
        \item If $\lambda_j\to 0$, then
        \begin{align*}
            &\lim_{j\to \infty} \| f_j-f^o_j\|_*=0, \\
            &\lim_{j\to \infty}\| \df^* u^{f_j}\|_{\GO}=\lim_{j\to \infty}\| \df^* u^{f_j} \|_{\overline{\GO}^c}= 0, \\
            &\lim_{j\to\infty}\| \df u^{f_j}\|_{\GO}=\lim_{j\to \infty}\left\| \df u^{f_j}\right\|_{\overline{\Omega}^c}= \frac{1}{\sqrt{2\Gm}}.
        \end{align*}
        \item If $\lambda_j\to k_0$, then
        \begin{align*}
            &\lim_{j\to \infty} \|f_j-f^+_j\|_*=0, \\
            &\lim_{j\to\infty}\| \df^* u^{f_j}\|_{\overline{\Omega}^c}=\lim_{j\to\infty}\|  \df u^{f_j} \|_{\overline{\Omega}^c}=0, \\
            &\lim_{j\to\infty} \left( \Gm \|  \df u^{f_j}\|_{\Omega}^2+\frac{\Gm}{2k_0}\| \df^* u^{f_j} \|_{\Omega}^2\right)= 1.
        \end{align*}
        \item If $\lambda_j\to -k_0$, then
        \begin{align*}
            &\lim_{j\to \infty} \|f_j-f^-_j\|_*=0, \\
            &\lim_{j\to\infty}\| \df^* u^{f_j}\|_{\Omega}=\lim_{j\to\infty}\| \df u^{f_j} \|_{\Omega}=0, \\
            &\lim_{j\to\infty} \left( \Gm \|  \df u^{f_j}\|_{\overline{\Omega}^c}^2+\frac{\Gm}{2k_0}\| \df^* u^{f_j} \|_{\overline{\Omega}^c}^2\right)= 1.
        \end{align*}
    \end{enumerate}
\end{theo}


\section{The two-dimensional case}\label{sec_2D}

We define
\[
    \nabla^\perp u:=\left(-\frac{\p u}{\p y}, \frac{\p u}{\p x}\right)
\]
for $u\in C^1(\Rbb^2)$, and
\[
    \rot u=\nabla^\perp \cdot u:=-\frac{\p u_1}{\p y}+\frac{\p u_2}{\p x}
\]
for $u=(u_1, u_2)\in C^1(\Rbb^2)^2$. Similarly to \eqref{eq_lame_vc} and \eqref{eq_conormal_d} for the three-dimensional case, the Lam\'e operator $\Lcal_{\lambda, \mu}$ is represented as
\begin{equation}
    \label{eq_lame_2_2d}
    \Lcal_{\lambda, \mu}u=\mu \nabla^\perp (\rot u)+\frac{\mu}{2k_0}\nabla (\nabla\cdot u).
\end{equation}
The Lam\'e constants $(\lambda, \mu)$ is assumed to satisfy the condition \eqref{eq_s_convex_condition_higher} with $m=2$.
Let $\GO$ be a bounded domain in $\Rbb^2$ with the Lipschitz boundary $\p\GO$. The div-free conormal derivative on $\p\GO$ is defined by
    \begin{equation}
    \label{eq_conormal_d_2d}
    \partial^\divf_\nu u:=\mu (\rot u) \nv^\perp+\frac{\mu}{2k_0}(\nabla\cdot u) \nv,
    \end{equation}
    where $\nv$ is the outward normal vector field on $\p \GO$ and
    \[
    (x_1,x_2)^\perp:=(-x_2, x_1).
    \]

Like the three-dimensional case, the div-free double layer potential $\Dcal^\divf$ and eNP $\Kcal^\divf$ are defined respectively by \eqref{eq_dd} and \eqref{eq_kd}. We compare these operators with the Cauchy integral $\Ccal$, which is defined to be
    \begin{equation}
        \label{eq_cauchy}
        \Ccal[f](z):=\frac{1}{2\pi\iu}\int_{\p\GO} \frac{f(w)}{w-z}\, \df w, \quad z\in \Cbb \setminus \p\GO.
    \end{equation}
    Let $\Ccal_b$ be the Cauchy transform on $\p\GO$, namely,
    \begin{equation}
        \label{eq_cauchy_trans}
        \Ccal_b[f](z):=\frac{1}{2\pi\iu} \mathrm{p.v.} \int_{\p\GO} \frac{f(w)}{w-z}\, \df w, \quad z\in \p\GO.
    \end{equation}
It is well-known that the following jump relation holds:
    \begin{equation}
        \label{eq_cc_2d_jump}
        \Ccal [f]|_\pm=\left(\Ccal_b\mp\frac{1}{2}I \right)[f].
    \end{equation}
See \cite[(17.2)]{Muskh} for the case when $\p\GO$ smooth, and \cite{Calderon, EFV92, Verchota84} for the Lipschitz case.

The Kelvin matrix $\GG = ( \GG_{ij} )_{i, j = 1}^2$ of fundamental solutions to the Lam\'{e} operator in two dimensions is given by
\beq
  \GG_{ij}(x) =
  \frac{\Ga_1}{2 \pi} \Gd_{ij} \ln{|x|} - \frac{\Ga_2}{2 \pi} \displaystyle \frac{x_i x_j}{|x| ^2}
\eeq
with $\Ga_1$ and $\Ga_2$ given in \eqref{eq_alpha_12}.
The operators $\Scal$, $\Kcal$, $\Dcal^\divf$ and $\Kcal^\divf$ are defined by \eqref{eq_single_layer_potential}, \eqref{eq_np}, \eqref{eq_dd} and \eqref{eq_kd}, respectively.

In what follows, we identify $(v_1,v_2) \in \Rbb^2$ with $v_1+i v_2 \in \Cbb$ and denote both of them by $v$, namely, $v=(v_1,v_2)$ and $v=v_1+i v_2$.

The following lemma plays a crucial role in this section.

\begin{lemm}\label{theo_dd_2d_cauchy}
    Let $\GO$ be a bounded domain in $\Rbb^2$ with the Lipschitz boundary. It holds for all $f\in \Hsb^2$ that
    \begin{equation}
        \label{eq_dd_2d_cauchy}
        \Dcal^\divf [f](x)=\overline{\Ccal [\overline{f}](x)}, \quad x \in \Rbb^2 \setminus \p\GO
    \end{equation}
    and
    \begin{equation}
        \label{eq_np_2d_cauchy}
        \Kcal^\divf [f]=\overline{\Ccal_b [\overline{f}]} \quad \mbox{on } \p\GO.
    \end{equation}
\end{lemm}

\begin{proof}
     Straightforward calculations show the following formula:
\begin{equation}
    \label{eq_dd_2d_explicit}
    (\p^\divf_{\nu_y}\Gamma (x-y))^T f(y) =-\frac{1}{2\pi} \frac{(\nv_y^\perp\cdot f(y))(x-y)^\perp + (\nv_y\cdot f(y))(x-y)}{|x-y|^2}.
\end{equation}
We use notation $z=x_1+\iu x_2$ for $x=(x_1, x_2)$ and $w=y_1+\iu y_2$ for $y=(y_1, y_2)$. Then we have
    \[
        x\cdot y =\Re (z\overline{w}), \quad x^\perp = \iu z.
    \]
    Suppose that $x \neq y$ and $y \in \p\GO$. Then \eqref{eq_dd_2d_explicit} becomes
    \begin{align}
        (\p^\divf_{\nu_y}\Gamma (x-y))^T f(y)
        &=-\frac{1}{2\pi} \frac{\Re (-\iu \overline{\nv_w}f(w))\iu (z-w)+\Re (\overline{\nv_w}f(w))(z-w)}{|z-w|^2} \nonumber\\
        &=-\frac{1}{2\pi} \frac{\overline{\nv_w}f(w)(z-w)}{|z-w|^2} \nonumber\\
        &=-\frac{1}{2\pi}\overline{\frac{\overline{f(w)}}{z-w} \nv_w }. \label{eq_dd_2d_cauchy_wip}
    \end{align}
    If $\partial\Omega$ is connected, then we parametrize the boundary $\p\GO$ by the arc-length parameter $s\in [0, L]$ where $L$ is the length of $\p\GO$. Let $w(s): [0, L]\to \p\GO$ be the parametrization and denote the velocity vector $\df w/\df s$ by $\dot w$. Since $s$ is the arc-length parameter, we have $\iu \dot w(s)=\nv_{w(s)}$. Then we have
    \begin{align*}
        \nv_w\, \df\sigma (w) =-\iu \dot w(s)\, \df s = \frac{1}{\iu} \df w.
    \end{align*}
    Thus \eqref{eq_dd_2d_cauchy_wip} yields \eqref{eq_dd_2d_cauchy}.

    If $\partial\Omega$ has multiple connected components, then we argue in the same way as above on each connected component and obtain the same result.
\end{proof}

We immediately obtain the following proposition from \eqref{eq_dd_2d_cauchy} and \eqref{eq_np_2d_cauchy}:

\begin{prop}
    \label{theo_dd_2d_jump}
    Let $\GO$ be a bounded domain in $\Rbb^2$ with the Lipschitz boundary. It holds that
    \begin{equation}
        \label{eq_dd_2d_jump}
        \Dcal^\divf [f]|_\pm=\left(\Kcal^\divf\mp \frac{1}{2}I \right)[f]
    \end{equation}
    for all $f\in \Hsb^2$.
\end{prop}

If the function $u=u_1+iu_2$ is holomorphic in $\GO$, then the vector-valued function $\overline{u}=(u_1,-u_2)$ satisfies $\nabla \cdot \overline{u}=0$ and $\rot \overline{u}=0$ in $\GO$, and vice versa. Thus we have
    \[
        \Hdrf^-=\{ f \in \Hsb^2 \mid \text{$\overline{f}$ extends to $\GO$ as a holomorphic function} \}.
    \]
If $f$ extends to $\GO$ as a holomorphic function, then $\Ccal[f](z)=0$ for all $z \in \Cbb \setminus \overline{\GO}$ by Cauchy's theorem, and hence $(\Ccal_b-1/2I )[f]=0$. The converse also holds. In fact, if $(\Ccal_b-1/2I )[f]=0$, then $\Ccal[f]$ is the holomorphic function in $\GO$ such that $\Ccal[f]|_{\p\GO}=(\Ccal_b+1/2I )[f]=f$.
Thus we have
\beq
\overline{f} \in \Hdrf^- \Longleftrightarrow f \in \Ker \left(\Ccal_b-\frac{1}{2}I \right).
\eeq
Moreover, if $f = (\Ccal_b+1/2I)[g]$ for some $g$, then $\Ccal[f]$ is the holomorphic function in $\GO$ such that $\Ccal[f]|_{\p\GO}=f$, and hence $f \in \Ker \left(\Ccal_b-1/2I \right)$. If $f \in \Ker \left(\Ccal_b-1/2I \right)$, then the identity
\beq\label{eq_identity_cauchy}
f= - \left(\Ccal_b-\frac{1}{2}I \right)[f]+ \left(\Ccal_b+\frac{1}{2}I \right)[f]
\eeq
shows that $f \in \Ran (\Ccal_b+1/2I)$. Thus we have
\beq
\Ker \left(\Ccal_b-\frac{1}{2}I \right) = \Ran \left(\Ccal_b+\frac{1}{2}I \right).
\eeq
Similarly, one can show that
\beq
\overline{f} \in \Hdrf^+ \Longleftrightarrow f \in \Ker \left(\Ccal_b+\frac{1}{2}I \right) = \Ran \left(\Ccal_b-\frac{1}{2}I \right).
\eeq

There are two-dimensional bounded domain $\GO$ on which the operator $\Scal: (\Hsb^2)^*\to \Hsb^2$ is not invertible (see \cite{AJKKY}). However, we may dilate the domain so that $\Scal$ becomes invertible on the dilated domain. This fact is well-known for the Laplace case (see \cite{Verchota84}), and we include in Appendix \ref{sect_dilation_elastic} a proof for the case of the Lam\'e system. Since dilation of a domain does not alter spectral properties of (e)NP operators, we assume from the beginning that $\Scal: (\Hsb^*)^2\to \Hsb^2$ is invertible and the bilinear form \eqref{eq_inner_product} is a genuine inner product on $\Hsb^2=H^{1/2}(\partial\Omega)$. Then, since the symmetrization principle \eqref{eq_plemelj_kd} holds in the two-dimensional case as well as in the three-dimensional case, the div-free eNP operator $\Kcal^\divf$ is realized as a self-adjoint operator on $\Hsb^2$ with respect to \eqref{eq_inner_product}.

We thus obtain the following theorem from \eqref{eq_np_2d_cauchy}.

\begin{theo}\label{coro_hdr_ker}
    Let $\GO$ be a bounded domain in $\Rbb^2$ with the Lipschitz boundary. It holds that
    \[
        \Hdrf^\pm=\rmop{Ker} \left(\Kcal^\divf\pm \frac{1}{2}I\right)=\rmop{Ran}\left(\Kcal^\divf\mp \frac{1}{2}I\right).
    \]
\end{theo}

We see from the identity \eqref{eq_identity_cauchy} that
\[
f= \left(\Kcal^\divf+ \frac{1}{2}I \right)[f]- \left(\Kcal^\divf- \frac{1}{2}I \right)[f],
\]
where the sum is direct. Thus we obtain the following theorem.

\begin{theo}\label{coro_direct_sum_2d}
    Let $\GO$ be a bounded domain in $\Rbb^2$ with the Lipschitz boundary. It holds that
    \begin{equation}\label{eq_2d_direct}
        \Hsb^2 = \Hdrf^- \oplus \Hdrf^+,
\end{equation}
    where the sum is orthogonal with respect to the inner product \eqref{eq_inner_product}.
\end{theo}

\begin{rema*}
    There is a decomposition analogous to \eqref{eq_2d_direct} in higher dimensions: a decomposition of Clifford algebra-valued functions in terms of Dirac operators. For this we refer to \cite{HMMPT09, Mitrea94}.
\end{rema*}

Like the three-dimensional case, we see that if $\p\GO$ is $C^{1,\Ga}$ for some $\Ga>0$, then $\Kcal +2k_0 \Kcal^\divf$ is compact on $\Hsb^2$, and hence obtain the following theorem proved in \cite{AJKKY}.

\begin{theo}[\cite{AJKKY}]
    \label{theo_np_compact_2d}
If $\GO$ is a bounded domain in $\Rbb^2$ whose boundary is $C^{1,\Ga}$ for some $\Ga >0$, then the restricted operators
\[
    \Kcal+k_0 I: \Hdrf^- \longrightarrow \Hsb^2, \quad
    \Kcal-k_0 I: \Hdrf^+ \longrightarrow \Hsb^2
\]
are compact, and the operator $\Kcal^2-k_0^2 I$ is compact on $\Hsb^2$. The eigenvalues of the eNP operator $\Kcal$ on $\Hsb^2$ consist of two infinite sequences converging to $k_0$ and $-k_0$.
\end{theo}


We now have the following theorem analogous to Theorem \ref{theo_enp_eigenfunctions}. We omit the proof since it can be proved in a similar way.

\begin{theo}\label{theo_enp_eigenfunctions2D}
    Let $\GO$ be a bounded domain in $\Rbb^2$ whose boundary is $C^{1,\Ga}$ for some $\Ga >0$. Let $\{ f_j\}_{j=1}^\infty$ be an orthonormal system in $\Hsb^2$ consisting of eigenfunctions of $\Kcal$ with respect to the inner product \eqref{eq_inner_product} and let $\Gl_j\in \Rbb$ be the corresponding eigenvalue, {\it i.e.}, $\Kcal [f_j]=\Gl_j f_j$. We decompose $f_j=f_j^-+f_j^+$ by the orthogonal sum \eqref{eq_2d_direct} where $f^\pm_j \in \Hdrf^\pm$. Then the following statements hold.
    \begin{enumerate}[label={\rm (\roman*)}]
        \item If $\Gl_j\to k_0$, then $f_j-f^+_j\to 0$ in $\Hsb^2$.
        \item If $\Gl_j\to -k_0$, then $f_j-f^-_j\to 0$ in $\Hsb^2$.
    \end{enumerate}
\end{theo}


\section*{Discussions}\label{sec:discussion}

In this paper, we proved that the following decomposition holds:
\beq\label{eq_decomposition_div_free2}
        \Hsb^3=\Hdrf^- \oplus \Hdrf^+ + \Hdf.
    \eeq
We then proved that if $\p\GO$ is $C^{1,\Ga}$ for some $\Ga >1/2$, then the operators $\Kcal + k_0I$, $\Kcal - k_0I$ and $\Kcal$ are compact on $\Hdrf^{-}$, $\Hdrf^{+}$ and $\Hdf$, respectively. As a consequence, we showed that the eigenvalues of the eNP operator $\Kcal$ on $\Hsb^3$ consists of three infinite real sequences converging to $0$, $k_0$, and $-k_0$, and corresponding eigenfunctions are characterized by $\Hdrf^{+}$, $\Hdrf^{+}$ and $\Hdrf^{-}$.

We proved that if $\p\GO$ is simply connected, then the decomposition \eqref{eq_decomposition_div_free2} is orthogonal. However, we do not know if this is the case in general and it is quite interesting to prove (or disprove) this. In relation to this question, we mention that the following three statements are equivalent.
\begin{itemize}
\item[(i)] $(\Hdrf^-+\Hdrf^+)^\perp = \Hdf$.
\item[(ii)] $\Hdrf^\pm \cap \Hdf = \{0\}$.
\item[(iii)] If $u_-$ and $u_+$ are divergence-free solutions to \eqref{eq_lame_bvp} and \eqref{eq_lame_bvp_ext}, respectively, then there are $f_-$ and $f_+$ in $\Hsb^3$ such that $\Dcal^\divf[f_-]= u_-$ and $\Dcal^\divf[f_+]= u_+$.
\end{itemize}

\section*{Acknowledgment}
The authors thank the unanimous referee(s) for valuable comments on this paper.

\appendix

\section{Smoothing properties of integral operators}\label{sect_smoothing}

In this appendix we prove regularity properties of certain integral operators which have been used at several places in this paper. These facts may be known. Since we fail to pinpoint a reference, we include proofs for readers' sake.

Let $\Omega\subset \Rbb^m$ ($m\geq 2$) be a bounded Lipschitz domain. The Sobolev norm on $H^s(\partial\Omega)$ ($0\leq s<1$) is equivalent to the square root of the quantity
\[
    \|\varphi\|_{L^2(\partial\Omega)}^2+\int_{\partial\Omega}\int_{\partial\Omega} \frac{|\varphi (x)-\varphi(y)|^2}{|x-y|^{m-1+2s}}\, \df \sigma (x)\df \sigma (y)
\]
(see \cite{Gilbarg-Trudinger01}). We observe that the norm in \eqref{eq_Sobolev_Besov} is the special case when $m=3$ and $s=1/2$.

\begin{theo}\label{theo_NP_bounded_Sobolev_g}
    Let $m\geq 2$ and $\Omega\subset \Rbb^m$ be a bounded Lipschitz domain and $T(x, y)\in L^1(\partial\Omega \times \partial\Omega)$ satisfy
    \begin{align}
        |T(x, y)|&\lesssim \frac{1}{|x-y|^{m-\beta-\gamma}} \quad (x\neq y), \label{eq_singularity_0}\\
        |T(x, y)-T(z, y)|&\lesssim \frac{|x-z|^\gamma}{|x-y|^{m-\beta}} \quad (2|x-z|<|x-y|)\label{eq_singularity_1}
    \end{align}
    for some $\beta, \gamma> 0$ ($\gamma \le 1$) with $\beta+\gamma>1$. The integral operator
    \[
        \Tcal [\varphi](x):=\int_{\partial\Omega} T(x, y)\varphi(y)\, \df \sigma (y)
    \]
    has the following mapping properties:
    \begin{enumerate}[label=(\roman*)]
        \item \label{enum_T_Sobolev} $\Tcal$ is a bounded linear operator from $L^2 (\partial\Omega)$ to $H^s (\partial\Omega)$ for all $s$ satisfying
        \begin{equation}\label{eq_Sobolev_condition}
          0\leq s<\min \{ \beta+\gamma-1, \gamma\}.
        \end{equation}
        \item \label{enum_T_T1} Additionally, if $\Tcal [1]$ is constant on $\partial\Omega$, then $\Tcal$ is a bounded linear operator from $H^t (\partial\Omega)$ to $H^s (\partial\Omega)$ for all $(t, s)$ satisfying
        \begin{equation}\label{eq_Sobolev_T1_condition}
            \begin{cases}
                0\leq s<\gamma, \\
                \max\{ 0, s-\beta-\gamma+1\}\leq t, \\
                (t, s)\neq (0, \beta+\gamma-1). 
            \end{cases}
        \end{equation}
    \end{enumerate}
\end{theo}

\begin{proof}
    Since $T(x, y)\in L^1(\partial\Omega \times \partial\Omega)$, $\Tcal$ is bounded on $L^2(\partial\Omega)$ (Schur's test). So we only have to estimate the (square of) seminorm
    \[
        \int_{\partial\Omega\times \partial\Omega} \frac{|\Tcal[\varphi] (x)-\Tcal[\varphi] (z)|^2}{|x-z|^{m+2s-1}}\, \df \sigma (x)\df \sigma (z)
    \]
    for $\varphi \in H^t (\partial\Omega)$, $s>0$ and $t\geq 0$ satisfying either \eqref{eq_Sobolev_condition} or \eqref{eq_Sobolev_T1_condition}.

    For $x, z\in \partial\Omega$, we set
    \[
        A_{x, z}:=\{ y\in \partial\Omega \mid 2|x-z|<|x-y|\}.
    \]
    Then, we have
    \begin{align*}
        \Tcal[\varphi](x)-\Tcal[\varphi](z)
        &=\int_{\partial\Omega}(T(x, y)-T(z, y))(\varphi (y)-\varphi (x))\, \df \sigma (y) \\
        &\quad +\varphi (x)\int_{\partial\Omega} (T(x, y)-T(z, y))\, \df \sigma (y) \\
        &= \int_{A_{x, z}} + \int_{A_{x, z}^c} (T(x, y)-T(z, y))(\varphi (y)-\varphi (x))\, \df \sigma (y) \\
        &\quad +\varphi (x) \left( \Tcal[1](x)-\Tcal[1](z) \right) \\
        &=: I_1+I_2+\varphi (x)I_3.
    \end{align*}
Since
    \[
        |\Tcal[\varphi](x)-\Tcal[\varphi](z)|^2 \lesssim |I_1|^2+|I_2|^2+|\varphi(x)|^2|I_3|^2,
    \]
    we have
    \beq\label{10006}
        \int_{\partial\Omega\times \partial\Omega}\frac{|\Tcal[\varphi](x)-\Tcal[\varphi](z)|^2}{|x-z|^{2s+m-1}}\, \df \sigma (x)\df \sigma (z)
        \lesssim J_1+J_{2}+J_3
    \eeq
    where
    \begin{align*}
        J_1&:=\int_{\partial\Omega\times \partial\Omega} \frac{\df \sigma (x)\df \sigma (z)}{|x-z|^{2s+m-2\gamma-2\tau}}\int_{A_{x, z}}\frac{|\varphi (y)-\varphi (x)|^2}{|x-y|^{m-2(\beta-\tau)}}\, \df \sigma (y), \\
        J_{2}&:=\int_{\partial\Omega\times \partial\Omega}\frac{\df \sigma (x)\df \sigma (z)}{|x-z|^{2s-2\theta+m}}\int_{3|x-z|\geq |z-y|}\frac{|\varphi (y)-\varphi (x)|^2\, \df \sigma (y)}{|z-y|^{m-2(\beta+\gamma-\theta)}} \\
        J_3 &:= \int_{\partial\Omega\times \partial\Omega} \frac{|\varphi(x)|^2|I_3|^2}{|x-z|^{m+2s-1}}\, \df \sigma (x)\df \sigma (z).
    \end{align*}

Since $s<\gamma$, we can take $\tau\in (1/2-\gamma+s, 1/2)$. Then, we use \eqref{eq_singularity_1} to have
    \begin{align*}
        |I_1|^2&\lesssim \left(\int_{A_{x, z}}|T(x, y)-T(z, y)| |\varphi (y)-\varphi (x)|\, \df \sigma (y)\right)^2 \\
        &\lesssim \left(\int_{A_{x, z}}\frac{|x-z|^\gamma}{|x-y|^{m-\beta}}|\varphi (y)-\varphi (x)|\, \df \sigma (y)\right)^2 \\
        &\leq |x-z|^{2\gamma}\left(\int_{A_{x, z}}\frac{\df \sigma (y)}{|x-y|^{m-2\tau}}\right)\left(\int_{A_{x, z}}\frac{|\varphi (y)-\varphi(x)|^2}{|x-y|^{m-2(\beta-\tau)}}\, \df \sigma (y)\right) \\
        &\lesssim |x-z|^{2\gamma+2\tau-1} \int_{A_{x, z}}\frac{|\varphi (y)-\varphi(x)|^2}{|x-y|^{m-2(\beta-\tau)}}\, \df \sigma (y),
    \end{align*}
where the last inequality holds since $0<\tau<1/2$.

Note that
\[
|I_2| \le I_{2,1}+I_{2,2}+I_{2, 3}
\]
where
\begin{align}
I_{2,1}&= \int_{A_{x, z}^c} |T(x, y)| |\varphi (y)-\varphi (x)|\, \df \sigma (y), \\
I_{2, 2}&=\int_{A_{x, z}^c} |T(z, y)||\varphi (y)-\varphi (z)|\, \df \sigma (y), \\
I_{2, 3}&=|\varphi (x)-\varphi (z)|\int_{A_{x, z}^c} |T(z, y)|\, \df \sigma (y). 
\end{align}
Take $\theta\in (1/2, 1/2+s)$. We use \eqref{eq_singularity_0} to have
    \begin{align*}
        I_{2,1}^2 & \lesssim \left( \int_{A_{x, z}^c} \frac{|\varphi (y)-\varphi (x)|}{|x-y|^{m-\beta-\gamma}} \, \df \sigma (y) \right)^2 \\
        &\lesssim \left(\int_{2|x-z|\geq |x-y|}\frac{\df \sigma (y)}{|x-y|^{m-2\theta}}\right)\left(\int_{2|x-z|\geq |x-y|}\frac{|\varphi (y)-\varphi (x)|^2\, \df \sigma (y)}{|x-y|^{m-2(\beta+\gamma-\theta)}}\right) \\
        &\lesssim |x-z|^{2\theta-1}\int_{2|x-z|\geq |x-y|}\frac{|\varphi (y)-\varphi (x)|^2\, \df \sigma (y)}{|x-y|^{m-2(\beta+\gamma-\theta)}} .
    \end{align*}
    If $y \in A_{x, z}^c$, then $2|x-z|\geq |x-y|$, and hence $|z-y|\leq |x-z|+|x-y|\leq 3|x-z|$. So, the same argument yields
    \begin{align*}
        I_{2,2}^2
        &\lesssim |x-z|^{2\theta-1}\int_{3|x-z|\geq |z-y|}\frac{|\varphi (y)-\varphi (z)|^2\, \df \sigma (y)}{|z-y|^{m-2(\beta+\gamma-\theta)}}.
    \end{align*}
    Concerning $I_{2, 3}$, since 
    \[
        \int_{A_{x, z}^c} |T(z, y)|\, \df \sigma (y)
        \lesssim \int_{3|x-z|\geq |z-y|} \frac{\df \sigma (y)}{|z-y|^{d-\beta-\gamma}}
        \lesssim |x-z|^{\beta+\gamma-1}
    \]
    by $\beta+\gamma-1>0$, we have 
    \[
        |I_{2, 3}|\lesssim |x-z|^{\beta+\gamma-1}|\varphi (x)-\varphi(z)|. 
    \]
    Thus we have
    \begin{align*}
        |I_2|^2
        &\lesssim |x-z|^{2\theta-1}\int_{2|x-z|\geq |x-y|}\frac{|\varphi (y)-\varphi (x)|^2\, \df \sigma (y)}{|x-y|^{m-2(\beta+\gamma-\theta)}} \\
        &\quad +|x-z|^{2\theta-1}\int_{3|x-z|\geq |z-y|}\frac{|\varphi (y)-\varphi (z)|^2\, \df \sigma (y)}{|z-y|^{m-2(\beta+\gamma-\theta)}} \\
        &\quad + |x-z|^{2(\beta+\gamma-1)}|\varphi(x)-\varphi (z)|^2.
    \end{align*}

    Since $s<\gamma+\tau-1/2$, we have
    \[
    \int_{2|x-z|<|x-y|}\frac{\df \sigma (z)}{|x-z|^{2s-2\gamma+m-2\tau}} \lesssim \frac{1}{|x-y|^{2s-2\gamma-2\tau+1}}.
    \]
    It thus follows that
    \begin{align*}
        J_1&=\int_{\partial\Omega\times \partial\Omega}\df \sigma (x)\df \sigma (y)\, \frac{|\varphi (y)-\varphi (x)|^2}{|x-y|^{m-2(\beta-\tau)}}\int_{2|x-z|<|x-y|}\frac{\df \sigma (z)}{|x-z|^{2s-2\gamma+m-2\tau}} \\
        &\lesssim \int_{\partial\Omega\times \partial\Omega} \frac{|\varphi (y)-\varphi (x)|^2}{|x-y|^{m-2\beta+2s-2\gamma+1}}\, \df \sigma (x)\df \sigma (y) .
    \end{align*}
    We see from this inequality that
    \begin{equation}\label{10001}
    J_1 \lesssim
    \begin{cases}
    \| \varphi \|_{H^t (\partial\Omega)}^2 \quad &\text{if } t \ge s-\beta-\gamma+1>0, \\
    \| \varphi \|_{L^2 (\partial\Omega)}^2 \quad &\text{if } s-\beta-\gamma+1 < 0.
    \end{cases}
\end{equation}

    Since $s>\theta-1/2$, we have
    \[
        \int_{2|x-z|\geq |x-y|}\frac{\df \sigma (z)}{|x-z|^{2s-2\theta+m}} \lesssim \frac{1}{|x-y|^{2s-2\theta+1}}
    \]
    and 
    \[
        \int_{3|x-z|\geq |z-y|}\frac{\df \sigma (x)}{|x-z|^{2s-2\theta+m}} \lesssim \frac{1}{|z-y|^{2s-2\theta+1}},
    \]
    and hence
    \begin{align*}
        J_{2}&=\int_{\partial\Omega\times \partial\Omega}\df \sigma (x)\df \sigma (y)\, \frac{|\varphi (y)-\varphi (x)|^2}{|x-y|^{m-2(\beta+\gamma-\theta)}}\int_{2|x-z|\geq |x-y|}\frac{\df \sigma (z)}{|x-z|^{2s-2\theta+m}} \\
        &\quad +\int_{\partial\Omega\times \partial\Omega}\df \sigma (z)\df \sigma (y)\, \frac{|\varphi (y)-\varphi (z)|^2}{|z-y|^{m-2(\beta+\gamma-\theta)}}\int_{3|x-z|\geq |z-y|}\frac{\df \sigma (x)}{|x-z|^{2s-2\theta+m}} \\
        &\quad +\int_{\partial\Omega\times \partial\Omega} \frac{|\varphi (x)-\varphi (z)|^2}{|x-z|^{m+2s-1-2(\beta+\gamma-1)}}\, \df \sigma (x)\df\sigma (z) \\
        &\lesssim \int_{\partial\Omega\times \partial\Omega}\frac{|\varphi (y)-\varphi (x)|^2}{|x-y|^{m-2(\beta+\gamma)+2s+1}}\, \df \sigma (x)\df \sigma (y).
    \end{align*}
    Thus we have
    \beq\label{10002}
    J_2 \lesssim
    \begin{cases}
    \| \varphi \|_{H^t (\partial\Omega)}^2 \quad &\text{if } t \ge s-\beta-\gamma+1> 0 , \\
    \| \varphi \|_{L^2 (\partial\Omega)}^2 \quad &\text{if } s-\beta-\gamma+1 < 0.
    \end{cases}
    \eeq

We now estimate $J_3$. We write
    \begin{align*}
        I_3
        &=\int_{A_{x, z}} +\int_{A_{x, z}^c} (T(x, y)-T(z, y))\, \df \sigma (y) =:I_{3, 1}+I_{3, 2}.
    \end{align*}
    We use \eqref{eq_singularity_1} to have
    \begin{align*}
        |I_{3, 1}|\lesssim \int_{A_{x, z}} \frac{|x-z|^\gamma}{|x-y|^{m-\beta}}\, \df \sigma (y)
        \lesssim |x-z|^\gamma h_\beta(|x-z|)
    \end{align*}
    where
    \[
        h_\beta (\xi):=
        \begin{cases}
            \xi^{\beta-1} & \text{if } 0\leq \beta<1, \\
            1+|\log \xi| & \text{if } \beta=1, \\
            1 & \text{if } \beta>1.
        \end{cases}
    \]
    for $\xi>0$. We then use \eqref{eq_singularity_0} to have
    \begin{align*}
        |I_{3, 2}|&\lesssim \int_{A_{x, z}^c} \left(\frac{1}{|x-y|^{m-\beta-\gamma}}+\frac{1}{|z-y|^{m-\beta-\gamma}}\right)\, \df \sigma (y) \\
        &\leq \int_{A_{x, z}^c} \frac{\df \sigma (y)}{|x-y|^{m-\beta-\gamma}}
        +\int_{3|x-z|\geq |z-y|} \frac{\df \sigma (y)}{|z-y|^{m-\beta-\gamma}} \\
        &\lesssim |x-z|^{\beta+\gamma-1}.
    \end{align*}
Thus we have
    \begin{align}
        J_3
        &\lesssim \int_{\partial\Omega \times \partial\Omega} \frac{h_\beta(|x-z|)^2|\varphi(x)|^2}{|x-z|^{2s-2\gamma+m-1}}\,\df \sigma (x)\df \sigma (z) \nonumber \\
        &\quad +\int_{\partial\Omega \times \partial \Omega} \frac{|\varphi (x)|^2}{|x-z|^{2s+m-2\gamma-2\beta+1}}\,\df \sigma (x)\df \sigma (z). \label{10004}
    \end{align}

If \eqref{eq_Sobolev_condition} holds, namely, if $s<\min \{ \beta+\gamma-1, \gamma\}$, then we see from \eqref{10004} that
    \begin{align*}
        J_3\lesssim \|\varphi\|_{L^2 (\partial\Omega)}^2.
    \end{align*}
We also see from \eqref{10001} and \eqref{10002} that
    \begin{align*}
        J_1+J_2 \lesssim \|\varphi\|_{L^2 (\partial\Omega)}^2.
    \end{align*}
It then follows from \eqref{10006} that
\[
\int_{\partial\Omega\times \partial\Omega}\frac{|\Tcal[\varphi](x)-\Tcal[\varphi](z)|^2}{|x-z|^{2s+m-1}}\, \df \sigma (x)\df \sigma (z)
        \lesssim \|\varphi\|_{L^2 (\partial\Omega)}^2,
        \]
and hence $\Tcal: L^2 (\partial\Omega) \to H^s (\partial\Omega)$ is bounded if $s<\min \{ \beta+\gamma-1, \gamma\}$.
This proves \ref{enum_T_Sobolev}.

If $\Tcal[1]$ is constant on $\partial\Omega$, then $I_3=0$. It thus follows from \eqref{10006}, \eqref{10001} and \eqref{10002} that
\begin{align*}
&\int_{\partial\Omega\times \partial\Omega}\frac{|\Tcal[\varphi](x)-\Tcal[\varphi](z)|^2}{|x-z|^{2s+m-1}}\, \df \sigma (x)\df \sigma (z) \\
        &\lesssim 
        \begin{cases}
            \|\varphi\|_{H^t (\partial\Omega)}^2 & \text{if } t\geq s-\beta-\gamma+1>0, \\
            \|\varphi\|_{L^2 (\partial\Omega)}^2 & \text{if } s-\beta-\gamma+1<0. 
        \end{cases}
    \end{align*}
    This completes the proof. 
\end{proof}

\begin{coro}\label{coro_T_cpt}
    Under the assumptions of Theorem \ref{theo_NP_bounded_Sobolev_g}, the following hold.
    \begin{enumerate}[label=(\roman*)]
        \item $\Tcal$ is compact on $H^s(\partial\Omega)$ for all $0\leq s<\min \{\beta+\gamma-1, \gamma \}$.
        \item \label{enum_T_T1_cpt} If $\Tcal[1]$ is constant on $\partial\Omega$, then $\Tcal$ is compact on $H^s (\partial\Omega)$ for all $s \in [0,\gamma)$.
    \end{enumerate}
\end{coro}

\begin{proof}
    Take $0<s<\min \{\beta+\gamma-1, \gamma\}$. Since $\Tcal: L^2 (\partial\Omega)\to H^s(\partial\Omega)$ is bounded by Theorem \ref{theo_NP_bounded_Sobolev_g} \ref{enum_T_Sobolev} and the embedding $H^s (\partial\Omega)\hookrightarrow L^2 (\partial\Omega)$ is compact by the Rellich-Kondrachov theorem, $\Tcal$ is compact on $H^s (\partial\Omega)$. \ref{enum_T_T1_cpt} is proved in the same way.
\end{proof}

We also obtain the following corollary.

\begin{coro}
    \label{coro_regularity}
    Under the assumptions of Theorem \ref{theo_NP_bounded_Sobolev_g}, the following hold.
    \begin{enumerate}[label=(\roman*)]
        \item \label{enum_regularity}If $\varphi\in L^2 (\partial\Omega)$ and $(\Tcal-\lambda I)[\varphi]\in H^s (\partial\Omega)$ for some $\lambda\neq 0$ and $0 < s<\min \{ \beta+\gamma-1, \gamma\}$, then $\varphi\in H^s (\partial\Omega)$.
        \item \label{enum_regularity_T1} Suppose that $\Tcal[1]$ is constant on $\partial\Omega$. If $\varphi\in L^2 (\partial\Omega)$ and $(\Tcal-\lambda I)[\varphi]\in H^s (\partial\Omega)$ for some $\lambda\neq 0$ and $0< s< \gamma$, then $\varphi\in H^s (\partial\Omega)$.
    \end{enumerate}
Furthermore, the following hold for the $L^2$-adjoint $\Tcal^*$ of $\Tcal$.
    \begin{enumerate}[label=(\roman*)]
        \setcounter{enumi}{2}
        \item \label{enum_regularity_adj} If $\varphi \in H^{-s}(\partial\Omega)$ and $(\Tcal^*-\lambda I)[\varphi]\in L^2(\partial\Omega)$ for some $\lambda\neq 0$ and $0< s<\min\{\beta+\gamma-1, \gamma\}$, then $\varphi\in L^2 (\partial\Omega)$.
        \item \label{enum_regularity_T1_adj} Suppose that $\Tcal[1]$ is constant on $\partial\Omega$. If $\varphi\in H^{-s}(\partial\Omega)$ and $(\Tcal^*-\lambda I)[\varphi]\in L^2(\partial\Omega)$ for some $\lambda\neq 0$ and $0< s< \gamma$, then $\varphi\in L^2 (\partial\Omega)$.
    \end{enumerate}
\end{coro}

\begin{proof}
    \ref{enum_regularity} follows from Theorem \ref{theo_NP_bounded_Sobolev_g} \ref{enum_T_Sobolev} and the identity $\varphi=\lambda^{-1}(\Tcal [\varphi]-(\Tcal-\lambda I)[\varphi])$. For the proof of \ref{enum_regularity_T1}, take $0<\varepsilon<\beta+\gamma-1$. Then Theorem \ref{theo_NP_bounded_Sobolev_g} \ref{enum_T_T1} implies that the operator $\Tcal: H^t (\partial\Omega)\to H^{t+\varepsilon}(\partial\Omega)$ is bounded if $0\leq t<t+\varepsilon<\gamma$. In particular, $\Tcal: L^2 (\partial\Omega)\to H^{s_1} (\partial\Omega)$, where $s_1:=\min \{s, \varepsilon\}$, is bounded since $s<\gamma$. Then, since $(\Tcal-\lambda I)[\varphi]\in H^s (\partial\Omega)$, we have $\varphi\in H^{s_1}(\partial\Omega)$. Since $\Tcal: H^{s_1}(\partial\Omega)\to H^{s_2}(\partial\Omega)$, where $s_2:=\min \{s, s_1+\varepsilon\}=\min \{s, 2\varepsilon\}$ is bounded, we have $\varphi \in H^{s_2}(\partial\Omega)$. We iterate this procedure to obtain $\varphi\in H^{s_k}(\partial\Omega)$, where $s_k:=\min \{ s, k\varepsilon\}$ for all integer $k\in \Nbb$. Thus we obtain $\varphi\in H^s (\partial\Omega)$. 
    
    \ref{enum_regularity_adj} and \ref{enum_regularity_T1_adj} hold by duality.
\end{proof}

\section{Dilation of domain and single layer potential in elasticity}\label{sect_dilation_elastic}

For a bounded Lipschitz domain $\Omega\subset \Rbb^2$ and a positive number $r>0$, we define the dilation of $\Omega$ as
\[
    \Omega_r:=\left\{ \frac{1}{r}x \,\middle|\, x\in \Omega\right\}.
\]
For $f: \partial\Omega \to \Rbb^2$, we define $f_r(x)$ for $x \in \partial\Omega_r$ by
\[
    f_r(x):=rf(rx).
\]
We denote $(\Hsb_r^2)^*:=H^{-1/2}(\partial\Omega_r)^2$, $\Hsb_r^2:=H^{1/2}(\partial\Omega_r)^2$ and the single layer potential on $\partial\Omega_r$ by $\Scal_r$ (without confusion with the single layer potential for the Laplacian $\Scal_0$ which appeared before). Then we obtain the following theorem.

\begin{theo}\label{theo_dilation_sl}
    For a bounded Lipschitz domain $\Omega\subset \Rbb^2$, there exists (large) $r>0$ such that $\Scal_r: (\Hsb_r^2)^*\to \Hsb_r^2$ is invertible and the bilinear form in \eqref{eq_inner_product} defines an inner product on $\Hsb_r^2$. The induced norm is equivalent to the Sobolev norm on $\Hsb_r^2$.
\end{theo}

On a two-dimensional domain $\GO$, the single layer potential $\Scal$ may have a non-trivial kernel and it prohibits the bilinear form in \eqref{eq_inner_product} from being an inner product. We show that if we dilate $\GO$ by sufficiently large $r$, then $\Scal_r$ has a trivial kernel. In fact, we prove that
\beq\label{B100}
    -\jbracket{f, \Scal_r [f]}_{\partial\Omega_r} >0.
\eeq
for all $f \in (\Hsb_r^2)^*$

\begin{proof}[Proof of Theorem \ref{theo_dilation_sl}]
    To prove \eqref{B100}, we assume $\GO$ is simply connected without loss of generality.

    Let us recall some facts whose proofs can be found in \cite{AJKKY}.
    \begin{enumerate}[label=(\roman*)]
        \item \label{enum_2d_sl_zero_mean} $-\jbracket{f, \Scal [f]}_{\partial\Omega}>0$ for all $f \in \Vcal:=\left\{ f\in (\Hsb^2)^* \,\middle|\, \int_{\partial\Omega} f\, \df \sigma=0\right\}$.
        \item \label{enum_2d_sl_exceptional} Let $\Wcal^*:=\Ker (\Kcal^*-1/2I)$ (on $(\Hsb^2)^*$) and $\Wcal:=\Ker (\Kcal-1/2I)$ (on $\Hsb^2$). Then $\Wcal^*$ and $\Wcal$ are three-dimensional, $\Scal$ maps $\Wcal^*$ into $\Wcal$, $\Ker \Scal \subset \Wcal^*$, and $\Wcal$ is spanned by $(1, 0)^T, (0, 1)^T, (-y, x)^T$. In particular, the constant vectors constitute a two-dimensional subspace of $\Wcal$.
    \end{enumerate}

We first show that there are linearly independent $f_1, f_2\in (\Hsb^2)^*$ such that $\Scal [f_1]$ and $\Scal[f_2]$ are constant on $\partial\Omega$. To prove this, we consider the restriction $\Scal|_{\Wcal^*}: \Wcal^*\to \Wcal$. If $\dim \Ker \Scal|_{\Wcal^*}=0$, then $\Scal|_{\Wcal^*}$ in invertible and we set $f_1:=(\Scal|_{\Wcal^*})^{-1}[(1, 0)^T]$ and $f_2:=(\Scal|_{\Wcal^*})^{-1}[(0, 1)^T]$. If $\dim \Ker \Scal|_{\Wcal^*}=1$, then we take $f_1\in \Ker \Scal|_{\Wcal^*}$ and $f_1 \neq 0$. Since $\Ran \Scal|_\Wcal$ must contain at least one non-zero constant vector, say, $b\in \Rbb^2$, we can take $f_2\in \Wcal$ such that $\Scal[f_2]=b$. Clearly $f_1$ and $f_2$ are linearly independent. If $\dim \Ker \Scal|_\Wcal\geq 2$, then we just take linearly independent $f_1, f_2\in \Ker \Scal$.

We then show that $\int_{\partial\Omega}f_1\, \df \sigma$ and $\int_{\partial\Omega}f_2\, \df \sigma$ are linearly independent in $\Rbb^2$. In fact, if $a_1\int_{\partial\Omega}f_1\, \df \sigma+a_2\int_{\partial\Omega}f_2\, \df \sigma=0$ for some $a_1, a_2\in \Rbb$, then $f:=a_1f_1+a_2f_2$ satisfies $\int_{\partial\Omega}f\, \df \sigma=0$, and
\[
        -\jbracket{\Scal [f], f}_{\partial\Omega}
        =-\Scal [f] \cdot \int_{\partial\Omega} f\, \df \sigma=0.
    \]
By \ref{enum_2d_sl_zero_mean} above, $f=0$. Since $f_1$ and $f_2$ are linearly independent, we obtain $a_1=a_2=0$.

In particular, by considering appropriate linear combinations of $f_1$ and $f_2$, we may assume
\[
\int_{\partial\Omega}f_1\, \df \sigma=(1, 0)^T, \quad \int_{\partial\Omega} f_2\, \df \sigma =(0, 1)^T.
\]

A direct calculation shows that for $f\in (\Hsb^2)^*$,
    \[
        -\jbracket{f_r, \Scal_r [f_r]}_{\partial\Omega_r}
        =\frac{\alpha_1}{2\pi}\ln r\left|\int_{\partial\Omega} f\, \df \sigma\right|^2-\jbracket{f,\Scal [f]}_{\partial\Omega}
    \]
    where $\Ga_1$ is the constant in \eqref{eq_alpha_12}.
    We decompose $f\in (\Hsb^2)^*$ as $f=f_0+\widetilde{f}$ where
    \[
        \widetilde{f}:=a_1f_1+a_2f_2 \quad \text{with } (a_1, a_2)^T=\int_{\partial\Omega} f\, \df\sigma
    \]
    and $f_0:=f-\widetilde{f}\in \Vcal$. Then $\Scal[\widetilde{f}]$ is constant on $\p\GO$, and hence
    \begin{align*}
        &\frac{\alpha_1}{2\pi}\ln r\left|\int_{\partial\Omega} f\, \df \sigma\right|^2-\jbracket{f,\Scal [f]}_{\partial\Omega} \\
        &=\frac{\alpha_1}{2\pi}\ln r\left|\int_{\partial\Omega} \widetilde{f}\, \df \sigma\right|^2-\jbracket{f_0, \Scal [f_0]}_{\partial\Omega}-2\jbracket{f_0,\Scal [\widetilde{f}]}_{\partial\Omega}-\jbracket{\widetilde{f}, \Scal [\widetilde{f}]}_{\partial\Omega} \\
        &=\frac{\alpha_1}{2\pi}\ln r\left|\int_{\partial\Omega} \widetilde{f}\, \df \sigma\right|^2-\jbracket{f_0, \Scal [f_0]}_{\partial\Omega} -\jbracket{\widetilde{f}, \Scal [\widetilde{f}]}_{\partial\Omega} .
    \end{align*}
Note that
    \begin{align*}
        &\frac{\alpha_1}{2\pi}\ln r\left|\int_{\partial\Omega} \widetilde{f}\, \df \sigma\right|^2-\jbracket{\widetilde{f}, \Scal [\widetilde{f}]}_{\partial\Omega} \\
        &=\frac{\alpha_1\ln r}{2\pi}(a_1^2+a_2^2)-\sum_{i, j=1}^2 a_ia_j\jbracket{f_i, \Scal [f_j]}_{\partial\Omega}.
    \end{align*}
Thus we have
\[
-\jbracket{f_r, \Scal_r [f_r]}_{\partial\Omega_r} =   \frac{\alpha_1\ln r}{2\pi}(a_1^2+a_2^2)-\sum_{i, j=1}^2 a_ia_j\jbracket{f_i, \Scal [f_j]}_{\partial\Omega}  -\jbracket{f_0, \Scal [f_0]}_{\partial\Omega}.
\]
One can easily see from this identity that $-\jbracket{f_r, \Scal_r [f_r]}_{\partial\Omega_r}>0$ for any $f \neq 0$ if $r>0$ is sufficiently large. Hence $\Scal_r: (\Hsb_r^2)^*\to \Hsb_r^2$ is invertible and the bilinear form \eqref{eq_inner_product} on $\Hsb_r^2$ is positive-definite. It can be proved easily that the inner product \eqref{eq_inner_product} induces a norm which is equivalent to the Sobolev norm on $\Hsb_r^2$.
\end{proof}



\end{document}